\documentclass[12pt]{amsart}
\usepackage{geometry}
\usepackage[ansinew]{inputenc}
\usepackage{amssymb}
\usepackage{amsmath}
\usepackage{amsthm}
\usepackage{bm}
\usepackage{bbm}
\usepackage{hyperref}
\usepackage{mathrsfs}
\usepackage{enumerate}
\usepackage{titletoc}

\usepackage{comment}

\usepackage[numbers, sort&compress]{natbib}
\input{math.sty}

\begin{document}

\title[BvM for Schr\"odinger Equation]{Bernstein - von Mises theorems for statistical inverse problems I:  Schr\"odinger equation}

\author{Richard Nickl \\ \\  University of Cambridge \\ \\  \today}

\maketitle

\begin{abstract}
The inverse problem of determining the potential $f>0$ in the partial differential equation $$\frac{\Delta}{2} u - fu =0 \text{ on } \mathcal O ~~\text{s.t. } u = g \text { on } \partial \mathcal O,$$ where $\mathcal O$ is a bounded $C^\infty$-domain in $\mathbb R^d$ and $g>0$ is a given function prescribing boundary values, is considered. The data consist of the solution $u$ corrupted by additive Gaussian noise. A nonparametric Bayesian prior for the function $f$ is devised and a Bernstein - von Mises theorem is proved which entails that the posterior distribution given the observations is approximated in a suitable function space by an infinite-dimensional Gaussian measure that has a `minimal' covariance structure in an information-theoretic sense. As a consequence the posterior distribution performs valid and optimal frequentist statistical inference on various aspects of $f$ in the small noise limit.

\medskip


\smallskip

\end{abstract}

 \tableofcontents

\section{Introduction}

Inverse problems form a vast and well-studied area within applied mathematics. In the `information age' we live in, algorithms that successfully solve these problems must be robust to the presence of statistical noise and  measurement error. A principled approach to such \textit{statistical} inverse problems is the Bayesian one, and it has been shown in influential work in the last decade that modern MCMC methodology can be used to construct computationally efficient Bayesian algorithms for complicated non-linear inverse problems in infinite-dimensional settings. This methodology is attractive for scientists because the Bayesian posterior distribution automatically delivers an estimate of the statistical uncertainty in the reconstruction, and hence suggests `confidence' intervals that allow to reject or accept scientific hypotheses.  The literature on applications of Bayes procedures in inverse problems is growing rapidly and cannot be reviewed here, we only mention Andrew Stuart's survey papers  \cite{S10, DS16} and the contributions \cite{CDRS09, LSS09, P12, DHS12, SS12, CRSW13, DLSV13, HB15}, where many further references can be found.

Algorithms that solve ill-posed inverse problems typically involve a regularisation step, for instance via a penalised variational procedure or a spectral cut-off. In the Bayesian approach this step is provided by the prior distribution, which represents a regularisation tool rather than subjective prior beliefs about the state of nature.  A natural question therefore arises as to whether such Bayesian algorithms deliver adequate inferential conclusions that are \textit{independent} of the prior. If so, one may further ask whether Bayes solutions of inverse problems allow for recovery of an unknown parameter $f$ in a statistically \textit{optimal} way. Of particular importance in this context is to understand whether `credible regions' constructed from the posterior distribution are objectively valid, approximate `frequentist' confidence sets. A paradigm to answer these questions is provided by performing an analysis of the Bayesian algorithm in the `large sample' or `small noise' limit, and under the assumption that the data are generated from a fixed `true' function $f=f_0$ (instead of $f$ being drawn at random from the prior distribution). That Bayesian inference can be valid in this setting has been well studied in mathematical statistics since Laplace (Chapter VI in \cite{L1812}). We refer to the recent monograph \cite{GvdV17} and also to Chapter 7.3 in \cite{GN16} for an account of such `frequentist consistency results' for nonparametric Bayes procedures in standard statistical models, and to the recent paper \cite{SVV15} and its discussion for results concerning the particular question of nonparametric Bayesian credible regions and their frequentist properties.

For \textit{linear} inverse problems the theory from Bayesian non-parametric statistics often carries over to the inverse setting, using the singular value decomposition (SVD) of the forward operator and/or conjugacy of Gaussian priors. See the papers \cite{K11}, \cite{R13} and \cite{S13} and also the more recent references \cite{SVV15, KLS16, KS18, MNP17}. From these results one can deduce objective guarantees for the Bayesian approach, and in fact Bayes point estimates (such as MAP statistics) can be shown to be closely linked to commonly used Tikhonov regularisation or penalisation methods in these inverse problems (see \cite{DLSV13}, \cite{HB15}).

In the case of \textit{non-linear} statistical inverse problems, which include many important examples arising with partial differential equations (PDEs), little is known about the frequentist performance of Bayesian methods. For the problem of inferring the diffusion coefficient from noisy observations of the solution of an elliptic PDE, \cite{V13} provides some contraction rate results that entail that the posterior distribution arising from certain priors concentrates on a neighborhood of the true parameter $f_0$ (although the conditions and rates obtained there are far from optimal). Recently it was shown in \cite{NS17} in a related parabolic problem that Bayesian inference for the coefficients of a scalar elliptic differential operator based on discrete samples of the associated diffusion Markov process can result in optimal posterior contraction rates about the true parameter. A form of weak consistency in this model had earlier been proved by \cite{MvZ13}, and a related recent contribution is \cite{A18}. While consistency and contraction rates are relevant results, they do not per se justify Bayesian `credible sets' and related uncertainty quantification methodology. Such guarantees may be derived from  more precise \textit{Bernstein-von Mises theorems}, which establish that the posterior distribution is approximated by a canonical Gaussian distribution in the small noise or large sample limit. While well understood in the finite-dimensional case \cite{LC86, vdV98}, the Bernstein-von Mises phenomenon is more subtle in the setting of high- and infinite dimensional statistical models (see \cite{F99}), and for non-linear inverse problems such results are currently not available: exceptions being \cite{L17} where normal approximations of posteriors in Euclidean spaces of increasing dimension are obtained, but under conditions whose suitability for inverse problems is unclear; and also the sequel \cite{NS17a} to the present paper.

This article constitutes an attempt to advance our understanding of the statistical performance of Bayes algorithms for non-linear inverse problems. In view of the importance of inverse problems arising in PDE models (see the various examples in the survey papers \cite{S10, DS16}) we study here a basic situation where one wishes to recover a coefficient of an elliptic partial differential operator from an observation of a solution to the associated PDE under some boundary conditions, corrupted by additive Gaussian noise. In fact we shall lay out the theory in a simple `Laplacian plus potential' situation which comprises all the conceptual difficulties but permits a clean analytic exposition of the main ideas behind the results we obtain. We first derive minimax optimal (within $\log$-factors) rates of posterior contraction about the unknown potential term in $L^2$-distance (Theorems \ref{rates} and \ref{meta}). The main contribution of this article however is Theorem \ref{main} which  provides an infinite-dimensional Bernstein-von Mises theorem for the posterior distribution resulting from a carefully chosen series prior for the potential term. It is shown that the posterior measure is approximated in a sense to be made precise, by a certain `canonical' Gaussian measure whose covariance structure attains the statistical information lower bound for this inference problem. As a consequence the resulting posterior parameter inferences are statistically optimal from an objective, information theoretic and asymptotic minimax point of view. 

In proving our main result we follow the program put forward in the papers \cite{CN13, CN14} (see also Section 7.3.4 in \cite{GN16}) and understand the Bernstein-von Mises problem in infinite-dimensional statistical models as one of showing weak convergence of the (scaled and centred) posterior measure to a fixed Gaussian measure in a function space that is large enough to support the limit distribution as a tight probability measure. The topology of this function space turns out to be weaker than the standard $L^p$-norm that one might otherwise consider, but one can show, again following \cite{CN13}, that arguments from interpolation theory imply that Bayesian credible sets constructed for these weak norms are valid frequentist confidence sets whose diameter also converges to zero in the stronger $L^1$-distance. 

In summary, our results give some theoretical support for the assertion that the Bayes approach can \textit{in principle} be expected to provide efficient solutions of \textit{non-linear} statistical inverse problems, and further that the associated uncertainty quantification can be objectively valid in the large sample/small noise limit. By emphasising `in principle' we wish to point out though very clearly here that our results should \textit{not} be construed as giving general guarantees for Bayes solutions of \textit{arbitrary} inverse problems: Even though our proofs do give a template for obtaining similar theorems in other settings, the details depend strongly on the fact that the inverse problem associated with the Schr\"odinger equation is in a certain sense `globally stable' on the parameter space we consider, and that the prior is taken to be supported in that parameter space (see also Remark \ref{lesson}). Obtaining similar results for different inverse problems or priors requires a careful analysis of various properties of the forward operator and of its linearisation, and it is conceivable that in certain situations the `information geometry' induced by the forward map may be such that Bernstein-von Mises theorems do in fact \textit{not} hold true. A comprehensive understanding of these questions for general non-linear problems remains a formidable challenge for future research in this area.

\section{A statistical inverse problem for the Schr\"odinger equation}
Let $\mathcal O$ be bounded $C^\infty$-domain in $\mathbb R^d, d \ge 2,$ with boundary $\partial \mathcal O$. Let $\bar{\mathcal O}$ be the closure of $\mathcal O$ in $\mathbb R^d$ and let $C(\bar {\mathcal O})$ be the space of continuous functions on $\bar {\mathcal O}$. For $\Delta = \sum_{i=1}^d \partial^2/\partial x_i^2$ the standard Laplacian operator, we consider the (time-independent) Schr\"odinger equation
\begin{equation}\label{PDE}
\frac{\Delta}{2} u - f u =0 \text{ on } \mathcal O~s.t.~u=g~\text{on } \partial \mathcal O
\end{equation}
where $g: \partial \mathcal O \to \mathbb R$ is a given function prescribing boundary values, and $f \in C(\bar {\mathcal O})$ is a non-negative `potential'. For $f > 0$ and $g$ sufficiently regular, a unique solution $u_f   \in C(\bar {\mathcal O})$ to (\ref{PDE}) exists and has probabilistic representation in terms of the Feynman-Kac formula 
\begin{equation} \label{fkac}
u_f(x) =u_{f,g}(x)= E^x\left[g(X_{\tau_\mathcal O}) e^{-\int_0^{\tau_\mathcal O} f(X_s)ds} \right],~x \in \mathcal O,
\end{equation} 
where $(X_s: s \ge 0)$ is a $d$-dimensional Brownian motion started at $x \in \mathcal O$, with exit time $\tau_\mathcal O$ from $\mathcal O$ satisfying $$\sup_{x \in \mathcal O}E^x \tau_\mathcal O \le K(vol(\mathcal O), d)<\infty.$$ We refer to Section \ref{sop} below, particularly Proposition \ref{sbaest}, for details. 

\smallskip

In inverse problems terminology, $f \mapsto u_f$ is the \textit{forward map}, and the \textit{inverse problem} is to recover $f$ given $u_f$ (and $g$). If $\|f\|_\infty \equiv \sup_{x \in \mathcal O}|f(x)| \le D$ and $\inf_{x \in \partial \mathcal O}g(x) \ge g_{\min} >0$ then we can apply Jensen's inequality to (\ref{fkac}) to obtain
\begin{equation} \label{jensen}
u_f(x) \ge g_{\min} e^{-\|f\|_\infty E^x\tau_\mathcal O} \ge c>0,~~c=c(g_{\min}, D, vol(\mathcal O), d),
\end{equation}
so that \textit{given} $u=u_f$ we can solve for $f$ simply by taking 
\begin{equation} \label{inv}
f(x) = \frac{1}{2}\frac{\Delta u}{u}(x),~~x \in \mathcal O.
\end{equation}

\smallskip

The potential $f$ models an attenuation of the solution of the Dirichlet boundary value problem for the standard Laplace equation. In physical language, $f$ describes a local `cooling' of the equilibrium temperature distribution of the classical heat equation, when initial `boundary temperatures' are given by $g$. In the theory that follows  $\Delta$ could be replaced by a general, strongly elliptic second order partial differential operator with \textit{known} coefficients, at the expense of mostly notational changes. Recent applications of inverse problems of this kind can be found, e.g., in photo-acoustics \cite{BU10} or scattering problems \cite{BL05, HW15}.

\smallskip

The question we ask here is how to optimally solve this non-linear inverse problem when the observations are corrupted by statistical noise.  The measurement model we consider is
$$Y_i = u_f(x_i) + w_i, ~~i=1, \dots, n,~ n \in \mathbb N,$$ where the $w_i$ are independent standard normal $N(0,1)$ random variables, and the $x_i$ are `equally spaced' approximate lattice points in the domain $\mathcal O$. By standard arguments from asymptotic statistics (see \cite{BL96, R08} or also Section 1.2.3 in \cite{GN16}) this discrete measurement model is asymptotically (as $n \to \infty$) equivalent to observing the functional equation
\begin{equation}\label{model0}
Y = u_f + \varepsilon \mathbb W,~~\varepsilon =\frac{1}{\sqrt n},
\end{equation}
where $\mathbb W$ is a Gaussian white noise (see Section \ref{obs} for details) in the Hilbert space $$L^2(\mathcal O)=\left\{h:\mathcal O \to \mathbb R, ~ \|h\|^2_{L^2(\mathcal O)}\equiv \int_\mathcal O h^2(x)dx<\infty \right\},$$ and $\varepsilon>0$ is the noise level. We will develop the theory that follows in this equivalent `continuous' model as it allows for a clear exposition of the key ideas. 

The noise process $\mathbb W$ belongs to the Sobolev space $H^{\eta}(\mathcal O)$ only for $\eta<-d/2$ and is not point-wise defined -- direct recovery of $f$ from $Y$ by a simple equation such as (\ref{inv}) thus becomes impossible. So while the deterministic inverse problem has an essentially straightforward solution, the statistical one has not: Even though (\ref{model0}) is a standard non-linear Gaussian regression model, the parameter space of admissible regression functions $u_f$ carries non-trivial non-linear constraints (since the $u_f$ have to be solutions of the PDE (\ref{PDE})), and to deal with these constraints in a statistically efficient way is a non-obvious task.

\medskip

A principled approach to solve such problems is the Bayesian one, which devises a prior distribution $\Pi$ for the `unknown' function $f$. More precisely, for $\mathcal F \subset L^2(\mathcal O)$ some parameter space, we consider $f$ distributed according to $\Pi$ where $\Pi$ is some (`prior') probability measure on the trace Borel-$\sigma$-field $\mathcal B$ of $\mathcal F$. Then by standard results (see Section \ref{obs} below), for any $u_f \in L^2(\mathcal O)$ we can define the likelihood function $p_f(Y)$ describing the density of the law $P_f^Y\equiv P_{u_f}^Y$ of $Y|f$ in (\ref{model0}) for a suitable fixed dominating measure. The posterior distribution $\Pi(\cdot|Y)$ is the law of $f|Y$ obtained from Bayes' formula
\begin{equation}\label{posterior}
\Pi(B|Y) = \frac{\int_B p_f(Y) d\Pi(f)}{\int_\mathcal F p_f(Y) d\Pi(f)} \equiv \frac{\int_B e^{\ell(f)}  d\Pi(f)}{\int_\mathcal F e^{\ell(f)} d\Pi(f)} ,~~B \in \mathcal B.
\end{equation}
As is common in the inverse problems literature, we consider the situation where $f$ has some regularity, say a prescribed number of bounded continuous derivatives on $\mathcal O$. We will see how the regularity influences our ability to reconstruct $f$ from $Y,g$, and we will show that a suitable Bayesian algorithm recovers $f$ in an optimal way in various loss functions, such as $L^2(\mathcal O)$ or certain Sobolev norms. 

\section{A posterior consistency result}

\subsection{Basic notation}\label{basic} 

For $\mathcal O \subset \mathbb R^d$ an open set and a multi-index $i=(i_1,\dots, i_d)$, $i_j \in \mathbb N \cup \{0\},$ of length $|i|=\sum_j i_j$ (not to be confused with $|\cdot|$ otherwise denoting the standard Euclidean norm on $\mathbb R^d$), and $D^i$ the associated (weak) partial differential operator, the usual Sobolev spaces are defined as 
\begin{equation}\label{sobsp}
H^\alpha(\mathcal O) = \left\{f \in L^2(\mathcal O): D^i f \in L^2(\mathcal O) ~\forall |i| \le \alpha\right\},~~ \alpha \in \mathbb N,
\end{equation}
where $H^0(\mathcal O)=L^2(\mathcal O)$. When no confusion may arise we will sometimes omit $\mathcal O$ in the notation, and we take as norm on $H^{\alpha}$ the functional 
\begin{equation} \label{sobnorm}
\|f\|_{H^\alpha} = \sum_{|i| \le \alpha}\|D^if\|_{L^2}.
\end{equation}
For $\alpha \notin \mathbb N$ one defines  $H^\alpha(\mathcal O)$ by interpolation, see Chapter I.9 in \cite{LM72} or Section 7 in \cite{AF03}. 

We can similarly define the spaces $C^\alpha(\mathcal O), \alpha \in \mathbb N \cup \{0\},$ by replacing $(L^2, \|\cdot\|_{L^2})$ in (\ref{sobsp}), (\ref{sobnorm}) by the space $(C(\mathcal O), \|\cdot\|_\infty)$ of bounded uniformly continuous functions on $\mathcal O$ -- for elements $f \in C^\alpha(\mathcal O)$ the functions $D^if, 0 \le |i| \le \alpha,$ then all have unique continuous extensions to $\bar {\mathcal O}$, and we sometimes write $C^\alpha(\bar {\mathcal O})$ to make explicit that we view such functions as being defined on $\bar {\mathcal O}$. The symbol $C^\infty(\mathcal O)$ denotes the set of all infinitely-differentiable functions on $\mathcal O$, and $C_0(\mathcal O)$ denotes those elements in $C(\bar{\mathcal O
})$ whose boundary trace $h_{|\partial \mathcal O}=0$ on $\partial \mathcal O$. For $\alpha \notin \mathbb N$ we say $f \in C^\alpha(\mathcal O)$ if $D^{i}f, |i| = [\alpha]$, where $[\alpha]$ is the integer part of $\alpha$, exists and is $\alpha-[\alpha]$-H\"older continuous. The norm on the space $C^\alpha(\mathcal O)$ is then given by 
\begin{equation}\label{holdlognorm}
\|f\|_{C^{\alpha}(\mathcal O)} \equiv \|f\|_{C^{[\alpha]}(\mathcal O)} + \max_{|i|=[\alpha]}\sup_{x, y \in \mathcal O; x \neq y} \frac{|D^{i}f(x)-D^{i}f(y)|}{|x-y|^{\alpha-[\alpha]}}.
\end{equation}
We also need H\"older-Zygmund spaces $\mathcal C^\alpha(\mathcal O), \alpha \ge 0$, see Section 3.4.2 in \cite{T83} for definitions. One has $\mathcal C^\alpha(\mathcal O)=C^\alpha(\mathcal O), \alpha \notin \mathbb N,$ (equivalent norms) and the continuous imbeddings $\mathcal C^{\alpha'} \subsetneq C^\alpha \subsetneq \mathcal C^{\alpha}, \alpha \in \mathbb N \cup \{0\}, \alpha'>\alpha$. Attaching a subscript $c$ to any space $S(\mathcal O)$ will denote the subspace $(S_c(\mathcal O), \|\cdot\|_{S(\mathcal O)})$ consisting of functions of compact support in $\mathcal O$, and $(S_K(\mathcal O), \|\cdot\|_{S(\mathcal O)})$ will denote the subspace consisting of elements of $S(\mathcal O)$ supported in a subset $K \subset \mathcal O$. 
 
\smallskip

For $\alpha>d/2$ and $\mathcal O$ a bounded $C^\infty$-domain, the Sobolev imbedding implies that $H^\alpha(\mathcal O)$ embeds continuously into $C^\beta(\bar{\mathcal O})$ for any $0 <\beta < \alpha-d/2$  and we further have 
\begin{equation} \label{mult0}
\|fg\|_{H^\alpha} \le c \|f\|_{H^\alpha} \|g\|_{H^\alpha},~~\alpha>d/2,
\end{equation}
for some $c=c(\alpha, d, \mathcal O)$. The above facts are classical for $\alpha \in \mathbb N$ (see \cite{AF03}) and also hold for $\alpha \notin \mathbb N$ by the use of interpolation theory \cite{LM72, T83}. We also repeatedly use the inequalities
\begin{equation}\label{mult}
\|fg\|_{H^\alpha} \le c\|f\|_{\mathcal C^\alpha} \|g\|_{H^\alpha}, ~ \|fg\|_{\mathcal C^\alpha} \leq c\|f\|_{\mathcal C^\alpha} \|g\|_{\mathcal C^\alpha},~~\alpha \ge 0.
\end{equation}
which follow from Remark 1 on p.143 in \cite{T83}.

\smallskip

On a Hilbert space $H$ we will denote the inner product generating the norm $\|\cdot\|_H$ by $\langle \cdot, \cdot \rangle_H$. For an arbitrary normed linear space $(X, \|\cdot\|_X)$,  the topological dual space is $$X^* = (X,\|\cdot\|_X)^* := \{L: X \to \mathbb R \text{ linear s.t. } |L(x)| \le C\|x\|_X \text{ for all } x \in X \text{ and some } C>0\},$$ which is a Banach space for the norm $\|L\|_{X^*} \equiv \sup_{x \in X, \|x\|_X \le 1} |L(x)|.$

\smallskip

If $\mu$ is a probability measure on some measurable space, then $X \sim \mu$ means that $X$ is a random variable in that space drawn from distribution $\mu$, also called the law $\mathcal L(X)=\mu$ of $X$. We write $X =^\mathcal L Y$ if two random variables $X,Y$ have the same law $\mathcal L(X)= \mathcal L(Y)$.
\smallskip

We will sometimes use the symbols $\lesssim, \gtrsim, \simeq$ to denote one- or two-sided inequalities up to multiplicative constants that may either be universal or `fixed' in the context where the symbols appear. We also use the standard $O_P,o_P, O, o$ notation to estimate the order of magnitude of sequences of random variables and real numbers, respectively.

\subsection{Wavelet bases for $L^2(\mathcal O)$}\label{fsp0}

Consider an orthonormal system of sufficiently smooth (`$S$-regular'), compactly supported Daubechies tensor wavelet basis functions $$\{\Phi_{l,r}: r \in \mathbb Z^d, l \in \mathbb N \cup \{-1,0\}\},~\Phi_{lr} = 2^{ld/2}\Phi_{0,r}(2^l\cdot)~\text{for } l \ge 0,$$ of the Hilbert space $L^2(\mathbb R^d),$ see \cite{M92, D92} and also Chapter 4 in \cite{GN16}. [We shall in proofs sometimes use the last dilation identity also when $l=-1$, in slight abuse of notation.] We will use the fact that such a basis characterises elements of classical function spaces on $\mathbb R^d$ by the decay of the sequence norms of wavelet coefficients; for instance
\begin{equation} \label{sobwav}
\|f\|_{H^\alpha (\mathbb R^d)}^2 \simeq \sum_{l,r} 2^{2l\alpha} \langle f, \Phi_{l,r} \rangle_{L^2(\mathbb R^d)}^2
\end{equation}
and for some constant $C>0$ and all $\alpha \ge 0$,
\begin{equation} \label{holdwav}
f \in C^\alpha (\mathbb R^d) \Rightarrow \sup_{l,r} 2^{l(\alpha+d/2)} |\langle f, \Phi_{l,r} \rangle_{L^2(\mathbb R^d)}| \le C \|f\|_{C^\alpha (\mathbb R^d)},
\end{equation}
with a converse of the last inequality holding as well when $\alpha \notin \mathbb N$. To be precise, the previous inequalities hold for all $\alpha \le S$, where $S \in \mathbb N$ measures the `regularity' of the wavelet basis used, in particular the $\Phi_{0,r} \in C^S_c(\mathbb R^d)$. Note that $S$ can be chosen as large as desired.

\smallskip

For a bounded $C^\infty$-domain $\mathcal O$ in $\mathbb R^d$, one can then also construct an orthonormal wavelet basis of the Hilbert space $L^2(\mathcal O)$ given by
\begin{equation} \label{basi}
\Big\{\Phi^{\mathcal O}_{l,r}: r \le N_l, l \in \mathbb N \cup \{-1,0\}\Big\}, N_l \in \mathbb N,
\end{equation}
consisting of all those $\Phi^\mathcal O_{l,r}=\Phi_{l,r}$ that are compactly supported within $\mathcal O$, and of some boundary corrected wavelets $\Phi_{l,r}^\mathcal O=\Phi_{l,r}^{bc}$ which are an orthonormalised linear combination $$\Phi^{bc}_{l,r}(x) = \sum_{|m-m'| \le K} d^l_{m,m'} \Phi_{l,m'}(x),~~ m=m(l,r), K \in \mathbb N,~x \in \mathcal O,$$ of those basic Daubechies wavelets $\Phi_{l,m'}$ that have support both in and outside $\mathcal O$. We refer to Theorem 2.33 (and Definition 2.4) in \cite{T08} for details, but record the key properties that for all $l$ and some fixed positive constants $c_0,c_1,D$,
\begin{equation}\label{wavprop}
N_l \le c_0 2^{ld}; \sum_{|m-m'|\le K} |d^l_{m,m'}| \le D;~~\supp(\Phi^{bc}_{l,r}) \subset \big\{x \in \mathcal O: |x-\partial \mathcal O| < \frac{c_1}{2^{l}}\big\}; \sum_r |\Phi_{0,r}| \in C(\mathbb R^d),
\end{equation}
the last property holding as well with $\Phi_{0,r}$ replaced by any derivative $D^i\Phi_{0,r}, |i| \le S,$ of $\Phi_{0,r}$.

\smallskip

For the above basis any function $f \in L^2(\mathcal O)$ has orthogonal wavelet series expansion $$f = \sum_l \sum_{r =1}^{N_l} \langle f, \Phi^\mathcal O_{l,r} \rangle_{L^2(\mathcal O)} \Phi^\mathcal O_{l,r}~~\text{ in } L^2(\mathcal O),$$ and we denote by $\Pi_{V_J}(f)$ the $J$-th partial sum of this series, equal to the $L^2(\mathcal O)$-projection onto the linear span $V_J$ of $\{\Phi^\mathcal O_{l,r}: r \le N_l, l \le J\}$.

\smallskip

We next define H\"older-Zygmund type spaces for this wavelet basis as 
\begin{equation} \label{hzbc}
f \in \mathcal C^{\alpha,W}(\mathcal O) \iff \sup_{l,r}2^{l(\alpha+d/2)}  |\langle f, \Phi^\mathcal O_{l,r}\rangle_{L^2(\mathcal O)}| \equiv \|f\|_{\mathcal C^{\alpha, W}(\mathcal O)} <\infty,
\end{equation}
a definition that makes sense for all values $\alpha \in \mathbb R$ if $f$ is a linear functional whose action $f(\Pi_{l,r}^\mathcal O)$ on the $\{\Phi_{l,r}^\mathcal O\}$ is well-defined. 

The above boundary corrected wavelet basis conveniently retains the multi-scale and orthonormal basis properties in $L^2(\mathcal O)$, but may not model the regularity of a function $f$ correctly near $\partial \mathcal O$. Describing smoothness of functions near the boundary by decay of wavelet coefficients is a delicate problem that we avoid here. For our purposes it will be sufficient that the spaces $\mathcal C^{\alpha}_c(\mathcal O)$ and thus also $C^\alpha_c(\mathcal O)$ are continuously embedded into $\mathcal C^{\alpha, W}(\mathcal O)$ if $\alpha<S$, see Proposition \ref{hzimb}.

\subsection{Construction of the prior distribution}

Since $f > 0$ is assumed, we seek a prior $\Pi$ for $f$ that is the law of a `generic' non-negative random function that possesses enough regularity so that the solution $u_f$ of (\ref{PDE}) exists for every $f \sim \Pi$. We work with the $S$-regular wavelet basis (\ref{basi}) of $L^2(\mathcal O)$ from the previous subsection, where $0<s<S$ is a fixed integer. We then take as prior distribution $\Pi =\Pi_J$ the law $\mathcal L(f)$ of the random function
\begin{equation}\label{prior}
f =e^{\varphi(x)} \equiv \exp\left\{\sum_{l\le J} \sum_{r=1}^{N_l} b_{l,r} \Phi^\mathcal O_{l,r} (x)\right\},~x \in \mathcal O,~J \in \mathbb N,
\end{equation}
where, for every $l$, the $(b_{l,r}: r =1, \dots, N_l)$ are drawn independently and identically from the uniform distribution $$U(-B 2^{-l(s+d/2)}\bar l^{-2}, B 2^{-l(s+d/2)}\bar l^{-2}),~~\bar l = \max (l,1).$$  For the prior to be fully adaptive one would have to further model $B, s, J$ by hierarchical priors, but for the results that follow we will restrict to the case where $s$ is given and $B$ an arbitrary but fixed positive constant. The truncation point $J \in \mathbb N$ will be chosen to increase as the noise level $\varepsilon$ decreases -- thus $\Pi$ is a `high-dimensional' prior.

\smallskip

The factors $2^{-l(s+d/2)}\bar l^{-2}$ appearing in the weights $b_{l,r}$ imply that  a function drawn from the prior lies in a ball of $C^s(\mathcal O)$ almost surely:  From (\ref{wavprop}) and if $|i| \le s<S$ we deduce that $$|D^i \phi| = \left|\sum_{l \le J,r} b_{l,r} D^i \Phi^\mathcal O_{l,r}\right| \le B \sum_{l\le J,r} \bar l^{-2}2^{l(|i|-s)} |D^i \Phi_{0,r}| \le cB  \sum_{l\le J} \bar l^{-2} \le C(B)$$ for some finite constant $C(B)$, hence $\varphi$ is contained in $C^\alpha(\mathcal O)$ for any $|\alpha| \le s, s \in \mathbb N$. Since the exponential map is smooth on bounded sets this further implies that 
\begin{equation} \label{b01}
\|\phi\|_{C^s(\mathcal O)} \le C, ~\|f\|_{C^s(\mathcal O)} \le C,
\end{equation}
for some finite constant $C$ that depend only on $s,B$ and the wavelet basis used. 

\smallskip

`Besov priors' as in (\ref{prior}) have been proposed in inverse problems settings before, see \cite{LSS09, SS12, DHS12, R13}, particularly \cite{SS12} studies such series priors with uniformly distributed coefficients in related PDE-type inverse problems. Other priors than (\ref{prior}) may be of interest, for instance those where $b_{l,r}$ are drawn from a Gaussian or Laplace distribution. The mathematical techniques we develop in the present paper apply in principle to such priors too, however the assumption that $\Pi$ is supported in a fixed ball of $C^s(\mathcal O)$ is used crucially in many places in our proofs, and cannot easily be relaxed. Generalising our results to priors with unbounded coefficients $b_{l,r}$ remains a challenging open problem for future research.

\subsection{A contraction theorem in $L^2$ and in Sobolev norms}

Our first result states that the posterior distribution is consistent in that it concentrates around any `true' function $f_0$ that generates equation (\ref{model0}), and we quantify the `contraction rate' in $L^2$ in terms of the smoothness $s$ of $f_0$. 

To reduce technicalities related to boundary issues we restrict to functions $\varphi_0 = \log f_0 \in C^{s}(\mathcal O)$ that have compact support in $\mathcal O$. Any such $\varphi_0$ is contained in $C^{s}_c(\mathcal O) \subset \mathcal C^{s,W}(\mathcal O)$ by Proposition \ref{hzimb}, which implies by (\ref{hzbc}) that the wavelet coefficients of $\varphi_0$ decay like $2^{-l(s+d/2)}$ as $l \to \infty$. We will strengthen this decay assumption slightly to 
\begin{equation}\label{decay}
\sup_{l,r}2^{l(s+d/2)}\bar l^2 |\langle \varphi_0, \Phi_{l,r}\rangle_{L^2(\mathcal O)}| \le B
\end{equation}
to exactly match the decay of the coefficients $b_{l,r}$ in the wavelet prior (\ref{prior}) -- this is equivalent to assuming some additional H\"older-regularity of $\varphi_0$ on the logarithmic scale. We will also assume $s \in \mathbb N$ for simplicity. Finally, and without loss of generality, we realise $g=g_{| \partial \mathcal O}$ as the boundary trace of an element of $C^{s+2}(\bar{\mathcal O})$. We then have the following contraction result.
 
\begin{theorem}\label{rates}
Let $f_0>0$ be such that $\varphi_0 = \log f_0 \in C^{s}_c(\mathcal O)$ satisfies (\ref{decay}) for some $s>2+d/2$, $s \in \mathbb N$. Let $\Pi=\Pi_J$ be as in (\ref{prior}) with $J \in \mathbb N$ such that $2^{J} \simeq \varepsilon^{-2/(2s+4+d)}$, and let $\Pi(\cdot|Y)$ be the resulting posterior distribution (\ref{posterior}) arising from observing (\ref{model0}) with $g \in C^{s+2}(\bar{\mathcal O}), g \ge g_{\min}>0$. 

If $P_{f_0}^Y$ is the law of $Y=u_{f_0} + \varepsilon \mathbb W$, then for all $M$ large enough and  $\gamma=s/(2s+4)$ we have as $\varepsilon \to 0$ that $$\Pi\left(f:\|f-f_0\|_{L^2(\mathcal O)} > M \varepsilon^{2s/(2s+4+d)} \log^\gamma (1/\varepsilon) |Y\right) \to^{P_{f_0}^Y} 0.$$
\end{theorem}

\smallskip

Since (\ref{b01}) implies that $f-f_0$ is bounded in $C^s(\mathcal O) \subset H^s(\mathcal O)$, the standard interpolation inequality for Sobolev norms (see (\ref{lionsm}) below) further implies the contraction rates
\begin{equation}
\Pi\left(f:\|f-f_0\|_{H^\alpha(\mathcal O)} > M \varepsilon^{(2s-2\alpha)/(2s+4+d)} \log^{\gamma'} (1/\varepsilon) |Y\right) \to^{P_{f_0}^Y} 0,~~0 \le \alpha < s,
\end{equation}
with $\gamma'=(s-\alpha)/(2s+4),$ which by the Sobolev imbedding (any $\alpha>d/2$) also implies a contraction rate in the uniform norm $\|\cdot\|_\infty$.

\smallskip

Except for the log-factor, the rates obtained in Theorem \ref{rates} are optimal, as the following proposition shows. Removal of the log-factor is possible by more sophisticated prior choices, but for the main results of this article that follow, this will not be relevant.

\smallskip

\begin{proposition} \label{sobrates}
We have for every $s>2+d/2, B>0, 0 <\alpha < s, f_{\min}>0$, any $g$ as in Theorem \ref{rates} and $E_f^Y$ the expectation operator corresponding to $P_{f}^Y$ that as $\varepsilon \to 0$ $$\inf_{\tilde f=\tilde f(Y,g)}~\sup_{f \in C^s(\mathcal O): \|f\|_{C^s(\mathcal O)} \le B, f \ge f_{\min} >0} ~\varepsilon^{-(2s-2\alpha)/(2s+4+d)} E_f^Y \|\tilde f - f\|_{H^\alpha(\mathcal O)} \simeq c,$$ where $c = c(\alpha, s,B, g_{\min}, f_{\min}, d, vol(\mathcal O))$ is a finite positive constant.

\end{proposition}

\smallskip

Inspection of the proof of this proposition reveals that the minimax convergence rate $\varepsilon^{(2s-2\alpha)/(2s+4+d)}$ improves by a  at most a power of $\log(1/\varepsilon)$ if assumption (\ref{decay}) is also imposed on $f$, and also that in case $\alpha=0$, $L^2$-loss can be replaced by $L^p$-loss for any $1\le p<\infty$. Moreover the lower bound for the minimax rate can be shown to remain valid when $f$ is such that $\log f$ has compact support in $\mathcal O$.

\section{Information Geometry}

While Theorem \ref{rates} is already quite satisfactory in that it shows that the posterior concentrates near the `true' value, we may hope that the posterior distribution reveals more precise information about $f_0$. This will ultimately be formalised by proving a Bernstein-von Mises theorem in a suitable function space, showing that the posterior distribution is approximated by a canonical Gaussian distribution with a covariance structure that is `minimal' in an information-theoretic sense. In the language of asymptotic statistics \cite{IK81, LC86, vdV98} the structure of this covariance is determined by the `LAN'-expansion of the log-likelihood ratio process, which we obtain now for the observation scheme (\ref{model0}) considered here. From it we can then construct the Gaussian measure that has to appear as the limit of the scaled and centred posterior measure.

\subsection{Score operator and LAN expansion}

We start with the following simple lemma which is a standard application of equation (\ref{likeli}) in the Appendix.
\begin{lemma} \label{nonlan}
Let $\ell(f) = \ell (f,Y) = \log p_f(Y)$ be the log-likelihood function arising from (\ref{model0}), write $G(f)=u_f$, and assume $Y= G(f_0) + \eps \mathbb W$ for some fixed $f_0$ such that $G(f_0) \in L^2(\mathcal O)$. Then for any $f,g$ for which $G(f),G(g)$ are contained in $L^2(\mathcal O)$ we have
$$\ell(f)-\ell(g) = -\frac{1}{2\varepsilon^2} \left(\| G(f) - G(f_0)\|_{L^2(\mathcal O)}^2 - \|G(g) - G(f_0)\|_{L^2(\mathcal O)}^2\right) + \frac{1}{\varepsilon} \langle G(f)-G(g), \mathbb W \rangle_{L^2(\mathcal O)}.$$
\end{lemma}

Since $G$ is non-linear we next have to find a suitable linear approximation to $G(f)-G(f+h)$ for a small perturbation $h$ of $f$. In statistical terminology this amounts to finding the `score operator' of the model. The following proposition shows that this score operator can be understood as the solution operator for an \textit{inhomogeneous} Schr\"odinger equation (\ref{scoreop}) (see Section \ref{sop} for existence of its solutions and further details). The remainder terms of the linear approximation can be controlled in a `weak' norm that we define now: 
 we introduce the functional
\begin{equation}\label{h0norm}
\|f\|_{(H_0^2)^*} = \sup_{\varphi \in C_0(\mathcal O), \|\varphi\|_{H^2(\mathcal O)}\le 1}\left| \int_\mathcal O \varphi(x) f(x)dx\right|.
\end{equation}
We clearly have $\|f\|_{(H_0^2)^*} \le \|f\|_{L^2}$ but $\|\cdot\|_{(H_0^2)^*}$ generates a strictly weaker topology than the norm of $L^2(\mathcal O)$, a fact that will be crucial in our proofs. 

We also write $\|h\|_{(d)}=\|h\|_{L^2}$ when $d<4$ but $\|h\|_{(d)}=\|h\|_\infty$ for $d \ge 4$. This accommodates the fact that the Sobolev space $H^2(\mathcal O)$, which plays a key role for regularity estimates of solutions of the Schr\"odinger equation, embeds into $C(\mathcal O)$ only when $d<4$.

\begin{proposition} [Score operator] \label{diff}
Let $G(f)=u_f$ solve (\ref{PDE}) where $f\in C^s(\mathcal O), f > 0,$ and $g \in C^{s+2}(\bar{\mathcal O}),$ for some $s>0$. For $h \in C(\bar {\mathcal O})$, denote by $DG_f[h]=v_h$ the solution $v$ to the inhomogeneous Schr\"odinger equation
\begin{equation} \label{scoreop}
\frac{\Delta v}{2}  - f v = h u_f \text{ on } \mathcal O~\text{s.t. } v=0 \text{ on } \partial \mathcal O.
\end{equation}
Let $f, h$ further be such that $f+h \ge f_{\min}>0, \|f\|_{C^s(\mathcal O)} +\|h\|_{C^s(\mathcal O)} \le D$ for some $D, f_{\min}>0$. Then there exists $c_1=c_1(\mathcal O, d, D, f_{\min}, \|g\|_{C^{s+2}(\bar{\mathcal O})})$ such that $$\|G(f+h) - G(f) -  DG_f[h]\|_{L^2} \le c_1\|h\|_{(d)} \|h\|_{(H^2_0)^*}.$$  Moreover the score operator $DG_{f}$ satisfies the estimate
\begin{equation} \label{smt}
\|DG_{f}[h]\|_{L^2} \le c_2 \|h\|_{(H^2_0)^*} \le c_2 \|h\|_{L^2},~~ h \in C(\bar{\mathcal O}),
\end{equation}
where the constant $c_2$ depends only on upper bounds for $\|f\|_{C^s(\mathcal O)}, \|g\|_{C^{s+2}(\bar{\mathcal O})}$ and on $\mathcal O, d$.
\end{proposition}
\begin{proof}
Let us write $u=u_f$ and $u_h=u_{f+h}$ in slight abuse of notation, in this proof. Then in $\mathcal O$ we must have
$$(\Delta/2) u_h - (f+ h) u_h = (\Delta/2) u - f u $$ or equivalently
$$(\Delta/2) (u_h-u) - (f+h) (u_h -u) = h u.$$ Thus if $v_h$ solves (\ref{scoreop}) then 
$$(\Delta/2) (u_h - u - v_h) - (f+h)(u_h - u -  v_h) =   h v_h$$ on $\mathcal O$ and $u_h - u - v_h =0$ on $\partial \mathcal O$. Conclude that $w_h=u_h - u -  v_h$ is itself a solution to the inhomogeneous Schr\"odinger equation 
\begin{equation} \label{sol}
(\Delta/2) w_h - (f+ h) w_h = h v_h,~~ f+h \ge f_{\min}>0,
\end{equation}
with zero boundary condition on $\partial \mathcal O$. Lemma \ref{ape}, specifically (\ref{ape-0}) below,  and definition (\ref{h0norm}) now imply that $$\|w_h\|_{L^2} \le C \|h v_h\|_{(H^2_0)^*} \le C \|h\|_{(H^2_0)^*} \sup_{\|\phi\|_{H^2}\le 1}\|\phi v_h\|_{H^2}.$$ Using (\ref{mult0}) when $d<4$ gives $\|\phi v_h\|_{H^2} \le c\|\phi\|_{H^2}\|v_h\|_{H^2}$ whereas for $d\ge 4$ we use (\ref{mult}) and have to replace $\|v_h\|_{H^2}$ by $\|v_h\|_{\mathcal C^2}$ in the last inequality. Finally since $v_h$ solves (\ref{scoreop}), by the regularity estimates in Lemma \ref{ape} with $\beta=0$ and since $\|u_f\|_{\mathcal C^2} \le c(D, \|g\|_{C^{s+2}(\bar{\mathcal O})})$ by Proposition \ref{sbaest},
$$\|v_h\|_{H^2}  \le C' \|u_f h\|_{L^2} \le C'' \|h\|_{L^2},~\|v_h\|_{\mathcal C^2} \le C' \|u_f h\|_{\mathcal C^0} \le C'' \|h\|_\infty,$$ proving the first claim. Inequality (\ref{smt}) now follows similarly, using the definition of $DG_f$ and again Lemma \ref{ape}, (\ref{mult}) and the bound on $\|u_{f}\|_{\mathcal C^2}$.
\end{proof}

We immediately obtain the path-wise LAN expansion for the log-likelihood ratio process.

\begin{proposition}
For any $f_0, h \in C^s(\bar{\mathcal O}), f_0>0, g \in C^{s+2}(\bar{\mathcal O}), s>0,$ if $Y=G(f_0)+\varepsilon \mathbb W$ then we have as $\varepsilon \to 0$ that
$$\log \frac{p_{f_0 + \varepsilon h}}{p_{f_0}}(Y) = \langle DG_{f_0}[h], \mathbb W \rangle_{L^2} - \frac{1}{2}\|DG_{f_0}[h]\|_{L^2}^2 +o_{P_{f_0}^Y}(1),$$ and the LAN-norm is given by $$\|h\|_{LAN} \equiv \|DG_{f_0}[h]\|_{L^2}.$$
\end{proposition}
\begin{proof}
Just combine Lemma \ref{nonlan} with Proposition \ref{diff} and notice that $\varepsilon^2 \|h\|_{(d)} \|h\|_{(H^2_0)^*} \le \varepsilon^2 \|h\|^2_\infty$ as well as that $\langle g, \mathbb W \rangle =O_P(\|g\|_{L^2})$ by Markov's inequality.
\end{proof}

\subsection{Information bound for one-dimensional subproblems}\label{infbd}

To apply the general theory of statistical efficiency (see Section \ref{crlbs} in the appendix for a review) for estimating a linear functional $\Psi(f)$ of $f$ at $f_0$ we need to find the Riesz-representer of such a functional for the LAN-inner product $\langle \cdot, \cdot \rangle_{LAN} \equiv \langle DG_{f_0}[\cdot], DG_{f_0}[\cdot]\rangle_{L^2(\mathcal O)}.$ We restrict to functionals of the form $\Psi(f)= \langle f, \psi \rangle_{L^2}$ for some fixed compactly supported function $\psi$. Using the results about the $f$-Green operator $V_f$ for the Schr\"odinger operator $S_f$ from Section \ref{sop} we have $$\langle h_1, h_2 \rangle_{LAN} = \langle DG_{f_0}[h_1], DG_{f_0}[h_2] \rangle_{L^2} = \langle V_{f_0} [u_{f_0}h_1], V_{f_0}[u_{f_0}h_2] \rangle_{L^2}.$$
For $f_0 \in C^s(\bar{\mathcal O}), g \in C^{s+2}(\bar{\mathcal O}), f_0,g>0,s>0,$ we know from Proposition \ref{sbaest} and (\ref{jensen}) that $u_{f_0} \ge c>0$ on $\bar{\mathcal O}$, and if also $s>2$ then the chain rule further implies $1/u_{f_0} \in C^4(\mathcal O)$. As a consequence for any $\psi \in C^4_c(\mathcal O)$ we can define
\begin{equation} \label{tildp}
\tilde \Psi = \tilde \Psi_{f_0} = [S_{f_0} S_{f_0}[\psi/u_{f_0}]]/u_{f_0}
\end{equation}
which has the point-wise representation
\begin{equation} \label{tildp2}
\tilde \Psi u_{f_0} =   \frac{1}{4} \Delta^2 \left[\frac{\psi}{u_{f_0}}\right] -\frac{1}{2} \Delta \left[\frac{f_0\psi}{u_{f_0}} \right] - \frac{f_0}{2} \Delta \left[\frac{\psi}{u_{f_0}}\right]+ \frac{f_0^2 \psi}{u_{f_0}},
\end{equation}
and hence is also compactly supported in $\mathcal O$. 
Then for all $h \in C(\bar {\mathcal O})$, Proposition \ref{inverse} implies
 \begin{equation} \label{invop}
\langle \tilde \Psi, h \rangle_{LAN} = \langle V_{f_0}[u_{f_0} \tilde \Psi], V_{f_0}[h u_{f_0}] \rangle_{L^2} =\langle V_{f_0}V_{f_0} S_{f_0} S_{f_0} [\psi/u_{f_0}], h u_{f_0} \rangle_{L^2} = \langle \psi, h \rangle_{L^2}
\end{equation}
so that $\tilde \Psi$ is indeed the Riesz-representer. From the results in Section \ref{crlbs} with $H= C^s(\bar{\mathcal O})$ and again Proposition \ref{inverse} it now follows that the information lower bound for estimating $\langle f, \psi \rangle$ at $f=f_0$ from observations $Y$ in (\ref{model0}) equals 
\begin{equation}
\|\tilde \Psi_{f_0}\|_{LAN}^2 = \|V_{f_0}[u_{f_0}\tilde \Psi_{f_0}]\|^2_{L^2} = \|S_{f_0}[\psi/u_{f_0}]\|_{L^2}^2.
\end{equation}

\subsection{The `canonical' Gaussian measure $\mathcal N_{f_0}$} \label{canmes}

Collecting the Riesz-representers $\tilde \Psi$ for all linear functions $(\Psi(f)=\langle f, \psi \rangle_{L^2(\mathcal O)}: \psi \in C_c^\infty(\mathcal O))$ from Section \ref{infbd}, and for $S_{f_0}$ the Schr\"odinger operator from Section \ref{sop}, we define the centred Gaussian process 
\begin{equation} \label{marg2}
(X(\psi): \psi \in C^\infty_c(\mathcal O)) ~\textit{s.t.}~ E[X(\psi)X(\psi')] = \langle S_{f_0}[\psi/u_{f_0}], S_{f_0}[\psi'/u_{f_0}] \rangle_{L^2}.
\end{equation}
By Kolmogorov's consistency theorem (Section 12.1 in \cite{D02}) this process defines a probability measure $\mathcal N_{f_0}$ on the cylindrical $\sigma$-field of $\mathbb R^{C_c^\infty(\mathcal O)}$.

In the Bernstein-von Mises theorem that follows we will want to prove weak convergence in probability of the centred and scaled posterior measure towards $\mathcal N_{f_0}$ in a suitable Banach space. The notion of weak convergence we will employ below naturally requires tightness of the limit law, and we thus first investigate when the measure $\mathcal N_{f_0}$ extends to a tight Gaussian Borel probability measure in the scale of dual spaces $(C^\alpha_c(\mathcal O))^*, \alpha >0$.

\begin{proposition} \label{gauss0} 
Let $f_0 \in C^s(\mathcal O), g \in C^{s+2}(\bar{\mathcal O}), s>0,$ satisfy $f_0,g>0$.

\smallskip

A)  For any $\alpha>2+d/2$, the law $\mathcal N_{f_0}$ induced by the Gaussian process (\ref{marg2}) defines a tight Gaussian Borel probability measure on $(C^\alpha_c(\mathcal O))^*$. 

B) If $\alpha<2+d/2$ then for every $f_0 \in C^s(\mathcal O), s>2+d/2,$ we have $$\mathcal N_{f_0}(x:\|x\|_{(C^\alpha_c(\mathcal O))^*}<\infty)=0.$$

C) Both A) and B) remain true if $(C^\alpha_c(\mathcal O))^*$ is replaced by $(C^\alpha_K(\mathcal O))^*=(C_K^\alpha(\mathcal O), \|\cdot\|_{C^\alpha(\mathcal O)})^*$ for any compact set $K \subset \mathcal O$ with non-empty interior. 
\end{proposition}

To be precise, the proof of Part A) shows that the process (\ref{marg2}) has a version that acts linearly on $C^\infty_c(\mathcal O)$, and that the cylindrical law of that linear version on $(C^\alpha_c(\mathcal O))^*= (C^\infty_c(\mathcal O), \|\cdot\|_{C^\alpha(\mathcal O)})^* \subset \mathbb R^{C^\infty_c(\mathcal O)}$ extends to a tight Gaussian measure on the Borel-$\sigma$-field induced by $\|\cdot\|_{(C^\alpha_c(\mathcal O))^*}$, when $\alpha>2+d/2$. In Part B), the proof implies that for any version of $(X(\psi))$ and $\alpha<2+d/2$, the norm $\sup_{\psi: \|\psi\|_{C^\alpha_c(\mathcal O)}\le 1}|X(\psi)|=\infty$ almost surely.

For $\alpha>d/2$ the measure $\mathcal N_{f_0}$ is tight and thus the results in Sections \ref{infbd} and \ref{crlbs} imply that its covariance structure represents the information lower bound for estimating $f$ in $(C_c^\alpha(\mathcal O))^*$ or $(C_K^\alpha(\mathcal O))^*$, based on observations from (\ref{model0}). 

One may investigate the limiting case $\alpha=2+d/2$ further, but this requires the introduction of a `logarithmic scale' to measure smoothness of functions, and we abstain from doing so for ease of exposition. Note that one may also show that $\mathcal N_{f_0}$ is \textit{not} tight when $\alpha=2+d/2$, see Remark \ref{sudoku}.

\section{A Bernstein von Mises theorem in $(C^{\alpha}_K(\mathcal O))^*$}

We now turn to the main result of this article which gives an infinite-dimensional normal approximation of the posterior distribution arising from the prior (\ref{prior}), in the small noise limit $\varepsilon \to 0$. The following condition on the prior and on $f_0$ will be employed. It requires $f_0$ to be an `interior' point of the support of the prior (\ref{prior}) -- just as in the finite-dimensional situation a Gaussian approximation of the posterior distribution cannot be expected when the `true value' lies at the boundary of the support $\supp (\Pi)$ of $\Pi$, see also Remark \ref{lesson}.

\begin{condition}\label{interior}
Let $s \in \mathbb N$ satisfy $s>\max (2+d/2, d)$. Suppose the prior $\Pi$ arises from (\ref{prior}) with $J \in \mathbb N$ chosen such that $2^J \simeq \varepsilon^{-2/(2s+4+d)}$, and is based on a $S$-regular wavelet basis, $S>2+s+d$.  Suppose $f_0 >0$ and that $\phi_0=\log f_0 \in C_c^s (\mathcal O)$ satisfies for some fixed $\epsilon>0$ $$|\langle \phi_0, \Phi^\mathcal O_{l,r} \rangle| \le (B-\epsilon) 2^{-l(s+d/2)}\bar l^{-2},  ~\bar l = \max(1,l).$$ 
\end{condition}

Again, the restriction to $s \in \mathbb N$ is just for convenience and could be removed at the expense of introducing  further technicalities in the proofs.

\smallskip

We will interpret the (scaled and centred) posterior distribution $\Pi(\cdot|Y)$ of the random function $f|Y$ as the law of the stochastic process 
\begin{equation} \label{process}
(\varepsilon^{-1} \langle f - \bar f, \psi \rangle_{L^2(\mathcal O)}|Y: \|\psi\|_{C^\alpha_K(\mathcal O)} \le 1), ~K \subset \mathcal O,
\end{equation}
conditional on $Y$, where $\bar f=\bar f(Y)$ is the posterior mean (since $f|Y$ is bounded in $L^2(\mathcal O)$, the posterior mean $E^\Pi[f|Y]$ is well defined as a Bochner integral, see, e.g., p.68 in \cite{GN16}). As $\varepsilon \to 0$ the conditional laws of these stochastic processes will be shown to converge weakly to the law $\mathcal N_{f_0}$ of the Gaussian process $X$ from (\ref{marg2}), uniformly in the collection of functions $\{\psi:\|\psi\|_{C^\alpha_K(\mathcal O)} \le 1\}$, with $P_{f_0}^Y$-probability approaching one, whenever $\alpha>2+d/2$, and for any compact subset $K \subset \mathcal O$. Since both processes induce a linear action on $C^\alpha_K(\mathcal O)$ this is  equivalent to weak convergence in probability of the induced probability measures $\mathcal L(\varepsilon^{-1} (f-\bar f)|Y)$ to $\mathcal N_{f_0}$ on the  dual space $(C^\alpha_K(\mathcal O))^*= (C_K^\alpha(\mathcal O), \|\cdot\|_{C^\alpha(\mathcal O)})^*$.

To make these notions of convergence rigorous, recall that a sequence of probability measures $\mu_n$ on a metric space $(S,\rho)$ converges weakly to $\mu$, or $\mu_n \to^\mathcal L \mu$, if $\int_S F d\mu_n \to \int_S F d\mu$ for all bounded $\rho$-continuous functions $F: S \to \mathbb R$. Whenever $\mu$ is tight then this notion of convergence can be metrised by the bounded Lipschitz (BL) metric 
\begin{equation} \label{blmetric}
\beta_S(\nu,\mu)= \sup_{F: S \to \mathbb R, \|F\|_{Lip} \le 1} \left|\int_S F d(\mu-\nu) \right|,~~~\|F\|_{Lip} \equiv \sup_{x\in S}|F(x)| +\sup_{x\neq y, x,y \in S} \frac{|F(x)-F(y)|}{\rho(x,y)}
\end{equation}
see \cite{D14}, Theorem 3.28. It is therefore natural to say that a sequence of random probability measures converges weakly in probability to some tight limiting probability measure if their BL-metric distance converges to zero in probability. This definition is easily seen to be independent of the metric for weak convergence that is used. We refer to Section \ref{weakprob} for some basic facts about this notion of convergence.

\begin{theorem}\label{main}
Let prior $\Pi$ and $f_0$ satisfy Condition \ref{interior} for the given $s>0$, and let $\Pi(\cdot|Y)$ be the posterior distribution (\ref{posterior}) resulting from observing (\ref{model0}) with $g \in C^{s+2}(\mathcal O), g \ge g_{\min}>0$. Let $f \sim \Pi(\cdot|Y)$ conditional on $Y$, let $\bar f=\bar f(Y)=E^\Pi[f|Y]$ be the posterior mean, and let $P_{f_0}^Y$ denote the law of $Y=G(f_0)+ \varepsilon \mathbb W$. 

For every compact subset $K \subset \mathcal O$, any $\alpha>2+3d/2$ and as $\varepsilon \to 0$ we have
\begin{equation} \label{limit}
\beta_{(C^\alpha_K(\mathcal O))^*}(\mathcal L(\varepsilon^{-1} (f-\bar f)|Y), \mathcal N_{f_0}) \to^{P_{f_0}^Y} 0
\end{equation}
where $\mathcal L(\varepsilon^{-1} (f-\bar f)|Y)$ is the law on $(C^\alpha_K(\mathcal O))^*$ induced by the stochastic process (\ref{process}), and where $\mathcal N_{f_0}$ is the tight Gaussian Borel probability measure on $(C^\alpha_K(\mathcal O))^*$ constructed in Proposition \ref{gauss0}. In particular $\bar f$ is an efficient estimator of $f_0$: as $\varepsilon \to 0$ and under $P_{f_0}^Y$, $$\varepsilon^{-1}(\bar f -f_0) \to^\mathcal L \mathcal N_{f_0}~~\text{in } (C^\alpha_K(\mathcal O))^*.$$
\end{theorem}

\begin{remark}\label{lesson}
From the point of view of weak convergence towards $\mathcal N_{f_0}$, the condition $\alpha>2+d/2$ would be sufficient (and necessary) in the above theorem, see Proposition \ref{gauss0}. The stronger condition $\alpha>2+3d/2$ is, however, required in Theorem \ref{main} for the prior constructed in (\ref{prior}), and we wish to give some intuition for this fact: In statistical models with a  parameter space of fixed finite dimension, it is well known that the Bernstein-von Mises theorem does not hold when the true parameter lies at the boundary of the support of the prior. When proving Bernstein-von Mises theorems in high and infinite dimensions, this phenomenon becomes `quantitative' -- this was already observed in \cite{CN14, C14} where priors with `heavy tails' were required to obtain sub-Gaussian approximations of posterior distributions. In the non-linear inverse problem setting here, the prior (\ref{prior}) strongly regularises the potential $f$ to lie in a fixed ball of $C^s(\mathcal O)$. As a consequence the `stability estimate' induced by (\ref{inv}) holds globally in the support of the prior (see the proof of Lemma \ref{tailexp}), ensuring consistency of Bayesian inversion (Theorem \ref{rates}). This choice of prior, however, also limits the directions $\psi$ along which $\varepsilon^{-1} \langle f - \bar f, \psi \rangle_{L^2(\mathcal O)}|Y$ is approximately sub-Gaussian by requiring $\psi$ to be sufficiently well aligned with the support ellipsoid described by the prior (see the `change of measure' argument in the proof, particularly in the construction of admissible `interior' directions in Lemma \ref{intpt}). While $\alpha>2+3d/2$ appears necessary for the prior (\ref{prior}), it remains an intriguing open question whether the condition $\alpha>2+d/2$ can be attained by other priors. 
\end{remark}

\subsection{Some applications to uncertainty quantification} \label{uq}

A key consequence of a Bernstein-von Mises result is that the uncertainty quantification provided by the posterior distribution for the parameter $f$ is optimal in a frequentist, prior-independent, sense, if the credible sets are constructed for the norm in which the normal approximation to the posterior holds. 

\smallskip

To start with a simple example, suppose $\psi \in C_c^\alpha(\mathcal O), \alpha>2+3d/2,$ is a fixed `test' function ($\psi\neq 0$ to avoid triviality) and our target of statistical inference is the scalar quantity $\Psi(f) = \langle f, \psi \rangle_{L^2}$. We can use the induced posterior distribution $\Pi_\psi(\cdot|Y) \equiv \Pi(\cdot|Y) \circ \Psi^{-1}$ to construct a prescribed level $1-\beta$ credible set for $\Psi(f)$ by taking
\begin{equation}\label{codi}
C_\varepsilon = \{z \in \mathbb R: |z - \langle \bar f, \psi \rangle_{L^2}| \le R_\varepsilon\},~\text{with } R_\varepsilon=R_\varepsilon(Y, \beta)~s.t.~ \Pi_\psi(C_\varepsilon|Y)=1-\beta, ~~\beta>0.
\end{equation}
The Bernstein-von Mises and continuous mapping theorems now imply that the law of $\varepsilon^{-1}(\Psi(f)-\Psi(\bar f))$ is approximated in the small noise limit by a $N(0,\|S_{f_0}(\psi/u_{f_0})\|_{L^2}^2)$-distribution. Let $\Phi(t)=\Pr(Z \in [-t,t]), t >0,$ for $Z \sim N(0,\|S_{f_0}(\psi/u_{f_0})\|_{L^2}^2)$ with continuous inverse $\Phi^{-1}:[0,1] \to \mathbb R$. The following corollary implies that the posterior credible set $C_\varepsilon$ is an efficient frequentist confidence set for the parameter $\Psi(f)$. 
\begin{corollary} \label{spconf}
The frequentist coverage probability of $C_\varepsilon$ satisfies $$P_{f_0}^Y(\Psi(f_0) \in C_\varepsilon) \to 1-\beta,$$ as $\varepsilon \to 0$ and the diameter of $C_\varepsilon$ is of order $R_\varepsilon = O_P(\varepsilon)$, in fact as $\varepsilon \to 0$, $$\varepsilon^{-1} R_\varepsilon \to^{P_{f_0}^Y} \Phi^{-1}(1-\beta).$$
\end{corollary}

\smallskip

The above result immediately generalises to fixed non-linear but differentiable functionals $\Psi$ of $f$, arguing just as in Section 2.3.2 in \cite{CN13}.

\smallskip

For the proof of the previous corollary one would not require the full strength of Theorem \ref{main}. But the infinite-dimensional approximation is required to demonstrate that also the uncertainty quantification provided by the posterior distribution for the \textit{entire} parameter $f$  --  by intersecting all admissible linear constraints in (\ref{codi}) for $\psi$ in the unit ball of $C^\alpha_K(\mathcal O)$ -- can be valid. More precisely, if we choose posterior quantiles $R_\varepsilon=R_\varepsilon(Y, \beta)$ such that 
\begin{equation} \label{pol}
C_\varepsilon=\{f \in \supp \Pi(\cdot|Y): \|f-\bar f\|_{(C^\alpha_K(\mathcal O))^*} \le R_\varepsilon\},~~ \Pi(C_\varepsilon|Y)=1-\beta,
\end{equation} 
where $\beta>0$ is some fixed significance level, and $\alpha>2+3d/2$ is arbitrary, then one can prove the following result.

\begin{corollary} \label{npconf}
The frequentist coverage probability of $C_\varepsilon$ from (\ref{pol}) satisfies $$P_{f_0}^Y(f_0 \in C_\varepsilon) \to 1-\beta,$$ and as $\varepsilon \to 0$ the quantile constants satisfy $$\varepsilon^{-1} R_\varepsilon \to^{P_{f_0}^Y} const.$$
\end{corollary}

Although the credible ball (\ref{pol}) is constructed in the weaker topology of $(C^{\alpha}_K(\mathcal O))^*$, its diameter can be shown to converge to zero in probability at rate $\varepsilon^{2(s-d-\kappa)/(2s+4+d)},$ for every $\kappa>0$, also in the much stronger $\|f\|_{L^1(\bar K)} = \int_{\bar K} |f|$ ~-norm (where the Bernstein-von Mises theorem does not hold), for any strict compact subset $\bar K \subsetneq K$. This can be proved by noting that the $L^1(\bar K)$-norm can be bounded via interpolation between the bounded support of $\Pi(\cdot|Y)$ in $C^s(\mathcal O)$ and the convergence rate in $(C^{\alpha}_K(\mathcal O))^*$. See \cite{CN13, CN14} and Section 5 in \cite{R17} for more discussion of this `multi-scale' phenomenon.

\section{Proofs}

\subsection{Proofs of Propositions \ref{gauss0} and \ref{sobrates}} \label{fsp}

To prove Proposition \ref{gauss0}A), let $X \sim \mathcal N_{f_0}$ with covariance metric $d(\psi, \psi') = \|S_{f_0}[(\psi-\psi')/u_{f_0}]\|_{L^2} \lesssim \|\psi-\psi'\|_{C^2}$. The metric entropy estimate 
\begin{equation} \label{meteor}
\log N(\{\psi: \|\psi\|_{C^\alpha_K(\mathcal O)} \le 1, \eta, d) \lesssim \log N(\{\psi: \|\psi\|_{C^{\alpha-2}_K(\mathcal O)} \le 1, c\eta, \|\cdot\|_\infty), \eta>0,
\end{equation} 
for some constant $c>0$, combined with (\ref{vdv}), $\alpha>2+d/2$, and Proposition 2.1.5 and Theorem 2.3.7 in \cite{GN16} imply that $X$ defines a tight Gaussian Borel random variable in the space of bounded and uniformly $d$-continuous functions on the unit ball of $C^\alpha_K(\mathcal O)$. That a version of $X$ exists that acts linearly on $C^\alpha_K(\mathcal O)$, thus defining an element of $(C^\alpha_K(\mathcal O))^*$, follows from the proof of Theorem 3.7.28 in \cite{GN16} and linearity in $\psi$ of the covariance in (\ref{marg2}).

\smallskip

We now prove Proposition \ref{gauss0}B) for $(C^\alpha_c(\mathcal O))^*$, the case $(C^\alpha_K(\mathcal O))^*$ is proved in the same way by replacing $\mathcal O$ with a ball contained in $K$. By the continuous imbeddings $C^{\alpha'}_c \subset C^{\alpha}_c, \alpha'>\alpha,$ it suffices to prove the result for $\alpha =2+d/2-\epsilon$ for every $\epsilon>0$ small enough. Suppose to the contrary that $$\mathcal N_{f_0}(x:\|x\|_{(C^\alpha_c(\mathcal O))^*}<\infty) = \Pr \Big(\sup_{h \in C^\alpha_c(\mathcal O): \|h\|_{C^\alpha}\le 1} |X(h)| <\infty \Big)>0,$$ then by separability of the unit ball of $C^\alpha_c(\mathcal O), \alpha >2,$ for the covariance metric $d$ of $\mathcal N_{f_0}$ (cf.~(\ref{meteor}) and (\ref{vdv})), and by Proposition 2.1.12 and Theorem 2.1.20 in \cite{GN16}, $X \sim \mathcal N_{f_0}$ must satisfy $E\|X\|_{(C^\alpha_c(\mathcal O))^*}<\infty$, which we will now lead to a contradiction.  If $V_{f_0}$ denotes the inverse Schr\"odinger operator from Section \ref{sop}, and if $X \sim \mathcal N_{f_0}$ then by Proposition \ref{inverse} the Gaussian process $\tilde X $ defined by the action $$\big(\tilde X(\phi) =X(u_{f_0} V_{f_0}[\phi]),~~ \phi \in C^\infty_c(\mathcal O)\big),$$ has the law of a standard Gaussian white noise $\mathbb W$. Using Proposition \ref{inverse} again
\begin{align*}
&E \sup_{h \in C^\alpha_c(\mathcal O): \|h\|_{C^\alpha}\le 1} |X(h)| \\
&= E \sup_{h \in C^\alpha_c(\mathcal O): \|h\|_{C^\alpha}\le 1} |X(u_{f_0} V_{f_0}[S_{f_0}(h/u_{f_0})]) | \\
& = E \sup_{h \in C^\alpha_c(\mathcal O): \|h\|_{C^\alpha}\le 1} |\tilde X (S_{f_0}(h/u_{f_0}))| \\
&= E \sup_{h \in C^\alpha_c(\mathcal O): \|h\|_{C^\alpha}\le 1} |\mathbb W(S_{f_0}(h/u_{f_0}))| \\
& \ge E \sup_{h \in C^\alpha_c(\mathcal O): \|h\|_{C^\alpha}\le 1} |\langle \mathbb W, \Delta(h/2u_{f_0}) \rangle_{L^2}|  - E \sup_{h \in C^\alpha_c(\mathcal O): \|h\|_{C^\alpha}\le 1} |\langle \mathbb W, (f_0h)/u_{f_0}) \rangle_{L^2}| \\
& \ge E \sup_{h \in C^\alpha_c(\mathcal O): \|h\|_{C^\alpha}\le 1} |\langle \mathbb W, \Delta(h/2u_{f_0}) \rangle_{L^2}| -C \\
& \ge (1/\kappa) E \sup_{\bar h \in C^\alpha_c(\mathcal O): \|\bar h\|_{C^\alpha}\le 1} |\langle \mathbb W, \Delta \bar h \rangle_{L^2}| -C
\end{align*}
for some $\kappa>0$, where we have used $f_0, u_{f_0}, 1/u_{f_0} \in C^{\alpha}(\mathcal O)$ under the maintained assumptions (Section \ref{sop}), that the supremum of the standard white noise process $\mathbb W$  on balls in $C_c^{\alpha}(\mathcal O)$ ($\alpha=2+d/2-\epsilon$, $\epsilon<2$) is bounded in expectation by a fixed constant (use Proposition \ref{pain}C and (\ref{dudmet1})), and where we have taken $h=2\bar h u_{f_0}/\|\bar h u_{f_0}\|_{C^\alpha}$ and used $2\|\bar h u_{f_0}\|_{C^\alpha} \le \kappa$ by (\ref{mult}). 

We complete the proof by showing
\begin{equation} \label{gplbd}
E \sup_{h \in C^\alpha_c(\mathcal O): \|h\|_{C^\alpha}\le 1} |\langle \mathbb W, \Delta h\rangle_{L^2(\mathcal O)}|=\infty,
\end{equation}
as follows: For every $j \in \mathbb N$ there exists a small positive constant $c_0'$ and $c_0'2^{jd}=n_j$ many Daubechies wavelets $(\Phi_{j,r}: r=1, \dots, n_j)$ that have disjoint compact support within $\mathcal O$ (Section \ref{fsp0}). Let $b_{m,\cdot}$ be a point in the discrete hypercube $\{-1,1\}^{n_j}$ and, for $\kappa'$ small enough chosen below, define functions 
\begin{equation} \label{lfp}
h_m (x) \equiv h_{m,j}(x) = \kappa' \sum_{r=1}^{n_j} b_{m,r} 2^{-j(\alpha+d/2)} \Phi_{j,r}(x), x \in \mathcal O,~m=1, \dots, 2^{n_j}.
\end{equation}
The interior wavelets $\Phi_{j,r}$ are all orthogonal to the boundary wavelets, thus
$$\|h_m\|_{\mathcal C^{\alpha,W}} = \sup_{l,r} 2^{l(\alpha+d/2)} |\langle \Phi_{l,r}^\mathcal O, h_m \rangle_{L^2(\mathcal O)}| = \kappa',$$ and since by after (\ref{holdwav}) the wavelet norm $\|\cdot\|_{\mathcal C^{\alpha,W}}$ is equivalent to $\|\cdot\|_{C^\alpha}$ on such $(h_m)$'s (we can assume $\alpha \notin \mathbb N$ by choice of $\epsilon$), the $h_m$'s are all contained in $\{h \in C^\alpha_c(\mathcal O): \|h\|_{C^\alpha}\le 1\}$ for $\kappa'$ small enough. Conclude that the supremum in (\ref{gplbd}) is lower bounded by $$E \max_m |\langle \mathbb W, \Delta h_m\rangle_{L^2(\mathcal O)}|.$$ The Gaussian process $(\mathbb W(h_m): m=1, \dots, 2^{n_j})$ has covariance metric $$d^2(h_m, h_m') = \|\Delta(h_m-h_m')\|^2_{L^2(\mathcal O)},$$ and by the Varshamov-Gilbert bound (Example 3.1.4 in \cite{GN16}) there exists $\{b_{m,\cdot}: m=1, \dots, M_j\} \subset \{-1,1\}^{n_j}$ with $M_j \ge 3^{n_j/4}$ that are $n_j/8$ -separated for the Hamming-distance. Then for the $h_m$'s corresponding to these separated $b_m$'s we have 
\begin{equation} \label{separat}
d^2(h_m, h_{m'}) = (\kappa')^2 2^{-2j(\alpha+d/2)} \left\|\sum_r (b_{m,r} - b_{m',r}) \Delta \Phi_{j,r}  \right\|_{L^2(\mathcal O)}^2.
\end{equation}
The $\Delta \Phi_{j,r}$ all have disjoint support and thus, when normalised by $\|\Delta \Phi_{j,r}\|_{L^2} = 2^{2j} \|\Delta \Phi_{0,0}\|_{L^2}$, form an orthonormal system in $L^2(\mathcal O)$. Thus the last term equals, by Parseval's identity
$$ d^2(h_m, h_{m'}) = (\kappa')^2 2^{2j(2-\alpha-d/2)} \sum_r (b_{m,r}-b_{m',r})^2 \ge c 2^{2j(2-\alpha-d/2)} n_j \ge (c')^2 2^{2j(2-\alpha)},$$ for $m \neq m'$ in the separated set. By Sudakov's lower bound (Theorem 2.4.12 in \cite{GN16} with $\varepsilon$ there chosen as $\varepsilon= c'2^{j(2-\alpha)}$) we then have
$$E \max_m |\langle \mathbb W, \Delta h_m\rangle| \ge c'' 2^{j(2-\alpha)} \sqrt {\log M_j}  \ge c''' 2^{j(2+d/2-\alpha)}$$ which for $\alpha < 2+d/2$ can be made as large as desired as $j \to \infty$, completing the proof.

\begin{remark}\label{sudoku}
A slightly more involved version of the above argument also shows that $\mathcal N_{f_0}$ cannot be tight in $(C^{\alpha}_c(\mathcal O))^*$ in the boundary case $\alpha=2+d/2$, as this would force the Gaussian process $X$ indexed by the unit ball in $C^{\alpha}_c(\mathcal O)$ to be sample continuous for its covariance metric (arguing as in Proposition 2.1.7 in \cite{GN16}). The above Sudakov' lower bound argument can then be refined (using Corollary 2.4.14 in \cite{GN16}) to also cover $\alpha=2+d/2$.
\end{remark}

We finally turn to the proof of Proposition \ref{sobrates}: The upper bound follows from Lemma \ref{tailexp} with $\eta \simeq \varepsilon^{(2s+4)/(2s+4+d)}$ when $\alpha=0$ -- since one may take $v \gtrsim  \bar D^2$ in that Lemma we can integrate tail probabilities to also bound the expected risk. The case of general $\alpha$ then follows from (\ref{lionsm}) below. For the lower bound we apply the general Theorem 6.3.2 in \cite{GN16} and arguments akin to those used in the proof of Proposition \ref{gauss0}B) just given. We assume that $g=1$  on $\partial \mathcal O$ and that $B$ is large and $f_{\min}$ small enough (chosen below), the general result requires only minor modifications. Let $\bar g$ be a smooth function such that $\bar g \ge 1$ on $\bar {\mathcal O}$ and let $g_0$ be the smooth solution to the Poisson equation $\Delta g_0 = \bar g$ on $\mathcal O$ with Dirichlet boundary conditions $g_0 = 1$ on $\partial \mathcal O$ (a unique such solution exists by Theorem 6.14 in \cite{GT98}). Since $\bar g \ge 1$ on $\mathcal O$ and $g \ge 1$ on $\partial \mathcal O$ the maximum principle for harmonic functions (Theorem 2.3 in \cite{GT98}) implies that also $g_0 \ge 1$ on $\bar {\mathcal O}$. Now define functions $$u_0=g_0, u_m = g_0 + h_m, m=1, \dots, M_j,$$ where $h_m$ is as in (\ref{lfp}) but with $\alpha$ there replaced by $s+2$, and $M_j$ as before (\ref{separat}). Arguing as after (\ref{lfp}), the $u_m$ are contained in a ball of $C^{s+2}(\mathcal O)$ of radius $\|g_0\|_{C^{s+2}}+c\kappa', c>0,$ and for $\kappa'$ small enough both $u_m$ and $\Delta u_m = \bar g + \Delta h_m$ are positive and bounded away from zero on $\bar {\mathcal O}$. Setting $2^{j} \simeq \varepsilon^{-2/(2s+4+d)}$ and since the KL-divergence in (\ref{model0}) equals $\varepsilon^{-2}$ times the squared $L^2$-distance (eq.(6.166) in \cite{GN16}), given $\epsilon>0$ we can choose $\kappa'$ small enough so that
$$\varepsilon^{-2}\|u_0-u_m\|_{L^2}^2 \le  (\kappa')^2 \varepsilon^{-2} \varepsilon^{2(2s+4+d)/(2s+4+d)} \sum_{r \le n_j}1 \le \epsilon \log M_j,$$ verifying eq.(6.103) in \cite{GN16}. Observe next that the 
$f_m = \frac{1}{2} \frac{\Delta u_m}{u_m}$ are all greater than some $f_{\min}>0$, lie in a fixed ball of $C^s(\mathcal O)$ of radius $B$, and solve the Schr\"odinger equation $\Delta u_m/2 - f_m u_m = 0 $ subject to $u_m=g$ on $\partial \mathcal O$. Now the result follows since, adapting the argument after (\ref{separat}) to the present choice of $s$, the $f_m$ are $H^\alpha(\mathcal O)$-separated by at least
$$\|f_m-f_m'\|_{H^\alpha} \ge c\|\Delta (u_m- u_m')\|_{H^\alpha} - C \|u_m^{-1}-u^{-1}_m\|_{H^\alpha} \ge c2^{-j(s-\alpha)} \simeq \varepsilon^{2(s-\alpha)/(2s+4+d)},$$ verifying the hypotheses of Theorem 6.3.2 in \cite{GN16}.

\subsection{Proof of Theorem \ref{rates}}\label{cont}

We start with a general contraction theorem that applies to priors that are supported on a fixed $s$-H\"older ball. Recall (\ref{h0norm})  for the definition of the $(H^2_0)^*$ norm. We apply the general approach of \cite{GGV00} to Bayesian contraction theorems (via Theorem \ref{ggv}), but, following \cite{GN11}, use frequentist estimators to construct the relevant `test functions'.

\begin{theorem}\label{meta}
Let $\Pi=\Pi_\varepsilon$ be a Borel probability measure supported on a measurable subset $\mathcal F$ of $L^2(\mathcal O)$ that satisfies $$ \mathcal F \subset  \Big \{\inf_{x \in \mathcal O}f(x) \ge f_{\min}, \|f\|_{C^{s}(\mathcal O)} \le D\Big\},~~D>0, f_{\min}>0, s>2+d/2, s \in \mathbb N.$$  Let $\Pi(\cdot|Y)$ be the resulting posterior distribution (\ref{posterior}) arising from observing (\ref{model0}) with $g \in C^{s+2}(\bar{\mathcal O}), g \ge g_{\min}>0$. Let $P_{f_0}^Y$ be the law generating $Y=u_{f_0}+\varepsilon \mathbb W$ for fixed $f_0 \in \mathcal F$.

Let $\eta=\eta_\varepsilon$ be a sequence satisfying
$$\eta_\varepsilon \ge \bar c \varepsilon^{(2s+4)/(2s+4+d)} $$ and 
\begin{equation} \label{sba}
\Pi(f: \|f-f_0\|_{(H_0^2)^*} < \eta_\varepsilon) \ge e^{-C(\eta_\varepsilon/\varepsilon)^2}
\end{equation}
for some constants $\bar c, C>0$ and all $\varepsilon$ small enough. Then there exists a finite constant $M$ depending only on $\bar c, C, D, d, s, \mathcal O, \|g\|_{C^{s+2}(\mathcal O)}$ such that we have as $\varepsilon \to 0$ $$\Pi(f: \|f-f_0\|_{L^2} > M \eta^{s/(s+2)}_\varepsilon|Y) \to^{P_{f_0}^Y} 0$$ and $$\Pi(f: \|f-f_0\|_{(H_0^2)^*} > M  \eta_\varepsilon|Y) \to^{P_{f_0}^Y} 0.$$ 
\end{theorem}
\begin{proof}
Proposition \ref{sbaest} implies that for each $f \in \mathcal F$ a unique solution $u_f$ to (\ref{PDE}) exists, and that $\{u_f: f \in \mathcal F\}$ is bounded in $H^{s+2}(\mathcal O)$. We apply Theorem \ref{ggv} below with $\mathcal F=\mathcal F_\varepsilon$ and its trace Borel-$\sigma$-field $\mathcal B_\mathcal F$ of $L^2(\mathcal O)$, $G(f)=u_f, \mathbb H=L^2(\mathcal O)$ and $\eta_\varepsilon=\bar \eta_\varepsilon/c$ where $c$ is the constant from Proposition \ref{sbaest}B, so that (\ref{sba}) gives
$$\Pi(f \in \mathcal F: \|G(f)-G(f_0)\|_{\mathbb H} < \bar \eta_\varepsilon) \ge \Pi(f \in \mathcal F: \|f-f_0\|_{(H_0^2)^*} < \eta_\varepsilon) \ge e^{-C(\eta_\varepsilon/\varepsilon)^2},$$ verifying (\ref{smallball}) in Theorem \ref{ggv} with $C'=Cc^2$. Thus if we can construct tests $\Psi(Y)$ such that 
\begin{equation} \label{tests0}
E^Y_{f_0}\Psi(Y) + \sup_{f \in \mathcal F, \|f-f_0\|_{L^2} \ge M \eta^{s/(s+2)}_\varepsilon} E^Y_f(1-\Psi(Y)) \le L e^{-(C'+4) (\eta_\varepsilon/\varepsilon)^2}.
\end{equation}
\begin{equation} \label{tests1}
E^Y_{f_0}\Psi(Y) + \sup_{f \in \mathcal F, \|f-f_0\|_{(H_0^2)^*} \ge M \eta_\varepsilon} E^Y_f(1-\Psi(Y)) \le L e^{-(C'+4) (\eta_\varepsilon/\varepsilon)^2},
\end{equation}
for $M$ large enough, respectively, then the result will follow from appropriate choices of $\eta^*_\varepsilon, d(\cdot, \cdot)$ in Theorem \ref{ggv}. We achieve this by first constructing an estimator for $u_{f}$ from which we obtain a plug-in test for $f$, using also (\ref{jensen}) and the resulting identification equation (\ref{inv}). The estimate of $u_f$ is obtained as the non-parametric least squares (or maximum likelihood) estimator $\hat u$ obtained from maximising the log-likelihood function (cf.~(\ref{likeli}))
\begin{equation}
\bar \ell(u) = \frac{1}{\varepsilon^2} \langle Y, u \rangle_{L^2(\mathcal O)} - \frac{1}{2 \varepsilon^2} \|u\|_{L^2(\mathcal O)}^2 ~\text{ over }~
\mathcal U_g = \big \{\|u\|_{H^{s+2}(\mathcal O)} \le D', u_{\partial \mathcal O} = g \text{ on } \partial \mathcal O \big \}.
\end{equation}
Here $D'$ is chosen large enough so that $\{u_f : f \in \mathcal F\} \subset \mathcal U_g$. By the Sobolev imbedding $H^{s+2}(\mathcal O) \subset C^1(\bar {\mathcal O})$, so by Theorem 2.4.7 in \cite{D02} the set $\mathcal U_g$ is totally bounded in $C(\bar {\mathcal O})$. Since $\{u:\|u\|_{H^{s+2}(\mathcal O)} \le D'\}$ is easily seen to be closed for the $\|\cdot\|_\infty$-topology and since limits $u$ of uniformly convergent sequences in $\mathcal U_g$ necessarily must satisfy $u=g$ on $\partial \mathcal O$, we deduce that $\mathcal U_g$ is also closed and hence a compact subset of $C(\bar {\mathcal O})$. Next, the $L^2$-metric entropy of the bounded subset $\mathcal U_g$ of $H^{s+2}(\mathcal O)$ can be shown (Chapter 3 in \cite{ET96}) to be at most 
\begin{equation}\label{etbd}
\log N(\mathcal U_g, \gamma, \|\cdot\|_{L^2(\mathcal O)}) \lesssim \left(A/\gamma \right)^{d/(s+2)},~\forall \gamma>0,~A=A(D').
\end{equation}
Since $s+2>d/2$ the square-root of the $L^2$-metric entropy is $\gamma$-integrable at zero and Proposition 2.1.5 and Theorem 2.3.7 in \cite{GN16} imply that the real-valued maps $u \mapsto \langle u, \mathbb W\rangle_{L^2}$ and then also $\bar \ell(u)$ define Borel random variables in the space of $L^2$-continuous functions on $\mathcal U_g$. Thus, using Exercise 7.2.3 in \cite{GN16} a (measurable) maximiser $\hat u \in \mathcal U_g$ of $\bar \ell(u)$ over $\mathcal U_g$ exists almost surely. We now derive its rate of convergence to $u_f$, initially in $L^2(\mathcal O)$-distance, using `peeling' techniques commonly used in $M$-estimation \cite{vdG93, vdG00}. We have for all $f \in \mathcal F$ and $D_0$ large enough,
\begin{align} \label{l2esti}
& P_{f}^Y(\|\hat u - u_{f}\|_{L^2} \ge D_0 \eta_\varepsilon) \notag \\
&= P_{f}^Y(\bar \ell(\hat u) - \bar \ell (u_{f}) \ge 0, \|\hat u - u_{f}\|_{L^2} \ge D_0 \eta_\varepsilon) \notag \\
&=P_f^Y \big(-\frac{1}{2\varepsilon^2} \|\hat u - u_{f}\|_{L^2}^2 + \frac{1}{\varepsilon} \langle \hat u - u_{f}, \mathbb W \rangle_{L^2} \ge 0, \|\hat u -u_{f}\|_{L^2} \ge D_0 \eta_\varepsilon \big) \notag \\
&=P_f^Y \left(\frac{\langle \hat u - u_{f}, \mathbb W \rangle_{L^2}}{\|\hat u - u_{f}\|_{L^2}^2} \ge \frac{1}{2\varepsilon}, \|\hat u -u_{f}\|_{L^2} \ge D_0 \eta_\varepsilon \right) \notag \\
&\le \sum_{r=0}^R \Pr \Big(\sup_{u \in \mathcal U_g: 2^r D_0 \eta_\varepsilon \le \|u-u_f\|_{L^2} \le 2^{r+1}D_0 \eta_\varepsilon} |\langle u-u_{f}, \mathbb W \rangle| \ge \frac{2^{2(r+1)} D^2_0 \eta^2_\varepsilon}{8\varepsilon} \Big),
\end{align}
where $R=R(D')<\infty$. Now using the metric entropy bound for Gaussian processes (\ref{dudmet2}) and (\ref{etbd}) we have for all $\sigma>0$ that $$E \sup_{u \in \mathcal U_g:  \|u-u_f\|_{L^2} \le \sigma} |\langle u-u_{f}, \mathbb W \rangle| \lesssim \int_0^\sigma (A/\gamma)^{d/(2s+4)}d\gamma \leq C' \sigma^{1-\frac{d}{2s+4}}.$$ For choices $\sigma =\sigma_r \equiv 2^{r+1}D_0 \eta_\varepsilon, r \ge 0$, using $\eta_\varepsilon \ge \bar c  \varepsilon^{(2s+4)/(2s+4+d)}$ and for $D_0$ large enough, we have  $C'\sigma^{1-\frac{d}{2s+4}} \le \sigma^2/(16\varepsilon)$ and thus, using the Borell-Sudakov-Tsirelson inequality (Theorem 2.5.8 in \cite{GN16}), the sum in (\ref{l2esti}) can be bounded by
\begin{align*}
& \sum_{r=0}^R \Pr \Big(\sup_{u \in \mathcal U_g:  \|u-u_f\|_{L^2} \le \sigma_r} |\langle u-u_{f}, \mathbb W \rangle| - E\sup_{u \in \mathcal U_g:  \|u-u_f\|_{L^2} \le \sigma_r} |\langle u-u_{f}, \mathbb W \rangle| \ge \frac{\sigma^2_r}{16\varepsilon} \Big) \\
&~~~~~~~~~~~\le 2\sum_{r=0}^R \exp\left\{-\frac{c D_0^2 2^{2r} \eta_\varepsilon^2}{\varepsilon^2}\right\} \le L \exp\left\{-\frac{c' D_0^2  \eta_\varepsilon^2}{\varepsilon^2}\right\}
\end{align*}
where $c'=c'(R), L>0$ are fixed constants. Since both $\hat u, u_{f}$ are contained in a ball of $H^{s+2}(\mathcal O)$ of radius $D'$, what precedes and the interpolation inequality (\ref{lionsm}) imply
\begin{equation}\label{halpha}
P_{f}^Y(\|\hat u - u_{f}\|_{H^\alpha} \ge c'' D_0 \eta^{(s+2-\alpha)/(s+2)}_\varepsilon) \le L \exp\left\{-\frac{c' D_0^2  \eta_\varepsilon^2}{\varepsilon^2}\right\}, 0 \le \alpha <s+2,
\end{equation}
for a fixed constant $c''>0$, which by the Sobolev imbedding, for $d/2+2 < \alpha<s+2$ also implies a convergence rate in $\|\cdot\|_{C^2(\mathcal O)}$-norm
\begin{equation}\label{c2}
P_{f}^Y(\|\hat u - u_{f}\|_{C^2} \ge c''' D_0 \eta^{\zeta}_\varepsilon) \le L \exp\left\{-\frac{c' D_0^2  \eta_\varepsilon^2}{\varepsilon^2}\right\}, \zeta>0, c'''>0.
\end{equation}
Now recall (\ref{inv}) and define
\begin{equation}\label{hatf}
\hat f =  \frac{\Delta \hat u}{2\hat u} 1_{A_\varepsilon}, ~~A_\varepsilon = \Big\{\inf_{x \in \mathcal O} \hat u(x) \ge c_0, \|\Delta \hat u\|_{\infty} \le D_2\Big\},~c_0>0, D_2<\infty,
\end{equation}
which defines a random variable in $C(\bar{\mathcal O})$. 

\begin{lemma} \label{tailexp}
For every $v >0$ there exist $c_0>0$ small enough and $D_2, D_3$ large enough such that for all $\bar D \ge D_3$ and $\eta_\varepsilon$ as in the theorem, we have, for all $f \in \mathcal F$,
$$P_f^Y \big(\{\|\hat f - f\|_{L^2} > \bar D \eta^{s/(s+2)}_\varepsilon \} \cup \{\|\hat f - f\|_{(H_0^2)^*} > \bar D \eta_\varepsilon\} \big) \le e^{-v(\eta_\varepsilon/\varepsilon)^2}.$$
\end{lemma}
\begin{proof}
If $B$ is the event whose probability we want to bound we can write $$P_f^Y (B) = P_f^Y (A_\varepsilon \cap B) + P_f^Y(A_\varepsilon^c \cap B)$$ and the second probability is no greater than $P_f^Y(A_\varepsilon^c)$, which is less than $e^{-v(\eta/\varepsilon)^2}/3$ for $c_0$ small and $D_2$ large enough, since $u_{f} \in C^2(\mathcal O)$ is bounded away from zero (see (\ref{jensen})), and using (\ref{c2}) for sufficiently large $D_0$. It remains to bound $P_f^Y (A_\varepsilon \cap B)$ by $2e^{-v(\eta/\varepsilon)^2}/3$. For $\hat f$ as in (\ref{hatf}) we see that (\ref{jensen}), (\ref{inv}) imply on $A_\varepsilon$ that 
\begin{align*}
 \|\hat f -f\|_{L^2} &\lesssim \|\hat u - u_f\|_{H^2}, \\
 \|\hat f-f\|_{(H^2_0)^*} &\lesssim \|[\hat u^{-1}-u^{-1}_f] \Delta \hat u_f\|_{L^2} + \|\Delta (\hat u - u_f)u_f^{-1}\|_{(H^2_0)^*} \\
&\lesssim \|\hat u-u_f\|_{L^2} + \|\Delta (\hat u - u_f)\|_{(H^2_0)^*}.
\end{align*}
Using density of $C^2(\mathcal O)$ in $H^2(\mathcal O)$, the last term can be further bounded by
\begin{align*}
\|\Delta (\hat u - u_f)\|_{(H_0^2)^*} &= \sup_{h \in C_0 \cap C^2:\|h\|_{H^2}\le 1} \left|\int_\mathcal O h \Delta (\hat u - u_f) \right|  = \sup_{h \in C_0 \cap C^2:\|h\|_{H^2}\le 1} \left|\int_\mathcal O \Delta h  (\hat u - u_f) \right| \\
& \le \sup_{\|h\|_{H^2} \le 1} \|\Delta h\|_{L^2}  \|\hat u-u_f\|_{L^2} \lesssim \|\hat u-u_f\|_{L^2}
\end{align*}
using Green's identity (\cite{GT98}, p.17) twice, that both $h$ and $\hat u -u_f$ have vanishing boundary traces $0$ and $g-g=0$, respectively, and the Cauchy-Schwarz inequality.

Summarising we conclude from (\ref{halpha}) that for every $v$ we can find $\bar D$ and $D_0$ large enough so that
\begin{equation}\label{tail}
P_f^Y(\|\hat f - f\|_{L^2} > \bar D \eta^{s/(s+2)}_\varepsilon ) \le P_f^Y(\|\hat u_f - u_f\|_{H^2} > d\bar D \eta^{s/(s+2)}_\varepsilon) \le  e^{-v(\eta_\varepsilon/\varepsilon)^2}/3
\end{equation}
\begin{equation}\label{tail2}
P_f^Y(\|\hat f - f\|_{(H_0^2)^*} > \bar D \eta_\varepsilon) \le P_f^Y(\|\hat u_f - u_f\|_{L^2} > d' \bar D \eta_\varepsilon) \le  e^{-v(\eta_\varepsilon/\varepsilon)^2}/3,
\end{equation}
which completes the proof of the lemma by the union bound.
\end{proof}

The usual plug in test (Proposition 6.2.2 in \cite{GN16}) defined by $$\Psi(Y) = 1\{\|\hat f - f_0\|_{L^2} \ge \bar D \eta^{s/(s+2)}_\varepsilon \}$$ for $\bar D$ a large enough constant, and likewise with $(H^2_0)^*, \eta_\varepsilon$ replacing $L^2(\mathcal O), \eta^{s/(s+2)}_\varepsilon$, gives
$$E^Y_{f_0}\Psi(Y) \le  P_{f_0}^Y(\|\hat f - f_0\|_{L^2} > \bar D  \eta^{s/(s+2)}_\varepsilon ) \le  e^{-(C'+4) (\eta_\varepsilon/\varepsilon)^2}$$ and also, for $M$ large enough and $f$ in the alternative (\ref{tests0})
\begin{align*}
E^Y_f(1-\Psi(Y)) &= P_f^Y(\|\hat f - f_0\|_{L^2} \le \bar D \eta^{s/(s+2)}_\varepsilon)  \le P_f^Y(\|f- f_0\|_{L^2} - \bar D \eta^{s/(s+2)}_\varepsilon \le \|\hat f -f\|_{L^2}) \\
& \le P_f^Y(\|\hat f- f\|_{L^2} \ge (M-\bar D)\eta^{s/(s+2)}_\varepsilon)  \le e^{-(C'+4) (\eta_\varepsilon/\varepsilon)^2},
\end{align*}
and likewise for the $(H^2_0)^*$-errors, completing the proof of the theorem.
\end{proof}

We finally turn to the verification of (\ref{sba}) for priors featuring in Theorem \ref{rates}.

\begin{proposition}\label{rateslate} Under the conditions of Theorem \ref{rates} and given $C>0$, we can choose $\bar L$ large enough such that the prior (\ref{prior}) with $J$ such that $2^J \simeq \varepsilon^{-2/(2s+4+d)}$ satisfies (\ref{sba})  for the choice $$\eta_\varepsilon = \bar L\sqrt {\log (1/\varepsilon)}  \cdot \varepsilon^{(2s+4)/(2s+4+d)}.$$ As a consequence $$\eta^{s/(s+2)}_\varepsilon \simeq \varepsilon^{\frac{2s}{2s+4+d}}\log^\gamma(1/\varepsilon),~~\gamma = s/(2s+4),$$ is the posterior contraction rate about $f_0$ in $L^2$, and in $(H^2_{0})^*$ we obtain the contraction rate
$\eta_\varepsilon$, both up to multiplicative constants.
\end{proposition}

\begin{proof}
The prior (\ref{prior}) defines a Borel probability measure on $V_J$ and thus on the class $\mathcal F$ from Theorem \ref{meta}. As a consequence of (\ref{b01}) the functions  $\phi, \phi_0$ and then also $e^{\varphi_0} \sum_{k=2}^\infty (\varphi-\varphi_0)^k/k!$ are uniformly bounded in $C^2(\mathcal O)\subset \mathcal C^2(\mathcal O)$. Thus, using also (\ref{mult}) we have $$\|f-f_0\|_{(H^2_0)^*} \le \sup_{\|h\|_{H^2(\mathcal O)} \le 1} \left| \int_\mathcal O h(e^\varphi- e^{\varphi_0}) \right| \le c_0  \|\varphi-\varphi_0\|_{(H^2)^*}.$$  
We can decompose $\|\phi-\phi_0\|_{(H^2)^*} \le \|\Pi_{V_J}(\phi-\phi_0)\|_{(H^2)^*} + \|\Pi_{V_J}(\phi_0)-\phi_0\|_{(H^2)^*}. $
For the second term we notice that by the compact support of $\phi_0$ and (\ref{wavprop}), for $l >J$ large enough necessarily $\int_\mathcal O \phi_0 \Phi^{bc}_{l,r} = 0$ for the boundary corrected wavelets and $\Pi_{V_J}(\phi_0)-\phi_0$ then consists of a sum of interior Daubechies wavelets $\Phi_{l,r}^\mathcal O = \Phi_{l,r}$ supported compactly in $\mathcal O$, in particular there exists $K \subset \mathcal O$ compact such that $\Pi_{V_J}(\phi_0)-\phi_0$ is supported within $K$ for $J$ large enough. Multiplying $h \in H^2(\mathcal O)$ by a function in $C^\infty_c(\mathcal O)$ that equals one on $K$ we obtain $\tilde h \in H^2(\mathbb R^d) \cap C_c(\mathcal O)$ such that $\tilde h =h$ on $K$ and $\|\tilde h\|_{H^2(\mathbb R^d)} \le c \|h\|_{H^2(\mathcal O)}$. Then using Parseval's identity and the Cauchy-Schwarz inequality, we can estimate the dual norm by
\begin{align} \label{intest}
\|\Pi_{V_J}(\phi_0)-\phi_0\|_{(H^2)^*}&=\sup_{\|h\|_{H^2(\mathcal O)} \le 1}\left|\int_\mathcal O h (\Pi_{V_J}(\phi_0)-\phi_0) \right| \\
&=\sup_{\|h\|_{H^2(\mathcal O)} \le 1}\left|\int_{\mathbb R^d} \tilde h (\Pi_{V_J}(\phi_0)-\phi_0) \right| \notag \\
& = \sup_{\|h\|_{H^2(\mathcal O)} \le 1}\left|\sum_{l>J,r} 2^{-2l} \langle \phi_0, \Phi_{l,r} \rangle_{L^2(\mathbb R^d)} 2^{2l}\langle \tilde h, \Phi_{l,r} \rangle_{L^2(\mathbb R^d)} \right| \notag \\
&\le c' \sup_{\|h\|_{H^2(\mathcal O)} \le 1} \|\tilde h\|_{H^2(\mathbb R^d)} \sqrt {\sum_{l>J,r} 2^{-4l}|\langle \Phi_{l,r}, \phi_0 \rangle|^{2}} \notag \\
&\le c'' 2^{-J(s+2)}J^{-2},
\end{align}
where we use (\ref{sobwav}) and (\ref{decay}). Now introducing `true' coefficients $$\phi_{0,l,r} = 2^{-l(s+d/2)}\bar l^{-2}u_{0,l,r}, |u_{0,l,r}| \le B,$$ and if $U$ is a uniform $U(-1,1)$ random variable and the $u_{l,r}\sim^{i.i.d.}U(-B,B)$, then using $\sum_{l \le J} N_l \le \bar c_0 2^{Jd}$ and what precedes we can lower bound, for $\bar L$ large enough,
\begin{align*}
\Pi (f: \|f-f_0\|_{(H^2_0)^*}<\eta_\varepsilon) &\ge \Pi (\phi: \|\phi-\phi_0\|_{(H^2)^*} <  \eta_\varepsilon/c_0) \\
&\ge \Pi \left(\phi: \|\Pi_{V_J}(\phi-\phi_0)\|_{(H^2)^*}< \eta_\varepsilon/c_0 - c'' J^{-2}2^{-J(s+2)} \right) \\
& \ge \Pi \left(\phi: \|\Pi_{V_J}(\phi-\phi_0)\|_{L^2}^2 < c_1 \eta_\varepsilon^2 \right) \\
& = \Pr \left(\sum_{l\le J,r } 2^{-l(2s+d)} \bar l^{-2} |u_{l,r}-u_{0,l,r}|^2 < c_1 \eta_\varepsilon^2  \right) \\
& \ge \Pr \left( \max_{l\le J, r} |u_{l,r}-u_{0,l,r}| < c_3 \eta_\varepsilon \right) \\
& = \prod_{l\le J} \prod_{r=1}^{N_l} \Pr (|U-(u_{0,l,r}/B)| < c_4 \eta_\varepsilon) \\
& \ge \Pr \left(c_4 \eta_\varepsilon/2 \right)^{\bar c_0 2^{Jd}} \ge e^{-c_5\log (c_4/\eta_\varepsilon) \varepsilon^{-2d/(2s+4+d)}} \ge e^{-C (\eta_\varepsilon/\varepsilon)^2},
\end{align*}
completing the proof.
\end{proof}
\begin{remark}\label{momrem}
If $\|\cdot\|$ is either the $L^2$ or $(H^2_0)^*$ norm and $r_\varepsilon$ the corresponding contraction rate, then the proof of Theorem \ref{rates} via Theorem \ref{ggv} actually implies that as $\varepsilon \to 0$, for some $c>0$, $$\Pi(\|f-f_0\| > r_\varepsilon|Y) = O_{P_{f_0}^Y} \big(e^{-c(\eta_\varepsilon/\varepsilon)^2}\big).$$
\end{remark}

\subsection{Proof of Theorem \ref{main}}

The proof is organised in a sequence of steps, and the main strategy is to prove the Bernstein-von Mises theorem via simultaneously controlling the Laplace transform of a collection of suitably regular linear functionals, an approach inspired by the papers of Isma\"el Castillo and co-authors \cite{CN13, CN14, C14, CR15, C17}. Of course a main challenge is to make this proof work outside of the basic LAN models considered in the references just mentioned, namely in the PDE setting considered here. However, even disregarding the different LAN structure, our proof needs to confront several new challenges when compared to the above papers: our prior has a boundary and hence using the `perturbation of the likelihood function' approach employed in \cite{RR12}, \cite{CR15}, \cite{CN14} needs some adjustments near the boundary. Moreover, our prior is supported in a $s$-regular H\"older wavelet ellipsoid that is asymptotically smaller than the ellipsoids considered in \cite{CN14}, which puts stronger constraints on the admissible directions one can choose when constructing perturbations (see also Remark \ref{lesson}).

\medskip

\textbf{Step I: Localisation of the posterior near $f_0$.}

\smallskip

We have from Proposition \ref{rateslate} that as $\varepsilon \to 0$,
$$\Pi\left(\|f-f_0\|_{L^2(\mathcal O)} > M \varepsilon^{2s/(2s+4+d)} \log^\gamma (1/\varepsilon) |Y\right) \to^{P_{f_0}^Y} 0$$ and 
$$\Pi\Big(\|f-f_0\|_{(H^2_0)^*} > M \varepsilon^{(2s+4)/(2s+4+d)} \sqrt{\log(1/\varepsilon)} |Y\Big) \to^{P_{f_0}^Y} 0.$$ 
From (\ref{b01}) we know $\|f\|_{C^s(\mathcal O)} \leq C'$ for any $f=e^\phi$ and so both prior and posterior are supported in a fixed ball of $C^s(\mathcal O) \subset H^s(\mathcal O)$. Since also $f_0 \in  H^s(\mathcal O)$ we have $\|f-f_0\|_{H^s} \le D$ for some fixed constant $D$ and recalling the standard interpolation inequality for Sobolev norms (see \cite{LM72}, Remark 1.9.1 on p.44)
\begin{equation} \label{lionsm}
\|f\|_{H^\alpha(\mathcal O)} \lesssim \|f\|_{H^s(\mathcal O)}^{\alpha/s} \|f\|_{L^2(\mathcal O)}^{(s-\alpha)/s},~~0<\alpha<s,
\end{equation}
we also obtain for all $M$ large enough and $\bar \gamma=\bar\gamma (\gamma, s,\alpha)>0$ that, as $\varepsilon \to 0$,
\begin{equation} \label{fullrates}
\Pi\left(\|f-f_0\|_{H^{\alpha}(\mathcal O)} > M \varepsilon^{2(s-\alpha)/(2s+4+d)} \log^{\bar \gamma} (1/\varepsilon) |Y\right) \to^{P_{f_0}^Y} 0,~0< \alpha < s.
\end{equation}
Since $s>2+d/2$ we can apply the Sobolev imbedding theorem to deduce
\begin{equation} \label{suprates}
\Pi\left(\|f-f_0\|_{C^\beta(\mathcal O)} > M \varepsilon^{(2s-2\beta-d)/(2s+4+d)} \log^{\bar \gamma} (1/\varepsilon) |Y\right) \to^{P_{f_0}^Y} 0,~0 < \beta < s-d/2.
\end{equation}
which also implies a contraction rate in the uniform norm $\|\cdot\|_\infty$.

Since all $f \in \supp \Pi$ and $f_0$ are bounded and bounded away from zero on $\mathcal O$, the $\|\cdot\|_{L^2}$ and $\|\cdot\|_\infty$ contraction rates extend to $\varphi =\log f$ around $\varphi_0=\log f_0$, by simply using the estimate $\|\varphi-\varphi_0\| \lesssim \|f-f_0\|$ for these norms. The $(H^2_0)^*$-rates carry over too: To see this, notice that on the events in (\ref{suprates}) with $\beta=2<s-d/2$ we have $\|f-f_0\|_{C^2} \to 0$ and then also $\|(f-f_0)/f_0\|_{C^2} \to 0$, so that using the Taylor expansion of the logarithm we also have
\begin{align} \label{logs}
\|\varphi - \varphi_0\|_{(H^2_0)^*} &= \sup_{g \in C_0(\mathcal O): \|g\|_{H^2} \le 1} \left| \int_\mathcal O g \log (f/f_0) \right| \notag \\
& =\sup_{g \in C_0(\mathcal O): \|g\|_{H^2} \le 1} \left|\int_\mathcal O g (f-f_0) f_0^{-1} \sum_k \frac{(-1)^k}{k} \left(\frac{f-f_0}{f_0}\right)^{k-1} \right|  \lesssim \|f-f_0\|_{(H^2_0)^*} 
\end{align} 
since for $\|(f-f_0)/f_0\|_{C^2}<1/2$ the series $w=f_0^{-1}\sum_k (-1)^k k^{-1} ((f-f_0)/f_0)^{k-1}$ converges absolutely in $C^2(\mathcal O)$ and hence, after dividing and multiplying by $\|gw\|_{H^2} \le c\|g\|_{H^2}\|w\|_{C^2}$ via (\ref{mult}), the result follows noting also that $g \in C_0(\mathcal O)$ implies $gw \in C_0(\mathcal O)$.

\smallskip

For a fixed constant $M$ to be chosen, let us now define the event
\begin{equation} \label{dn}
D_\varepsilon^{M} = \{f =e^\varphi \in \supp \Pi: \|\varphi-\varphi_0\|^{(i)} \le M r(i, \varepsilon),~i=1,2,3\},
\end{equation}
where, for some $\eta=\eta(s,\gamma)>0$ and $0<\xi/2<s-2-d/2$ small enough,
 $$\|\cdot\|^{(1)} =\|\cdot\|_{L^2}, ~ r(1, \varepsilon) = \varepsilon^{2s/(2s+4+d)} \log^\eta(1/\varepsilon),$$ $$\|\cdot\|^{(2)} = \|\cdot\|_{(H^2_0)^*}, ~ r(2, \varepsilon) = \varepsilon^{(2s+4)/(2s+4+d)} \log^\eta(1/\varepsilon),$$ $$ \|\cdot\|^{(3)} =\|\cdot\|_\infty,~ r(3,\varepsilon) = \varepsilon^{(2s-d-\xi)/(2s+4+d)} \log^\eta(1/\varepsilon).$$ 
By increasing the constant $M$ to $\bar M$ we also obtain 
\begin{equation} \label{frate}
D^{M}_\varepsilon \subset \{f \in \supp \Pi: \|f-f_0\|^{(i)} \le \bar M r(i, \varepsilon),~i=1,2,3\},
\end{equation}
which is obvious for $i=1,3$ and follows also for $i=2$ by just applying the argument in (\ref{logs}) with the Taylor series expansion of the exponential (instead of the logarithm) function, using that $\|(f-f_0)/f_0\|_{C^2}$ is bounded by a fixed constant. 
 
By Theorem \ref{rates} and what precedes we have $$\Pi((D^M_\varepsilon)^c|Y)  \to 0~\text{in}~P_{f_0}^Y~\text{probability},$$ for all $M$ large enough. In particular if $W$ is the constant in inequality (\ref{iota}) below, then there exists $M_0$ such that the last limit holds for all $M \ge M_0/8W$, and given such $M_0$ we fix any value $M \ge M_0$ in what follows and just write $D_\varepsilon$ for $D_\varepsilon^M$ when no confusion may arise. Then, if $\Pi^{D_\varepsilon}(\cdot|Y)$ is the posterior distribution arising from the prior $\Pi$ restricted to the set $D_\varepsilon$, that is, from prior $\Pi(\cdot \cap D_\varepsilon)/\Pi(D_\varepsilon)$, standard arguments imply that
\begin{equation} \label{tvlim}
\sup_{B}|\Pi(B|Y)-\Pi^{D_\varepsilon}(B|Y)| \le 2\Pi(D_\varepsilon^c|Y) \to^{P_{f_0}^Y} 0.
\end{equation}
where the supremum extends over all measurable sets $B$ in $\supp \Pi$. 

\bigskip

\textbf{Step II: Construction of the perturbation $f_\tau$.}  

\smallskip

Fix $\gamma>0$. For arbitrary $\psi \in C^4_K(\mathcal O)$ such that $\| \psi\|_{C^{2+d/2+\gamma}(\mathcal O)}\le 1$, we now construct a suitable perturbation $f_\tau$ of $f \in \supp \Pi$ -- an asymptotic expansion of the log-likelihood ratio $\ell(f)-\ell(f_\tau)$ obtained in the next step will be a key element of our proof. For such $\psi$ define 
\begin{equation}
\tilde \psi=-\frac{\tilde \Psi}{f_0}
\end{equation}
with $\tilde \Psi$ as in (\ref{tildp}). From the representation (\ref{tildp2}), Proposition \ref{sbaest}A), the definition of the $C^{\alpha}$-norms and the hypotheses on $f_0, g$ which imply that $u_{f_0}^{-1} \in C^{s+2}(\mathcal O), f_0^{-1} \in C^s(\mathcal O),$ we deduce  that $\tilde \psi \in C_K(\mathcal O)$, and when $-2+d/2 \ge 0$ then also $\|\tilde \psi\|_{C^{-2+d/2+\gamma}(\mathcal O)} \le c \|\psi\|_{C^{2+d/2 +\gamma}(\mathcal O)}$ for some finite constant $c>0$. Thus from Proposition \ref{hzimb} and (\ref{hzbc})
\begin{equation} \label{tildpdec}
\|\tilde \psi\|_{\mathcal C^{-2+d/2+ \gamma, W}(\mathcal O)} \le C \|\psi\|_{C^{2+d/2 + \gamma}(\mathcal O)} ~\Rightarrow |\langle \tilde \psi, \Phi^\mathcal O_{l,r} \rangle_{L^2}| \le C 2^{-l(d-2+\gamma)}.
\end{equation}
The last estimate is true also in the more delicate case $-2+d/2 <0$, as follows from (\ref{tildp2}) and the inequalities (\ref{regul}), (\ref{multc}) in Proposition \ref{pain} below.

\smallskip

For $f \in \supp(\Pi)$ the perturbation will be
\begin{equation} \label{taud}
f_\tau = f  \exp\{\tau\} = e^{\varphi +\tau}
\end{equation}
(with a slight abuse of notation when $\tau =0$) where $\tau$ is defined as follows: Let $J$ be the cut-off parameter of the prior, and for $L \in \mathbb N, L\le J,$ define $W_L \subset L^2(\mathcal O)$ to be the linear span of those wavelet basis functions $(\Phi_{l,r}^\mathcal O: l \le L)$ for which \textit{either} $\langle \tilde \psi,\Phi^\mathcal O_{l,r} \rangle \neq 0$ \textit{or} $\langle \Phi^\mathcal O_{l,r}, \varphi_0 \rangle_{L^2} \neq 0$; in other words $W_L$ is the subspace of $V_L$ where $\Pi_{V_L}(\phi_0), \Pi_{V_L}(\tilde \psi)$ are supported.  For $\Pi_{W_L}$ the corresponding $L^2$-projection operator set
\begin{equation} \label{tau!}
\tau \equiv \tau(\phi, \psi, t, L)= \Pi_{W_L}[t\varepsilon \tilde \psi + \delta_\varepsilon (\varphi_0-\varphi)] = t \varepsilon \Pi_{W_L}[\tilde \psi] + \delta_\varepsilon \Pi_{W_L} [\varphi_0 - \varphi],~t \in \mathbb R,
\end{equation}
with, for $\gamma_1>0$ small enough chosen below (depending on $\gamma$),
\begin{equation} \label{dol}
\delta_\varepsilon = \varepsilon^{\frac{2d+\gamma_1}{2s+4+d}}.
\end{equation}
[The second summand in the perturbation (\ref{tau!}) is different from the perturbations used in previous proofs of this kind \cite{RR12, CN14, CR15}. It will be useful to deal with the boundary of the support of the prior in Step IV below, and was used in a related context of nonparametric maximum likelihood estimation over parameter spaces with boundaries in \cite{N07}.] 

The projection $\Pi_{W_L}$ is a bounded operator for all the norms introduced after (\ref{dn}), that is,
\begin{equation}\label{iota}
\|\Pi_{W_L}(\varphi-\varphi_0)\|^{(i)} \le W \|\varphi-\varphi_0\|^{(i)},~i=1,2,3,~\forall \varphi \in V_J,
\end{equation}
for some fixed finite constant $W \ge 1$. This is clear for $i=1$ and also follows easily also for $i=2,3$, proceeding as in the corresponding proofs for wavelet bases of $L^2(\mathbb R^d)$ (\cite{M92, D92, GN16}) and making use of  (\ref{wavprop}) and the compact support of $\varphi_0, \tilde \psi$. [In case $i=2$, one initially proves that $\Pi_{W_L}$ defines a bounded operator on $H^2(\mathcal O)$.]

Using Parseval's identity and (\ref{tildpdec}) we then have
\begin{equation} \label{regtil}
\|\Pi_{W_L} \tilde \psi\|_{L^2(\mathcal O)} \le c \sqrt {\sum_{l\le J}  \sum_r \langle \tilde \psi, \Phi^\mathcal O_{l,r} \rangle_{L^2}^2} \leq C \sqrt{\sum_{l<J} 2^{l(4-d-2\gamma)}} \lesssim 2^{J(2-d/2)} =o(\varepsilon^{-1})
\end{equation}
for all $\gamma>0$ and likewise, using (\ref{wavprop}) and (\ref{tildpdec}), for every $i, |i|=\beta, 0 \le \beta<s+d, \gamma>0$,
\begin{align} \label{regtil1}
\|D^i \Pi_{W_L}\tilde \psi\|_\infty &\leq \sum_{l\le J} 2^{l(\beta+d/2)} \sup_x \sum_r |\langle \tilde \psi, \Phi^\mathcal O_{l,r} \rangle_{L^2}| |(D^i\Phi_{0,r}^\mathcal O)(2^lx)| \notag\\
& \le c\sum_{l\le J} 2^{l(2+\beta-d/2-\gamma)}  \lesssim \varepsilon^{(-4-2\beta+d)/(2s+4+d)} =o(\varepsilon^{-1}),
\end{align}
in particular we have  $\varepsilon\|\Pi_{W_L}\tilde \psi\|_{C^\beta(\mathcal O)} = o(1)$.
Now since $\varphi, \varphi_0$ are all bounded functions  we conclude that $\|\tau\|_\infty \to 0$ as $\varepsilon \to 0$, and thus also 
\begin{equation} \label{expdev}
f_\tau \ge c'>0,~~|f_\tau - f| = |f||1-e^\tau| \le c''|\tau| \to_{\varepsilon \to 0} 0, ~c'>0, c''>0.
\end{equation}
We have likewise that 
\begin{align} \label{negst}
\notag \|f_\tau - f\|_{(H^2_0)^*} &= \sup_{g \in C_0(\mathcal O): \|g\|_{H^2}\le 1} \left|\int_\mathcal O g f (1-e^{\tau}) \right| \\
&=  \sup_{g \in C_0(\mathcal O):\|g\|_{H^2}\le 1} \left|\int_\mathcal O \tau \Big(g f \sum_{k=1}^\infty (\tau^{k-1})/k! \Big) \right| \lesssim \|\tau\|_{(H^2_0)^*}
\end{align}
since the above and the assumptions on $f, f_0,\varphi, \varphi_0$ imply that $f \sum_{k=1}^\infty (\tau^{k-1})/k!$ is bounded in $C^2(\mathcal O)$ and hence by (\ref{mult}) the functions $$g f \sum_{k=1}^\infty \frac{\tau^{k-1}}{k!} \in C_0(\mathcal O), ~ \|g\|_{H^2}\le 1$$  are bounded in $H^2(\mathcal O).$

\bigskip

\textbf{Step III: Expanding the likelihood in the Laplace transform.}
\begin{theorem}\label{work}
For $\psi \in C_K^4(\mathcal O)$ define $\tilde \Psi$ as in (\ref{tildp}). Let $t>0$ and $f_\tau$ as in (\ref{taud}), (\ref{tau!}) for any $\gamma_1>0$. Suppose $L=L_\varepsilon \uparrow \infty$ as $\varepsilon \to 0$ is a sequence of integers such that either 

i) $L=J$ or that

ii) $L \le J$ but $\|\Pi_{W_L}(\tilde \Psi)-\tilde \Psi\|_{(H^2_0)^*} \le c\varepsilon^{\bar \gamma/(2s+4+d)}$ for some $\bar \gamma>d$ and some constant $c>0$.

If $\Pi^{D_\varepsilon}(\cdot|Y)$ is the posterior distribution arising from the prior $\Pi$ restricted to the set $D_\varepsilon$ from (\ref{dn}) with $M$ chosen as before (\ref{tvlim}), and if $DG_{f_0}$ is the score operator from Lemma \ref{diff}, then for all $t \in \mathbb R$,
\begin{equation} \label{lanfinal}
E^{\Pi^{D_\varepsilon}} \left[e^{\frac{t}{\varepsilon} \langle f -f_0, \psi \rangle_{L^2}} |Y\right] =  e^{-t\langle DG_{f_0}[\tilde \Psi], \mathbb W \rangle_{L^2}  + \frac{t^2}{2}\|DG_{f_0}[\tilde \Psi]\|_{L^2}^2} \times \frac{\int_{D_\varepsilon}  e^{\ell(f_\tau)} d\Pi(f)}{\int_{D_\varepsilon} e^{\ell(f)} d\Pi(f)} \times e^{R_\varepsilon}
\end{equation}
where $R_\varepsilon = o_{P_{f_0}^Y}(1)$ uniformly in $|t| \le T$ for any $T$ and  in $\psi \in \mathcal C(b) \equiv \{\psi \in C^4_K(\mathcal O): \|\psi\|_{C^{2+d/2+\gamma}(\mathcal O)}\le b\}$ for any $b>0$, every $\gamma>0$, and every fixed compact subset $K \subset \mathcal O$.
\end{theorem}
\begin{proof}
We have from (\ref{posterior}) and recalling $\ell(f)=\log p_f(Y)$ that for all $t \in \mathbb R$,
\begin{equation} \label{laplace}
E^{\Pi^{D_\varepsilon}} \left[e^{\frac{t}{\varepsilon} \langle f -f_0, \psi \rangle_{L^2}} |Y\right] =  \frac{\int_{D_\varepsilon} e^{(t/\varepsilon) \langle f - f_0, \psi \rangle_{L^2} +\ell(f)-\ell (f_\tau) + \ell(f_\tau)} d\Pi(f)}{\int_{D_\varepsilon} e^{\ell(f)} d\Pi(f)}
\end{equation}
The main technical work to follow now is to obtain a sharp asymptotic expansion of $\ell(f)-\ell (f_\tau)$ that is uniform in $f,\psi$. To do this, we need to start with two linearisation steps, the first takes care of the non-linearity of the forward operator $G$, and the second of the exponential nature of our perturbation $f_\tau$, ultimately leading to (\ref{lanlin}) below.

\smallskip

\textit{A: Linearisation of the $G$ operator:}

Using Lemma \ref{nonlan} and Proposition \ref{diff} to expand both $G(f)=u_f$ and $G(f_\tau)=u_{f_\tau}$ about $G(f_0)=u_{f_0}$ -- from (\ref{regtil1}) and definition of the prior we know that $f, f_\tau, f_0$ are all bounded in $C^s(\bar {\mathcal O})$ and bounded away from zero -- we obtain the approximation, for $f \in D_\varepsilon$,
\begin{align} \label{lan1}
\ell(f)-\ell (f_\tau) &= -\frac{1}{2\varepsilon^2} \left(\|u_f-u_{f_0}\|_{L^2}^2 - \|u_{f_\tau}-u_{f_0}\|_{L^2}^2 \right) + \frac{1}{\varepsilon} \langle u_f-u_{f_\tau}, \mathbb W \rangle_{L^2} \notag \\
&= -\frac{1}{2\varepsilon^2}\left(\|DG_{f_0}[f-f_0]\|_{L^2}^2 - \|DG_{f_0}[f_\tau-f_0]\|_{L^2}^2 \right) +  \frac{1}{\varepsilon}\langle DG_{f_0}[f-f_\tau], \mathbb W \rangle_{L^2} \\
&~~~+ R(f, \psi) + N(f, \psi) \notag
\end{align}
where for some $c>0$,
$$|R(f, \psi)| \le \frac{c}{\varepsilon^2}  \|f-f_0\|_{(d)}^2 \|f-f_0\|^2_{(H^2_0)^*} + \frac{c}{\varepsilon^2} \|f_\tau-f\|^2_{(d)} \|f_\tau-f\|^2_{(H^2_0)^*} $$ and where the stochastic remainder term is given by
$$N(f) = \varepsilon^{-1} \langle w(f, \psi), \mathbb W \rangle_{L^2}$$ with $w(f, \psi)$ described below.

\medskip

For the first term in the bound for $R(f,\psi)$ let us first assume $d<4$: then on $D_\varepsilon$ (cf.~(\ref{frate})) the inequality
$$\frac{c\bar M^4}{\varepsilon^2} r^2(2,\varepsilon) r^2(1,\varepsilon)   \lesssim \varepsilon^{-(4s+8+2d)/(2s+4+d)} \varepsilon^{(8s+8)/(2s+4+d)} \log^{4\eta} (1/\varepsilon)  =  o(1) $$ holds since $s>d/2$. Also from (\ref{expdev}), (\ref{negst}) we know $$\|f_\tau-f\|^2_{L^2} \|f_\tau-f\|^2_{(H^2_0)^*} \lesssim \|\tau\|^2_{L^2} \|\tau\|^2_{(H^2_0)^*}$$ which is $o(\varepsilon^2)$ uniformly in $f \in D_\varepsilon, \psi \in \mathcal C(b)$  since
\begin{equation} \label{delta1}
\delta^2_\varepsilon \|\varphi-\varphi_0\|_{L^2} \|\varphi-\varphi_0\|_{(H^2_0)^*} = o(\varepsilon),
\end{equation} 
since by (\ref{regtil})
\begin{equation}
\|\Pi_{W_L}[\tilde \psi]\|_{L^2}\|\Pi_{W_L}[\tilde \psi]\|_{(H^2_0)^*} \lesssim \|\Pi_{W_L}[\tilde \psi]\|_{L^2}^2  =o(\varepsilon^{-2}),
\end{equation}
and bounding the `cross term' similarly. When $d \ge 4$ then the $\|\cdot\|_{L^2}$-norms have to be replaced by $\|\cdot\|_\infty$-norms in the above estimates, resulting in slightly worse convergence rates on $D_\varepsilon$ and the requirement $s>d$ in place of $s>d/2$. 

\smallskip

For the stochastic term $N(f,\psi)$ we notice that $w(f,\psi)=w_{f, \psi} - w'_{f, \psi}$ in
\begin{align*}
G(f)-G(f_\tau) &= DG_f[f-f_\tau] - w'_{f,\psi} \\
&= DG_{f_0}[f-f_\tau] + (DG_{f}-DG_{f_0})(f-f_\tau) - w_{f, \psi}' \\
& \equiv DG_{f_0}[f-f_\tau] - w'_{f, \psi} + w_{f, \psi} 
\end{align*}
where $w_{f,\psi}'$ solves (just as in (\ref{sol})) the inhomogeneous Schr\"odinger equation
$$(\Delta/2) w - f_\tau w = -(f-f_\tau)V_{f}[u_f(f_\tau-f)]~~\text{ on }~\mathcal O$$ s.t. $w=0$ on $\partial \mathcal O$, and where $V_f$ denotes the inverse Schr\"odinger operator from Section \ref{sop}. Recalling (\ref{dn}) and (\ref{regtil}), both $\{\varphi: e^\varphi \in D_\varepsilon\}$ and $\{\varepsilon \Pi_{W_L}[\tilde \psi]: \psi \in \mathcal C(b)\}$ are bounded subsets of the linear space $(V_J, \|\cdot\|_{L^2})$ which is isomorphic to a $c_02^{Jd}$-dimensional Euclidean space (by Parseval's identity). Any such set can be covered by $(A/\eta)^{c_0 2^{Jd}}$ balls of $L^2$-diameter at most $\eta$ for all $0<\eta<A$ and some $A<\infty$ (Proposition 4.3.34 in \cite{GN16}). Moreover, using Lemma \ref{pest} and Proposition \ref{sbaest}  one shows
\begin{align*}
\|w_{f,\psi}-w_{f',\psi}\|_{L^2} &= \|V_{f_0}[u_{f_0}(f-f_\tau)] - V_f[u_f(f-f_\tau)] -(V_{f_0}[u_{f_0}(f'-f'_\tau)] - V_{f'}[u_{f'}(f'-f'_\tau)]) \|_{L^2}  \\
& \lesssim \|f-f'\|_{L^2} \lesssim \|\phi-\phi'\|_{L^2}
\end{align*}
and likewise that $\|w_{f', \psi} - w_{f', \psi'}\|_{L^2} \lesssim \|\Pi_{W_L}[\tilde \psi]-\Pi_{W_L}[\tilde \psi']\|_{L^2}$. Thus the set $\mathcal W =\{w_{f,\psi}: f \in D_\varepsilon, \psi \in \mathcal C(b)\}$ can be covered in $L^2$-distance by at most $(A'/\eta)^{c_0' 2^{Jd}}$ balls of radius $\eta$, and if $\sigma \ge \sup_{w \in \mathcal W}\|w\|_{L^2}$ then (\ref{dudmet2}) below applies to the Gaussian process $\{\langle \mathbb W, w\rangle_{L^2(\mathcal O)} : w \in \mathcal W\}$ which is sub-Gaussian for the metric $\|w_{f,\psi}-w_{f',\psi'}\|_{L^2}$. Conclude that
\begin{equation}\label{dudmodel}
E \sup_{f \in D_\varepsilon, \psi \in \mathcal C(b)} |\langle w_{f,\psi}, \mathbb W \rangle_{L^2(\mathcal O)}| \le C\left(\sigma + \int_0^\sigma 2^{Jd/2} \sqrt{2\log (A' /\eta)} d\eta\right).
\end{equation}
We see from Proposition \ref{sbaest} and Lemma \ref{pest} that 
\begin{equation*}
\|w_{f,\psi}\|_{L^2}  =  \|V_{f_0}[u_{f_0}(f-f_\tau)] - V_{f}[u_f(f-f_\tau)]\|_{L^2}  \lesssim \|f-f_0\|_{L^2} \|\tau\|_\infty 
\end{equation*}
a bound that will be seen to converge to zero at a polynomial rate in $\varepsilon$, uniformly in $w 
\in \mathcal W$.  By basic calculus (cf.~p.190 in \cite{GN16}) we can thus bound (\ref{dudmodel}) up to constants by  $\sup_{f \in D_\varepsilon, \psi \in \mathcal C(b)} 2^{Jd/2} \|f-f_0\|_{L^2} \|\tau\|_{\infty} \sqrt{\log(1/\varepsilon)}$. Next we turn to suprema of the Gaussian process $(\langle \mathbb W, w'_{f,\psi} \rangle_{L^2}: f \in D_\varepsilon, \psi \in \mathcal C(b)\}$. Using again the results from Section \ref{sop} we have
\begin{align*}
\|w'_{f,\psi}-w_{f', \psi}'\|_{L^2} &= \|V_{f_\tau}[(f-f_\tau)V_{f}[u_f(f_\tau-f)]] - V_{f'_\tau}[(f'-f'_\tau)V_{f'}[u_{f'}(f'_\tau-f')]]\|_{L^2} \\
& \lesssim \|f-f'\|_{L^2} \lesssim \|\phi-\phi'\|_{L^2},
\end{align*}
and $\|w'_{f',\psi}-w_{f', \psi'}'\|_{L^2}\lesssim \|\Pi_{W_L}(\tilde \psi- \tilde\psi')\|_{L^2}$ as well as
$$\|w_f'\|_{L^2} = \|V_{f_\tau}[(f-f_\tau)V_{f}[(f_\tau-f)u_f]]\|_{L^2} \lesssim \|\tau\|_\infty \|f_\tau-f\|_{L^2} \equiv \sigma,$$ so that repeating the arguments from above we obtain
\begin{equation*}
E \sup_{f \in D_\varepsilon, \psi \in \mathcal C(b)}|N(f,\psi)| \lesssim \sup_{f \in D_\varepsilon, \psi \in \mathcal C(b)} \varepsilon^{-1} 2^{Jd/2} (\|f-f_0\|_{L^2}+\|f_\tau-f\|_{L^2}) \|\tau\|_{\infty} \sqrt{\log(1/\varepsilon)}.
\end{equation*}
Using the definition of $\tau$ and that $s>2+d/2+\xi/2$,
\begin{align*}
&\varepsilon^{-1} 2^{Jd/2} \sqrt {\log (1/\varepsilon)} \delta_\varepsilon \|\varphi-\varphi_0\|_\infty \|f-f_0\|_{L^2} \\
&\lesssim \varepsilon^{(-2s-4-d-d+2d+2s-d-\xi+2s+\gamma_1)/(2s+4+d)} \log^{2\eta+1/2}(1/\varepsilon)=o(1),
\end{align*}
the same bound with $\|f-f_0\|_{L^2}$ replaced by $\delta_\varepsilon \|\phi-\phi_0\|_{L^2}$, and using (\ref{regtil1}) with $\beta=0$ in
$$ 2^{Jd/2} \sqrt {\log (1/\varepsilon)} \|\Pi_{W_L}\tilde \psi\|_\infty \|f-f_0\|_{L^2} \lesssim \varepsilon^{(-d-4+d+2s)/(2s+4+d)}\log^{\eta+1/2}(1/\varepsilon)=o(1),$$ 
and in
$$\varepsilon 2^{Jd/2} \sqrt {\log (1/\varepsilon)} \|\Pi_{W_L}\tilde \psi\|_\infty\|\Pi_{W_L}\tilde \psi\|_{L^2} = o(1).$$ All the preceding bounds being uniform in $f,\psi$ we obtain $\sup_{f \in D_\varepsilon, \psi \in \mathcal C(b)}|N(f,\psi)|=o_{P_{f_0}^Y}(1)$.

\smallskip

Using the above estimates for $R(f,\psi), N(f,\psi)$ the LAN expansion (\ref{lan1}) now simplifies to
\begin{align}\label{lan11}
&\ell(f)-\ell (f_\tau) \notag \\
&= \frac{1}{\varepsilon}\langle DG_{f_0}[f-f_\tau], \mathbb W \rangle_{L^2}  + \frac{1}{2\varepsilon^2} \|DG_{f_0}[f-f_\tau]\|_{L^2}^2 + \frac{1}{\varepsilon^2} \langle DG_{f_0}[f-f_0], DG_{f_0}(f-f_\tau) \rangle_{L^2} + Z_n \notag \\
&\equiv I + II + III +Z_\varepsilon 
\end{align}
where $Z_\varepsilon=o_{P_{f_0}^Y}(1)$ uniformly in $f \in D_\varepsilon, \psi \in \mathcal C(b)$, for any $\gamma>0$.

\medskip

\textit{B: Linearisation of the exponential perturbation }

We now linearise $f_\tau -f = f(1-e^\tau)$ in $\tau$. Since $\|\tau\|_\infty \to 0$ as $\varepsilon \to 0$ we assume $\|\tau\|_\infty <1/2$ in what follows. We have
$$DG_{f_0}[f-f_\tau] - DG_{f_0}[f_0 \tau] =DG_{f_0}[f-f_\tau-f \tau -(f_0-f)\tau].$$  Thus from  (\ref{smt}) and the triangle inequality we deduce
\begin{align*}
\|DG_{f_0}[f-f_\tau] - DG_{f_0}[f_0 \tau]\|^2_{L^2} &\le c_2\|f-f_\tau - f \tau\|^2_{(H^2_0)^*} + c_2\|(f_0-f)\tau\|^2_{(H^2_0)^*} \\
& \lesssim \|f\|^2_\infty \|\tau^2\|^2_{L^2} + \|(f_0-f)\tau \|^2_{L^2} \\
& \lesssim \|\tau\|_\infty^2( \|\tau\|_{L^2}^2 + \|f-f_0\|^2_{L^2} )
\end{align*}
which we shall use in the following estimates:
First, for term $II$ in (\ref{lan11}) we have
$$\left|II-\frac{1}{2\varepsilon^2} \|DG_{f_0}[f_0 \tau]\|_{L^2}^2\right| \le \frac{c'}{\varepsilon^2}  \|\tau\|_\infty^2 (\|\tau\|_{L^2}^2 + \|f-f_0\|^2_{L^2} ) =o(1)$$
using the definition of $\tau$ in (\ref{tau!}), (\ref{regtil}), (\ref{regtil1}) and $s>2$ in the bounds 
$$\|\Pi_{W_L}\tilde \psi \|^2_\infty  \|f-f_0\|_{L^2}^2 \lesssim \varepsilon^{(-8+2d+4s)/(2s+4+d)} \log^{2\eta}(1/\varepsilon)=o(1),$$ $$ \|\Pi_{W_L}\tilde \psi \|^2_{\infty} \varepsilon^2 \|\Pi_{W_L}\tilde \psi \|^2_{L^2} \lesssim \varepsilon^2 \varepsilon^{(-16+4d)/(2s+4+d)}= o(1),$$ $$\varepsilon^{-2}\delta_\varepsilon\| \Pi_{W_L}(\varphi-\varphi_0)\|_{\infty}^2\|f-f_0\|_{L^2}^2\lesssim \varepsilon^{-2} \varepsilon^{(4d+4s-2d+4s-2\xi)/(2s+4+d)} \log^{4\eta}(1/\varepsilon) =o(1),$$ and noting that the term involving $\delta_\varepsilon (\varphi-\varphi_0)$ in place of $f-f_0$ is of even smaller order of magnitude on $D_\varepsilon$. Similarly
\begin{align*}
&\left| III - \frac{1}{\varepsilon^2} \langle DG_{f_0}[f-f_0], DG_{f_0}(f_0\tau) \rangle_{L^2} \right| \lesssim \varepsilon^{-2} \|f-f_0\|_{(H^2_0)^*} \|\tau\|_\infty (\|\tau\|_{L^2}+\|f-f_0\|_{L^2}) \lesssim \\
& \varepsilon^{-2} \|f-f_0\|_{(H^2_0)^*} \big( \delta_\varepsilon \|\varphi-\varphi_0\|_\infty + \varepsilon \|\Pi_{W_L}\tilde \psi\|_\infty \big) \times \big(\|f-f_0\|_{L^2} + \delta_\varepsilon \|\varphi-\varphi_0\|_{L^2} + \varepsilon \|\Pi_{W_L}\tilde \psi\|_{L^2} \big)
\end{align*}
is $o(1)$ on $D_\varepsilon$ using $s>2$ and (\ref{regtil1}) in
$$\varepsilon^{-1} \|f-f_0\|_{(H^2_0)^*} \| \Pi_{W_L}\tilde \psi\|_{\infty} \|f-f_0\|_{L^2} \lesssim \varepsilon^{(-2s-4-d+4s+4+d-4)/(2s+4+d)} \log^{2\eta}(1/\varepsilon) =o(1),$$
$$\|f-f_0\|_{(H^2_0)^*}\|\Pi_{W_L}\tilde \psi \|_{\infty} \|\Pi_{W_L}\tilde \psi \|_{L^2}\lesssim \varepsilon^{(2s+4 -8+2d)/(2s+4+d)} \log^{\eta}(1/\varepsilon) = o(1),$$
and $s>2+d/2+\xi/2$ in 
\begin{align*}
&\varepsilon^{-2} \delta_\varepsilon \|f-f_0\|_{(H^2_0)^*} \|\varphi-\varphi_0\|_{\infty}\|f-f_0\|_{L^2}  \lesssim \varepsilon^{(-4s-8-2d +2d + 2s+4+2s-d -\xi + 2s)/(2s+4+d)} \log^{3\eta}(1/\varepsilon) \\
&= \varepsilon^{(2s-4-d-\xi)/(2s+4+d)} \log^{3\eta}(1/\varepsilon) =o(1),
\end{align*}
that $\delta_\varepsilon \|\varphi-\varphi_0\|_{L^2}$ is of smaller order of magnitude than $\|f-f_0\|_{L^2}$, and that a similar bound holds if in the last display $\|f-f_0\|_{L^2}$ is replaced by $\|\Pi_{W_L}\tilde \psi \|_{L^2}$. 

Finally for the stochastic term $I$ we have
$$\left|I- \frac{1}{\varepsilon}\langle DG_{f_0}[f_0 \tau], \mathbb W \rangle_{L^2}\right| \lesssim \varepsilon^{-1} |\langle DG_{f_0}[f-f_\tau - f \tau], \mathbb W \rangle | + \varepsilon^{-1} |\langle DG_{f_0}[(f_0-f)\tau], \mathbb W \rangle |$$ and we apply arguments as above (\ref{dudmodel}) to the Gaussian processes $\{\langle W, w_{f,\psi} \rangle_{L^2}: f = e^\varphi \in D_\varepsilon, \psi \in \mathcal C(b)\}$ now with $w_{f,\psi}$ equal to either $DG_{f_0}[f-f_\tau - f \tau]$ or $DG_{f_0}[(f_0-f)\tau]$: In both cases by $L^2$-continuity of $DG_{f_0}$ (see (\ref{smt})) and the properties of $f_\tau, \tau$ we can bound $$\|w_{f,\psi}-w_{f',\psi'}\|_{L^2} \lesssim \|\varphi-\varphi'\|_{L^2} + \|\Pi_{W_L}(\tilde\psi - \tilde \psi')\|_{L^2},$$ so that using (\ref{dudmet2}) as after (\ref{dudmodel}) gives, for some $\eta'>0$,
\begin{align*}
& \varepsilon^{-1} E\sup_{f \in D_\varepsilon, \psi \in \mathcal C(b)} |\langle DG_{f_0}[(f_0-f)\tau], \mathbb W \rangle | \\
&\lesssim \varepsilon^{-1} 2^{Jd/2} \sup_{f \in D_\varepsilon, \psi \in \mathcal C(b)} \|DG_{f_0}[(f_0-f)\tau]\|_{L^2} \sqrt{\log (1/\varepsilon)} \\
&\lesssim \varepsilon^{-1} 2^{Jd/2} \sup_{f \in D_\varepsilon, \psi \in \mathcal C(b)} (\delta_\varepsilon\|\varphi-\varphi_0\|_\infty + \varepsilon \|\Pi_{W_L}\tilde \psi\|_{\infty}) \|f-f_0\|_{L^2} \sqrt{\log (1/\varepsilon)} \\
& \lesssim \log^{\eta'}(1/\varepsilon)( \varepsilon^{(-2s-4-2d +2d +2s-d - \xi +2s)/(2s+4+d)} +  \varepsilon^{(-d-4+d+2s)/(2s+4+d)}) = o(1)
\end{align*}
since $s>2+d/2+\xi/2$ and likewise,
\begin{align*}
&\varepsilon^{-1} E \sup_{f \in D_\varepsilon, \psi \in \mathcal C(b)} |\langle DG_{f_0}[f-f_\tau - f \tau], \mathbb W \rangle | \\
& \lesssim \varepsilon^{-1} 2^{Jd/2} \sup_{f \in D_\varepsilon, \psi \in \mathcal C(b)} \|DG_{f_0}[f-f_\tau-f\tau]\|_{L^2} \sqrt{\log (1/\varepsilon)} \\
&\le  \varepsilon^{-1} 2^{Jd/2} \sup_{f \in D_\varepsilon, \psi \in \mathcal C(b)} \|\tau\|_\infty \|\tau\|_{L^2} \sqrt{\log (1/\varepsilon)} \\
&\le  \varepsilon^{-1}\sqrt{\log (1/\varepsilon)} 2^{Jd/2}\sup_{f \in D_\varepsilon, \psi \in \mathcal C(b)}( \delta_\varepsilon \|\varphi-\varphi_0\|_\infty + \varepsilon\|\Pi_{W_L}\tilde \psi \|_\infty) ( \delta_\varepsilon \|\varphi-\varphi_0\|_{L^2} + \varepsilon\|\Pi_{W_L}\tilde \psi \|_{L^2})
\end{align*}
is $o(1)$. So overall we obtain that (\ref{lan11}) (and then (\ref{lan1})) becomes
\begin{align}\label{lanlin}
&\ell(f)-\ell (f_\tau) \notag \\
& = \frac{1}{\varepsilon}\langle DG_{f_0}[f_0 \tau], \mathbb W \rangle_{L^2}  + \frac{1}{2\varepsilon^2} \|DG_{f_0}[f_0 \tau]\|_{L^2}^2 + \frac{1}{\varepsilon^2} \langle DG_{f_0}[f-f_0], DG_{f_0}(f_0 \tau) \rangle_{L^2} + \bar Z_\varepsilon
\end{align}
where $\bar Z_\varepsilon=o_{P_{f_0}^Y}(1)$ uniformly in $f \in D_\varepsilon, \psi \in \mathcal C(b)$, for any $\gamma>0$.

\medskip

\textit{C: Completion of the expansion (critical terms):}

Now recalling (\ref{tau!}) the right hand side of (\ref{lanlin}) can further be rewritten as
\begin{align}\label{lannex}
&t\langle DG_{f_0}[f_0 \Pi_{W_L}\tilde \psi], \mathbb W \rangle_{L^2}  + \frac{t^2}{2} \|DG_{f_0}[f_0 \Pi_{W_L}\tilde \psi]\|_{L^2}^2 + \frac{t}{\varepsilon} \langle DG_{f_0}[f-f_0], DG_{f_0}(f_0 \Pi_{W_L}\tilde \psi) \rangle_{L^2} \\
&~~ + \bar Z_\varepsilon +\bar Z_\varepsilon' \notag
\end{align}
where $\bar Z_\varepsilon'= o_{P_{f_0}^Y}(1)$ uniformly in $f \in D_\varepsilon, \psi \in \mathcal C(b)$, since
\begin{align*}
&|\bar Z_\varepsilon'| \le  \frac{\delta_\varepsilon}{\varepsilon} |\langle DG_{f_0}[f_0 \Pi_{W_L}(\varphi-\varphi_0)], \mathbb W \rangle_{L^2}|  \\
&~~~~+ \frac{\delta_\varepsilon^2}{2\varepsilon^2} \|DG_{f_0}[f_0\Pi_{W_L}(\varphi-\varphi_0)]\|_{L^2}^2 + \frac{\delta_\varepsilon}{\varepsilon^2} |\langle DG_{f_0}[f-f_0], DG_{f_0}(f_0\Pi_{W_L}(\phi-\phi_0)) \rangle_{L^2}| \\
& = A+B+C.
\end{align*}
For term $A$ we need to bound the supremum of the Gaussian process $\{\langle \mathbb W, q_\varphi \rangle_{L^2}: e^\varphi \in D_\varepsilon\}$ where $q_\varphi =DG_{f_0}[f_0 \Pi_{W_L}(\varphi-\varphi_0)]$: The $L^2$-continuity of $DG_{f_0}$ from (\ref{smt}) implies $\|DG_{f_0}[f_0 \Pi_{W_L}(\varphi-\varphi')]\|_{L^2} \lesssim \|f_0\|_\infty \|\varphi-\varphi'\|_{L^2}$ and thus arguing as after (\ref{dudmodel}) we obtain $$E\sup_{\phi: e^\phi \in D_\varepsilon} |\langle DG_{f_0}[f_0\Pi_{W_L}(\varphi-\varphi_0), \mathbb W\rangle_{L^2}| \lesssim  \sup_{\phi: e^\varphi \in D_\varepsilon} \|DG_{f_0}[f_0 \Pi_{W_L}(\varphi-\varphi_0)]\|_{L^2} 2^{Jd/2} \sqrt{\log (1/\varepsilon)}$$ which combined with previous bounds and Proposition \ref{diff} gives   
\begin{align*}
A&=O_{P_{f_0}^Y} \left(\varepsilon^{-1}\delta_\varepsilon \|\varphi-\varphi_0\|_{(H_0^2)^*}2^{Jd/2} \sqrt{\log (1/\varepsilon)} \right) \\
&= O_{P_{f_0}^Y} \left(\varepsilon^{(-2s-4-d+2d +\gamma_1 +2s+4-d)/(2s+4+d)} \log ^{\eta+1/2} (1/\varepsilon) \right) =o_P(1)
\end{align*} 
uniformly in $f \in D_\varepsilon$ for every $\gamma_1>0$. Likewise we see that $C=o_P(1)$ since for every $\gamma_1>0$,
$$O_P\left(\delta_\varepsilon \varepsilon^{-2} \|f-f_0\|_{(H^2_0)^*}\|\varphi-\varphi_0\|_{(H^2_0)^*} \right) = O_P\left(\varepsilon^{(2d+\gamma_1 -4s -8 -2d + 4s+8)/(2s+4+d)} \log ^{2\eta+1/2}(1/\varepsilon)\right),$$ and $B=o_P(1)$ follows too as $B$ is stochastically smaller than $C$ thanks to the extra $\delta_\varepsilon$ factor.

\smallskip

The last step is to pass $L \to \infty$ in the three main terms in (\ref{lannex}). We recall from Step II that $\tilde \psi \in C_K(\mathcal O)$. For $l$ large enough the boundary corrected wavelets do not intersect with the support $K$ of $\tilde \psi$ and the dual norm of $(H^2)^*$ is thus estimated by the dual norm of $(H^{2}(\mathbb R))^*$, just as in (\ref{intest}); we thus have, under assumption i) of Theorem \ref{work} with $L=J$, and by (\ref{mult}), (\ref{tildpdec}) that
\begin{align} \label{intest1}
\|DG_{f_0}[f_0(\Pi_{W_J}\tilde \psi -\tilde \psi)]\|^2_{L^2} & \lesssim \|f_0\|^2_{C^2}\|\Pi_{W_J}\tilde \psi -\tilde \psi\|^2_{(H^2_0)^*} \notag \\
&=  \sum_{l>J}2^{-4l} \sum_r \langle \tilde \psi, \Phi_{l,r} \rangle_{L^2(\mathbb R^d)}^2\\
& \le \sum_{l>J} 2^{-l(d+2\gamma)}  \le 2^{-J(d+2\gamma)} = o(1). \notag
\end{align}
By (\ref{tildp}) we have that $\{DG_{f_0}[f_0 \tilde \psi]=- V_{f_0}[S_{f_0}[S_{f_0}[\psi/u_{f_0}]]]= -S_{f_0}[\psi/u_{f_0}]: \psi \in \mathcal C(b)\}$ is bounded in $C_c^{d/2+\gamma}(\mathcal O)$ and has covering numbers bounded by Proposition \ref{pain}C. Moreover by $L^2$-continuity of $DG_{f_0}$ and arguing as before (\ref{dudmodel}), the covering 
numbers of $\{DG_{f_0}[f_0\Pi_{W_J}\tilde \psi]: \psi \in \mathcal C(b)\}$ are bounded by those of a ball in a $c_02^{Jd}$-dimensional space. The class of differences of such functions then has covering numbers bounded by the product of the  covering numbers of each class, and using (\ref{dudmet1}) and also (\ref{intest1}) to bound $\sigma$ we have, for any $\gamma>0$,
\begin{align*}
E \sup_{\psi \in \mathcal C(b)}|\langle DG_{f_0}[ f_0 (\Pi_{W_J}\tilde \psi-\tilde \psi)], \mathbb W \rangle_{L^2}| &\lesssim 2^{Jd/2} \int_0^\sigma \sqrt {\log(A/\eta)}d\eta + \int_0^\sigma (A'/\gamma)^{(d/2)/(\gamma+d/2)} d\gamma \\
&\lesssim 2^{Jd/2}2^{-J(d/2+\gamma)} \sqrt{\log 1/\varepsilon} + o(\sigma) =o(1).
\end{align*}
Finally using again the previous estimate for $\|\Pi_{W_J}\tilde \psi -\tilde \psi\|^2_{(H^2_0)^*}$ we obtain for every $\gamma>0$
\begin{align*}\frac{t}{\varepsilon} |\langle DG_{f_0}[f-f_0], DG_{f_0}(f_0(\Pi_{W_J}\tilde \psi - \tilde \psi) \rangle_{L^2}| &\lesssim \varepsilon^{-1} \|f_0\|_{C^2}  \|f-f_0\|_{(H^2_0)^*} \|\Pi_{W_J}(\tilde \psi - \tilde \psi)\|_{(H^2_0)^*}\\
&=O(\varepsilon^{-1} \varepsilon^{(2s+4+d+2\gamma)/(2s+4+d)} \log^{\eta}(\varepsilon^{-1})) =o(1)
\end{align*}
finishing the proof under assumption i). Under assumption ii) the proof of the last step proceeds analogously, with the previous estimates for $\|\Pi_{W_J}\tilde \psi -\tilde \psi\|^2_{(H^2_0)^*}$ replaced by the hypothesis featuring in Theorem \ref{work}ii).

\smallskip

To conclude, since $f_0 \tilde \psi = -\tilde \Psi$ the LAN expansion (\ref{lannex}) becomes 
\begin{equation} \label{lannex2}
\ell(f)-\ell(f_\tau) = -t\langle DG_{f_0}[\tilde \Psi], \mathbb W \rangle_{L^2}  + \frac{t^2}{2}\|DG_{f_0}[\tilde \Psi]\|_{L^2}^2 - \frac{t}{\varepsilon} \langle DG_{f_0}[f-f_0], DG_{f_0}[\tilde \Psi] \rangle_{L^2} + Z''_\varepsilon,
\end{equation}
for some $Z_\varepsilon''=o_{P_{f_0}^Y}(1)$ uniformly in $f \in D_\varepsilon$ and $\psi \in \mathcal C(b)$. Now using (\ref{invop}) we have
\begin{equation}
\langle DG_{f_0}[f-f_0], DG_{f_0}[\tilde \Psi] \rangle_{L^2} = \langle f-f_0, \psi \rangle_{L^2}
\end{equation}
and we can insert (\ref{lannex2}) into (\ref{laplace}) to obtain
\begin{equation} \label{lanfinal}
E^{\Pi^{D_\varepsilon}} \left[e^{\frac{t}{\varepsilon} \langle f -f_0, \psi \rangle_{L^2}} |Y\right] =  e^{-t\langle DG_{f_0}[\tilde \Psi], \mathbb W \rangle_{L^2}  + t^2\|DG_{f_0}[\tilde \Psi]\|_{L^2}^2} \times \frac{\int_{D_\varepsilon}  e^{\ell(f_\tau)} d\Pi(f)}{\int_{D_\varepsilon} e^{\ell(f)} d\Pi(f)} \times e^{R_\varepsilon}
\end{equation}
completing the proof.
\end{proof}

\medskip

\textbf{Step IV: Change of variables.}

\smallskip

We now analyse the ratio of high-dimensional integrals appearing on the r.h.s of (\ref{lanfinal}).

\begin{proposition} \label{pertu}
For every $\gamma>0$ and $\gamma_1$ (appearing in the definition of $\delta_\varepsilon$ in (\ref{dol})) satisfying $0<\gamma_1<2\gamma$ we have for the perturbation $f_\tau$ from (\ref{taud}) and any integer $L \le J$, as $\varepsilon \to 0$
\begin{equation}\label{crudebd}
\frac{\int_{D_\varepsilon}  e^{\ell(f_\tau)} d\Pi(f)}{\int_{D_\varepsilon} e^{\ell(f)} d\Pi(f)}= (1+o(1))\times \frac{\int_{D_{\varepsilon,\tau}} e^{\ell(f)} d\Pi(f)}{\int_{D_\varepsilon} e^{\ell(f)}d\Pi(f)} = O_{P_{f_0}^Y}(1),
\end{equation}
uniformly in $|t| \le T$ for any $T$ and in $\psi \in \mathcal C(b)$ from Theorem \ref{work}, and where $$D_{\varepsilon,\tau}=\{g=f_\tau = e^{\varphi +\tau}: f = e^\varphi \in D_\varepsilon\} \subset \supp(\Pi).$$ 
\end{proposition}
\begin{proof}
Let us write $\bar \ell (\varphi) = \ell (f(\varphi))$ when viewing the function $\ell(f), f=e^\varphi,$ as a map of $\varphi$. Thus  $\ell(f_\tau) = \bar \ell (\varphi + \tau)$, and so 
\begin{equation}\label{ratio}
\frac{\int_{D_\varepsilon}  e^{\ell(f_\tau)} d\Pi(f)}{\int_{D_\varepsilon} e^{\ell(f)} d\Pi(f)} = \frac{\int_{D_\varepsilon}  e^{\bar \ell(\phi + \tau)} d\pi(\varphi)}{\int_{D_\varepsilon} e^{\bar \ell(\varphi)} d\pi(\varphi)},
\end{equation}
where $\pi$ is the law of $\varphi = \log (f), f \sim \Pi$, and where in what follows, in abuse of notation, the integration domain $D_\varepsilon$ is viewed interchangeably as a set of $f$'s or of $\phi$'s. By definition of the prior $\Pi$ and (\ref{wavprop}) the $d\pi$-integrals are product integrals over hyper-ellipsoids $$\supp(\pi)=\prod_{l \le J} I_l^{N_l} \subset \mathbb R^{\bar c_02^{Jd}},~~N_l \le  c_0 2^{ld},~\bar c_0 <\infty,$$ with each marginal distribution being the uniform distribution on the intervals $$I_l = [-B2^{-l(s+d/2)} \bar l^{-2}, B2^{-l(s+d/2)}\bar l^{-2}].$$ By definition of $\tau$ in (\ref{tau!}) we have the orthogonal decomposition
\begin{equation}\label{convcomb}
\varphi + \tau = \Pi_{V_J \setminus W_L}\varphi + (1-\delta_\varepsilon)\Pi_{W_L}\varphi + \delta_\varepsilon p(\tilde \psi),~~\text{ where }~~p(\tilde \psi) = \frac{t \varepsilon}{\delta_\varepsilon} \Pi_{W_L}(\tilde \psi) + \Pi_{W_L}\varphi_0,
\end{equation}
thus the perturbation $\tau$ affects only the subspaces $W_L$ in the support of the prior, and in these subspaces it equals a convex combination of points $\varphi$ in the support with the point $p(\tilde \psi)$ which, as the next lemma shows, is also contained in $\supp(\pi)$.
\begin{lemma}\label{intpt}
Let $L \le J$ and $\psi \in \mathcal C(b), |t| \le T$. Then for every $\gamma>0$ and $0<\gamma_1<2\gamma$  we have that $p(\tilde \psi) \in \supp(\pi)$ for all $\varepsilon$ small enough.
\end{lemma}
\begin{proof}
By Condition \ref{interior} we know that $\varphi_0$ is an interior point satisfying $$|\langle \varphi_0, \Phi^\mathcal O_{l,r}\rangle| \le (B-\epsilon)2^{-l(s+d/2)} \bar l^{-2}$$ for some $\epsilon>0$ and all $l,r$. Thus to show that $p(\tilde \psi) \in \supp (\pi)$ it suffices to verify that $$\frac{t \varepsilon}{\delta_\varepsilon}|\langle \tilde \psi, \Phi^\mathcal O_{l,r} \rangle| \le \epsilon 2^{-l(s+d/2)}\bar l^{-2}$$ for all $l \le J,r$. Using (\ref{tildpdec}) we check that for our choice of $\delta_\varepsilon, J$ and $\gamma>\gamma_1/2$,
\begin{align*}
\max_{l\le J,r}2^{l(s+d/2)}\bar l^2\frac{\varepsilon}{\delta_\varepsilon} |\langle \tilde \psi, \Phi^\mathcal O_{l,r} \rangle| &\le \max_{l\le J, r}2^{l(s+2-d/2-\gamma)}  \bar l^{2}\varepsilon \varepsilon^{-(2d+\gamma_1)/(2s+4+d)}  \\
&\lesssim \varepsilon^{(-2s-4+d +2s+4-d+2\gamma-\gamma_1)/(2s+4+d)} \log^2(1/\varepsilon)< \epsilon
\end{align*}
for $\varepsilon$ small enough.
\end{proof}
We deduce from the lemma and (\ref{convcomb}) that by convexity of $\supp \pi$ we must have $\varphi+\tau  \in \supp(\pi)$ too. If $\pi_\tau$ denotes the law of $\varphi + \tau$ then obviously on $V_J \setminus W_L$ we have $\pi_\tau=\pi$ and for the marginal coordinates of the subspace $W_L$ the densities of the law  $\pi_\tau$ with respect to the law $\pi$ equal the constant  $(1-\delta_\varepsilon)^{-1}$ on a strict subinterval $\tilde I_l \subset I_l$ (they are the densities of $(1-\delta_\varepsilon)U+\delta_\varepsilon p$ for $U$ a uniform random variable and $p$ a constant). The density of the product integrals is then also constant and on the support of $\pi_\tau$ given by
\begin{equation} \label{deltac}
\frac{d\pi_\tau}{d\pi}=\prod_{l\le L}\left(\frac{1}{1-\delta_\varepsilon}\right)^{N^l_0} = \left(\frac{1}{1-\delta_\varepsilon}\right)^{\bar c2^{Jd}} = 1+o(1), \bar c>0,
\end{equation}
where $N^l_0\le N_l$ is the dimension of $W_l \setminus W_{l-1}$ and where we have used the definition of $\delta_\varepsilon$ from (\ref{dol}) in $$2^{Jd} \delta_\varepsilon = \varepsilon^{(-2d+2d+\gamma_1)/(2s+4+d)}   \to 0$$ as $\varepsilon \to 0$ for any $\gamma_1>0$. We can thus write the l.h.s. of (\ref{crudebd}) as
\begin{equation*}
\frac{\int_{D_{\varepsilon, \tau}}  e^{\ell(f(\varphi))} \frac{d\pi_\tau(\varphi)}{d\pi(\varphi)} d\pi(\varphi)}{\int_{D_\varepsilon} e^{\ell(f(\varphi))} d\pi(\varphi)}=(1+o(1))\times \frac{\int_{D_{\varepsilon,\tau}} e^{\ell(f)} d\Pi(f)}{\int_{D_\varepsilon} e^{\ell(f)}d\Pi(f)}
\end{equation*}
where $$D_{\varepsilon,\tau}=\{g=f_\tau = e^{\varphi +\tau}: f = e^\varphi \in D_\varepsilon\} \subset \supp (\Pi),$$  completing the proof of the first identity in the proposition. Renormalising  by $\int_\mathcal F e^{\ell(f)}d\Pi(f)$ the last ratio equals $$\frac{\Pi(D_{\varepsilon,\tau}|Y)}{\Pi(D_{\varepsilon}|Y)} \le \frac{1}{\Pi(D_{\varepsilon}|Y)}=O_{P_{f_0}^Y}(1)$$ using also $\Pi(D_{\varepsilon}|Y) \to 1$ in $P_{f_0}^Y$-probability (Step I above), completing the proof.
\end{proof}

\medskip

\textbf{Step V: Convergence of finite-dimensional distributions}

\smallskip

For any $\psi \in C^4_K(\mathcal O)$ and $\mathbb W= Y-u_{f_0}$ define the random variables
\begin{equation} \label{centr}
\tilde f (\psi)   = \langle \tilde f, \psi \rangle_{L^2} =  \langle f_0, \psi \rangle_{L^2} - \varepsilon \langle DG_{f_0}[\tilde \Psi], \mathbb W \rangle_{L^2}
\end{equation}
which in view of Proposition \ref{gauss0} and the results from Section \ref{infbd} define the random variable
\begin{equation} \label{tildef}
\tilde f = \tilde f(Y) =^\mathcal L f_0 + \varepsilon \tilde X, ~~\tilde X \sim \mathcal N_{f_0},
\end{equation}
in $(C^{\alpha}_K(\mathcal O))^*$. Let further $\tilde \Pi(\cdot|Y)$ be the law of $\varepsilon^{-1}(f-\tilde f)|Y$ in $(C^{\alpha}_K(\mathcal O))^*$ conditional on $Y$.

\begin{proposition}\label{fidi}
For any finite vector $(\psi_1, \dots, \psi_k), k \in \mathbb N,$ of functions $\psi_i \in C^S_K(\mathcal O)$, $S \in \mathbb N$ as in Condition \ref{interior}, and any random variable $Z$ of law $\mu$ in $(C^{\alpha}_K(\mathcal O))^*$, let $\mu_{k}$ be the distribution of the random vector $(Z(\psi_1), \dots, Z(\psi_k))$ in $\mathbb R^k$. Then for $\beta=\beta_{\mathbb R^k}$ the bounded-Lipschitz distance for weak convergence of probability measures on $\mathbb R^k$ and every fixed $k \in \mathbb N$, we have
$$\beta (\tilde \Pi(\cdot|Y)_k, (\mathcal N_{f_0})_k) \to^{P_{f_0}^Y} 0~\text{as } \varepsilon \to 0.$$ 
\end{proposition}
\begin{proof}
In view of (\ref{tvlim}) and since the total variation distance dominates any metric for weak convergence, it suffices to prove the result for $\Pi(\cdot|Y)$ replaced by $\Pi^{D_\varepsilon}(\cdot|Y)$.  We first need

\begin{lemma} Fix arbitrary $\psi \in C^S_c(\mathcal O)$. Let $\tau$ from $(\ref{tau!})$ be the perturbation in $W_L$ associated to $\psi$ with $L$ satisfying $2^{L(s+d/2)} L^{2+\eta} \varepsilon^{2s/(2s+4+d)} = o(1)$, and let $D_{\varepsilon,\tau}$ be as in Proposition \ref{pertu}. Then we have, as $\varepsilon \to 0$ and in $P_{f_0}^Y$-probability, $$\Pi(\mathcal D_{\varepsilon, \tau} |Y) \to 1,~~\text{and }~ \frac{\int_{D_{\varepsilon,\tau}} e^{\ell(f)} d\Pi(f)}{\int_{D_\varepsilon} e^{\ell(f)}d\Pi(f)} \to 1.$$
\end{lemma}
\begin{proof} 
Recall that by our choice of $M_0$ in the paragraph above (\ref{tvlim}) we have $\Pi(\mathcal D_\varepsilon^{M_0/8W}|Y) \to 1$ in $P_{f_0}^Y$-probability. We show that $\mathcal D_{\varepsilon, \tau}=\mathcal D_{\varepsilon, M, \tau}$ contains the set $\mathcal D_{\varepsilon}^{M_0/8W}$ for all $\varepsilon$ small enough so that the first limit of the lemma will follow. By definition of $\mathcal D_{\varepsilon, M, \tau}$ and (\ref{convcomb}) we need to show that for every $\bar f \in \mathcal D_{\varepsilon}^{M_0/8W}$ the function $\bar \phi = \log \bar f$ equals $$\phi^* +\tau=\Pi_{V_J\setminus W_L}\varphi^* + (1-\delta) \Pi_{W_L}\varphi^* + \delta \Pi_{W_L}(\varphi_0) + \varepsilon t \Pi_{W_L}(\tilde \psi) $$ for some $\varphi^*$ such that $e^{\varphi^*} \in \mathcal D_{\varepsilon}^{M}$. Define
\begin{equation} \label{deconvex}
\varphi^* = \Pi_{V_J \setminus W_L} \bar \phi + \Pi_{W_L}(\varphi_0) + \frac{\Pi_{W_L} \bar \varphi - \Pi_{W_L}(\varphi_0) - \varepsilon t \Pi_{W_L}(\tilde \psi)}{1-\delta_\varepsilon}.
\end{equation}
Then by (\ref{convcomb}) we have $$\varphi^* + \tau = \Pi_{V_J \setminus W_L} \bar \phi +  \Pi_{W_L}\bar \phi =\bar \phi$$ so it remains to show that $e^{\varphi^*}$ is indeed in $\mathcal D_{\varepsilon}^{M}.$ To achieve this we need to show $\varphi^* \in \supp \pi$ and that the $\|\cdot\|^{(i)}$-norm inequalities in (\ref{dn}) are satisfied. Since $\bar \varphi \in V_J$ we have $$\|\varphi^* - \varphi_0\|^{(i)} \le \|\Pi_{W_L}(\varphi^*-\varphi_0)\|^{(i)} + \|(Id-\Pi_{W_L})(\bar \varphi - \varphi_0)\|^{(i)}$$ and the second term is less than $ [(M_0/8W)+ (M_0/8)] r(i, \varepsilon) \le M_0/4$ since $e^{\bar \varphi} \in  \mathcal D_{\varepsilon}^{M_0/8W}$ and using (\ref{iota}). For the first we use (\ref{iota}) once more and have for $\delta_\varepsilon$ and $\varepsilon$ small enough 
\begin{align*}
\|\Pi_{W_L}(\varphi^*-\varphi_0)\|^{(i)} &\le W(1-\delta_\varepsilon)^{-1} \left[\|\bar \varphi -\varphi_0\|^{(i)} + t \varepsilon \|\Pi_{W_L}(\tilde \psi)\|^{(i)}\right] \\
&\le (M_0/4) r(i, \varepsilon) + o(r(i,\varepsilon)) \le (M_0/2) r(i,\varepsilon),
\end{align*}
using $e^{\bar \phi} \in \mathcal D_{\varepsilon}^{M_0/8W}$ and $\varepsilon \|\tilde \psi\|^{(i)}=O(\varepsilon)$ for any fixed $\psi \in C^S_c(\mathcal O), S>4,$ by definition of $\tilde \psi$. Finally we also have $\varphi^* \in \supp \Pi$: this is clear for $l > L$ as then $\phi^*=\bar \phi$, and also for $l \le L$ since then, using Condition \ref{interior} and $\tilde \psi \in L^2$ in view of $S>4$,
\begin{align*}
&2^{l(s+d/2)} \bar l^2 |\langle \varphi^*, \Phi^\mathcal O_{l,r} \rangle| \\
&\le 2^{l(s+d/2)}\bar l^2 |\langle \varphi_0, \Phi^\mathcal O_{l,r} \rangle_{L^2}| + (1-\delta_\varepsilon)^{-1} 2^{l(s+d/2)} \bar l^2\left[|\langle \bar \varphi - \varphi_0, \Phi^\mathcal O_{l,r} \rangle_{L^2}| + \varepsilon t |\langle \tilde \psi, \Phi^\mathcal O_{l,r}\rangle_{L^2}| \right] \\
& \le (B-\epsilon) + (4/3) 2^{L(s+d/2)}L^2( \|\bar \varphi - \varphi_0\|_{L^2} + O(\varepsilon)) \le B
\end{align*}
for $\varepsilon$ small enough, using that $\bar \phi \in \mathcal D_{\varepsilon}^{M_0/8W}$ implies $\|\bar \varphi - \varphi_0\|_{L^2} =  O(\varepsilon^{2s/(2s+4+d)} \log^\eta(1/\varepsilon))$ and the hypothesis on $L$. This proves the first claim of the lemma. For the second we can renormalise both numerator and denominator by $\int_\mathcal F e^{\ell(f)} d\Pi(f)$ to obtain
\begin{equation} \label{renormo}
\frac{\int_{D_{\varepsilon,\tau}} e^{\ell(f)} d\Pi(f)}{\int_{D_\varepsilon} e^{\ell(f)}d\Pi(f)} = \frac{\Pi(D_{\varepsilon, \tau}|Y)}{\Pi(D_\varepsilon|Y)} \to^{P_{f_0}^Y} 1.
\end{equation}
\end{proof}

We now are ready to combine this lemma for the choice $$2^L \simeq \varepsilon^{-(2s/(s+d/2))/(2s+4+d)} /\log^{\gamma_0} (1/\varepsilon), L \le J,$$ $\gamma_0$ a large enough constant, with Theorem \ref{work}ii) for arbitrary but fixed $\psi \in C^S_K(\mathcal O)$: Since $\psi \in C^S_K(\mathcal O)$ we have $\tilde \psi \in C^{S-4}_K(\mathcal O)$, and since for large $l$ the support $K$ of $\tilde \psi$ does not overlap with the support of the boundary corrected wavelets, we have as in (\ref{intest}) $$\|\Pi_{W_L}(\tilde \psi) - \tilde \psi\|_{(H^2_0)^*} \lesssim 2^{-L(S-2)} \lesssim \varepsilon^{[(S-2)(2s/(s+d/2)]/(2s+4+d)} \log^{\gamma_0'} (1/\varepsilon) =  O(\varepsilon^{\bar \gamma/(2s+4+d)})$$ for some $\bar \gamma>d$ by assumption on $S,s$. Thus Theorem \ref{work}ii and the previous lemma imply, by definition of $\tilde f$ and for   $r_\varepsilon = o_{P_{f_0}^Y}(1)$ that
$$E^{D_\varepsilon}[ e^{t \varepsilon^{-1} \langle f - \tilde f, \psi \rangle}|Y] = e^{r_\varepsilon} e^{\frac{t^2}{2}\|DG_{f_0}[\tilde \Psi]\|_{L^2}^2}, ~\forall t \in \mathbb R,$$ for all $\psi \in C^S_K(\mathcal O)$. Applying this to any linear combination $\psi = \sum_{i \le k} a_i \psi_i, a_i \in \mathbb R,$ which still defines an element of $C^S_K(\mathcal O)$, the convergence of the Laplace transform and the Cramer-Wold device (Section \ref{weakprob}) gives the desired weak convergence in probability in the $\beta_{\mathbb R^k}$-metric. 
\end{proof}

\textbf{Step VI: Tightness estimates and convergence in function space}

\smallskip

We now prove Theorem \ref{main} with $\tilde f$ from (\ref{tildef}) in place of $\bar f$. That $\tilde f$ can be replaced by $\bar f$ will constitute the last step VII. By (\ref{tvlim}) and since the total variation distance dominates any metric for weak convergence, it suffices to prove Theorem \ref{main} for $\Pi^{D_\varepsilon}(\cdot|Y)$ replacing $\Pi(\cdot|Y)$. Let $f \sim \Pi^{D_\varepsilon}(\cdot|Y)$ conditionally on $Y$ and consider the stochastic process  $$(X_1(\psi) = \varepsilon^{-1}(\langle f,\psi \rangle_{L^2} - \tilde f(\psi)): \psi \in C^\alpha_K(\mathcal O))$$ whose law on $(C^\alpha_K(\mathcal O))^*$ we denote by $\tilde \Pi^{D_\varepsilon}(\cdot|Y).$ Further let $X_2 \sim \mathcal N_{f_0}$. For $\lambda \in \mathbb N$ large enough to be chosen below and finite-dimensional spaces $W_{\lambda}$ with projection operator $\Pi_{W_{\lambda}}$ as in Step II, define probability measures $\tilde \Pi^{D_\varepsilon}_{\lambda}(\cdot|Y), \mathcal N_{f_0,\lambda}$ as the laws of the stochastic processes $$P_{\lambda}(X_i)\equiv(X_i(\Pi_{W_{\lambda}}\psi): \psi \in C^\alpha_K(\mathcal O)), i=1,2,$$ respectively, which, as projections, are defined on the same probability space as the $X_i$'s. Then using the triangle inequality for the metric $\beta=\beta_{(C^\alpha_K(\mathcal O))^*}$,
\begin{align}  \label{blconv}
\beta(\tilde \Pi^{D_\varepsilon}(\cdot|Y), \mathcal N_{f_0}) & \le \beta (\tilde \Pi^{D_\varepsilon}_{\lambda}(\cdot|Y), \mathcal N_{f_0,\lambda}) + \beta(\tilde \Pi^{D_\varepsilon}(\cdot|Y), \tilde \Pi^{D_\varepsilon}_{\lambda}(\cdot|Y)) + \beta(\mathcal N_{f_0}, \mathcal N_{f_0, \lambda}) \notag \\
&\notag = \beta (\tilde \Pi^{D_\varepsilon}_{\lambda}(\cdot|Y), \mathcal N_{f_0,\lambda}) + \sum_{i=1}^2\sup_{\|F\|_{Lip}\le 1}|E[F(X_i) -F(P_{\lambda} (X_i))]|  \\
&\le \beta_{W_{{\lambda}}} (\tilde \Pi_\lambda^{D_\varepsilon}(\cdot|Y), (\mathcal N_{f_0})_{\lambda}) +  \sum_i  E\|X_i - P_{\lambda}(X_i)\|_{(C^\alpha_K(\mathcal O))^*} = A+B+C,
\end{align}
and we wish to show that the last three terms converge to zero in $P_{f_0}^Y$-probability. 

For $B$, still writing $E=E^{\Pi^{D_\varepsilon}}[\cdot|Y]$, we have from Parseval's identity, (\ref{holdwav}), for `interior' wavelets $\Phi_{l,r}^\mathcal O=\Phi_{l,r}, l \ge \lambda$ large enough, and for some $\gamma>0$ such that $\alpha'=2+d/2+\gamma< \alpha-d$,
\begin{align} \label{normo}
&E\|X_1-P_{\lambda}(X_1)\|_{(C^\alpha_K(\mathcal O))^*}  = E \sup_{\|\psi\|_{C^\alpha_K(\mathcal O)} \le 1} \varepsilon^{-1} \left | \langle f- \tilde f, \psi - \Pi_{W_{\lambda}}\psi \rangle_{L^2}\right| \\
& \lesssim E \sum_{{\lambda}<l,r} 2^{-l(\alpha+d/2)} \varepsilon^{-1} \left | \langle f- \tilde f, \Phi_{l,r}^\mathcal O \rangle_{L^2} \right|  =  \sum_{{\lambda}<l,r} 2^{l(\alpha'-\alpha)}  E\varepsilon^{-1} \left | \langle f- \tilde f, 2^{-l(\alpha'+d/2)}\Phi_{l,r}^\mathcal O \rangle_{L^2} \right|. \notag
\end{align}
The functions $2^{-l(\alpha'+d/2)}\Phi^\mathcal O_{l,r},$ are all contained in the set $\mathcal C(b)$ from Theorem \ref{work}  with $\gamma=\alpha'-2-d/2$ (cf.~after (\ref{holdwav})), and that Theorem combined with (\ref{crudebd}) and the results from Section \ref{infbd} now imply that 
\begin{equation}
E^{\Pi^{D_\varepsilon}} \left[e^{\frac{t}{\varepsilon} \langle f -\tilde f, 2^{-l(\alpha'+d/2)}\Phi_{l,r}^\mathcal O \rangle_{L^2}}  |Y\right] \le r_\varepsilon e^{\frac{t^2}{2}\|S_{f_0}[2^{-l(\alpha'+d/2)}\Phi_{l,r}^{\mathcal O}/u_{f_0}]\|_{L^2}^2} \lesssim r_\varepsilon e^{ct^2},~ |t| \le T,
\end{equation}
for some $r_\varepsilon = O_{P_{f_0}^Y}(1), c>0$. Then using the inequality $E|Z| \le Ee^{Z}+Ee^{-Z}$ for any random variable $Z$ we have the bound
\begin{align*}
&\sum_{\lambda<l \le J,r} 2^{l(\alpha'-\alpha)}  E^{\Pi^{D_\varepsilon}} \left[ \varepsilon^{-1} \big | \langle f- \tilde f, 2^{-l(\alpha'+d/2)}\Phi_{l,r}^\mathcal O \rangle_{L^2} \big| |Y\right] \lesssim r_\varepsilon \sum_{\lambda<l \le J}2^{l(\alpha'+d-\alpha)}.
\end{align*}
which is $o(1)$ as $\lambda \to \infty$ since $\alpha>\alpha'+d$. 

For term $C$ in (\ref{blconv}), using that $\langle X_2, \Phi^\mathcal O_{l,r} \rangle \sim N(0,\|S_{f_0}[\Phi^\mathcal O_{l,r}/u_{f_0}]\|^2_{L^2})$ with $\|S_{f_0}[\Phi^\mathcal O_{l,r}/u_{f_0}]\|_{L^2} \lesssim 2^{2l}$, we have
\begin{align} \label{gpapp}
&E\|X_2-P_{\lambda}(X_2)\|_{(C^\alpha_K(\mathcal O))^*}  \le E \sup_{\|\psi\|_{C^\alpha_K(\mathcal O)} \le 1}  \left | \langle X_2, \psi - \Pi_{W_{\lambda}}\psi \rangle_{L^2}\right| \notag \\
& \lesssim  \sum_{\lambda<l,r} 2^{-l(\alpha+d/2)}  E\left | \langle X_2, \Phi_{l,r}^\mathcal O \rangle_{L^2} \right| \lesssim \sum_{\lambda<l} 2^{-l(\alpha-d/2-2)} =_{\lambda \to \infty} o(1).
\end{align}
To conclude the proof, let $\epsilon'>0$ be given. By the preceding bounds we can choose $\lambda=\lambda(\epsilon')$ large enough so that the terms $B,C$ in (\ref{blconv}) are each less than $\epsilon'/3$ (with $P_{f_0}^Y$-probability as close to one as desired in the case $B$). Next, for every fixed $\lambda=\lambda(\epsilon)$, by Proposition \ref{fidi} with $k=dim(W_{\lambda})$, the term $A$  can be made less than $\epsilon'/3$ with $P_{f_0}^Y$-probability as close to one as desired, for $\varepsilon$ small enough. Overall the quantity in (\ref{blconv}) is thus less than $\epsilon'>0$ arbitrary, with probability as close to one as desired, proving $\beta(\tilde \Pi^{D_\varepsilon}(\cdot|Y), \mathcal N_{f_0}) \to^{P_{f_0}^Y} 0$.

\medskip

\textbf{STEP VII: Convergence of moments and posterior mean}

\smallskip

Recall $\bar f = E^\Pi[f|Y]$. From the previous step we know that the law of $\varepsilon^{-1}(f - \tilde f)$ converges weakly in probability to $\mathcal N_{f_0}$ in $(C^\alpha_{K}(\mathcal O))^*$ as $\varepsilon \to 0$. By tightness of the Gaussian law $\mathcal N_{f_0}$ in $(C^\alpha_{K}(\mathcal O))^*$ and Theorem 2.1.20 and Exercise 2.1.2 in \cite{GN16}, we have $E\|\tilde f\|^4_{(C^\alpha_{K}(\mathcal O))^*} <\infty$ and also that for some constants $C>2\|f_0\|_{(C^\alpha_{K}(\mathcal O))^*}, c'>0$, $$P_{f_0}^Y(\|\tilde f\|_{(C^\alpha_{K}(\mathcal O))^*} > C) \le \Pr (\|\tilde X\|_{(C^\alpha_{K}(\mathcal O))^*} >C/2\varepsilon) \le 2e^{-c'/\varepsilon^2}.$$ Moreover, since $f|Y$ is bounded in $(C^\alpha_{K}(\mathcal O))^*$ we see from (\ref{tvlim}), Remark \ref{momrem} and the Cauchy-Schwarz inequality  
\begin{align*}
& \varepsilon^{-2} E^\Pi[\|f -\tilde f\|_{(C^\alpha_{K}(\mathcal O))^*}^2|Y] \\
&\lesssim E^{\Pi_{D_\varepsilon}}  \left[\varepsilon^{-2} \|f -\tilde f\|^2_{(C^\alpha_{K}(\mathcal O))^*}|Y\right] + C^2\varepsilon^{-2} \Pi(D_{\varepsilon}^c|Y) + \varepsilon^{-2} \|\tilde f\|^2_{(C^\alpha_{K}(\mathcal O))^*} 1\{\|\tilde f\|_{(C^\alpha_{K}(\mathcal O))^*} > C\}  \notag \\
& = O_{P_{f_0}^Y}(1) +  O_{P_{f_0}^Y}\big(\varepsilon^{-2} e^{-c(\eta_\varepsilon/\varepsilon)^2} + \varepsilon^{-2}e^{-c'/2\varepsilon^2} \big)  = O_{P_{f_0}^Y}(1), \notag
\end{align*} 
where the first term can be bounded by similar arguments as after (\ref{normo}) above, using also $EZ^2 \le 2 (Ee^{Z}+Ee^{-Z})$ for any random variable $Z$. Conclude that on an event of probability as close to one as desired, $\varepsilon^{-1}(f - \tilde f)|Y$ has uniformly bounded second (norm-) moments. Using the Skorohod imbedding (Theorem 11.7.2 in \cite{D02}) and standard uniform integrability arguments we can argue by contradiction and deduce that weak convergence implies convergence of the first moment (see also Section \ref{weakprob}) and thus, on the above event,
\begin{equation}
\varepsilon^{-1}(\bar f - \tilde f)=E^\Pi\left[\varepsilon^{-1}(f - \tilde f)|Y\right] \to E^{\mathcal N_{f_0}}(X)=0~~\text{in } (C^\alpha_{K}(\mathcal O))^*,
\end{equation}
 so that indeed $\|\bar f - \tilde f\|_{(C^\alpha_{K}(\mathcal O))^*} = o_{P_{f_0}^Y}(\varepsilon)$ and we can hence replace $\tilde f$ by $\bar f$ in Theorem \ref{main}.

\subsection{Proofs for Section \ref{uq}}

We only prove Corollary \ref{npconf}, the proof of Corollary \ref{spconf} is the same (in fact simpler) and omitted. The proof of coverage of $C_\varepsilon$ in Corollary \ref{npconf} follows the proof of Theorem 7.3.23 in \cite{GN16} (see also Theorem 1 in \cite{CN13}), and we only remark on the necessary modifications: We first notice that space $(C^\alpha_{K}(\mathcal O))^*$ is separable (since finite-dimensional wavelet approximations are norm dense, estimating the dual norm as in (\ref{normo})), which implies that balls in this space form uniformity classes for weak convergence towards $\mathcal N_{f_0}$. Moreover the mapping $\Phi(t)= \mathcal N_{f_0}(x:\|x\|_{(C^\alpha_{K}(\mathcal O))^*}\le t)$ is strictly increasing in $t$ since any shell $\{x: t < \|x\|_{(C^{\alpha}_K(\mathcal O))^*} < t+\delta\}, \delta>0$, of $(C^\alpha_{K}(\mathcal O))^*$ contains an element of the reproducing kernel Hilbert space of $\mathcal N_{f_0}$ (this space equals the image of $L^2$ under the Schr\"odinger operator $S_{f_0}$; using arguments from Section \ref{sop} it is easily seen that this space contains elements of $(C^{\alpha}_K(\mathcal O))^*$ of any given norm $\|h\|_{(C^{\alpha}_K(\mathcal O))^*}, \alpha>2+d/2$). Thus the proof of Theorem 7.3.23 in \cite{GN16} applies directly to give coverage of $C_\varepsilon$ and also that $\varepsilon^{-1}R_\varepsilon$ converges in $P_{f_0}^Y$-probability to $\Phi^{-1}(1-\beta)$.

\section{Appendix}

\subsection{The metric entropy inequality for sub-Gaussian processes}

`Dudley's metric entropy inequality' for suprema of sub-Gaussian random processes has been used repeatedly in the proofs: A centred stochastic process $(X(s): s \in S)$ is said to be sub-Gaussian for some metric $d$ on its index set $S$ if $E e^{\lambda (X(s)-X(t))} \le e^{\lambda^2 d^2(s,t)/2}~\forall s,t \in S, \lambda \in \mathbb R$. Denote by $N(S, \gamma, d)$ the  $\gamma$-covering numbers of the metric space $(S,d)$. Then we have for some $s_0 \in S$, $D$ any upper bound for the diameter of the metric space $(S,d)$ and any $\delta>0$, the inequalities
\begin{equation} \label{dudmet1}
E\sup_{s \in S}|X(s)| \le E|X(s_0)| + 4\sqrt 2 \int_0^{D/2} \sqrt {2\log N(S, \gamma, d)} d \gamma,
\end{equation}
\begin{equation} \label{dudmet2}
E\sup_{s,t\in S, d(s,t)\le \delta}|X(s)-X(t)| \le (16 \sqrt 2 +2) \int_0^{\delta} \sqrt {2\log N(S, \gamma, d)} d \gamma.
\end{equation}
We also note that provided the last integrals converge, the process $(X(s):s \in S)$ has a version for which the last suprema are measurable. See Theorem 2.3.7 in \cite{GN16} for proofs.

\subsection{Some properties of Schr\"odinger operators} \label{sop}

The PDE (\ref{PDE}) has been extensively studied, and we review here a few facts that were used in our proofs. For $f \in C(\bar{\mathcal O})$, the Schr\"odinger operator 
\begin{equation}
S_f[u] = \frac{\Delta}{2} u - f u
\end{equation}
is defined classically for all $u \in C^2(\mathcal O)$, and weakly for all locally integrable functions $u$ by the action $\int_\mathcal O S_f [u] \phi = \int_\mathcal O u S_f [\phi]$ on all $\phi \in C^\infty_c(\mathcal O)$.

Under suitable conditions the Schr\"odinger operator will be seen to have an inverse `$f$-Green'-operator $V_f$ which describes the unique solutions $V_f[h] $ of the \textit{inhomogeneous} Schr\"odinger equation 
\begin{equation} \label{inhom}
S_f[u]= \frac{\Delta u}{2} - fu=h~~\text{ on } \mathcal O~~s.t.~ u =0~~\text{on } \partial \mathcal O.
\end{equation}
The operator $V_f$ has a probabilistic representation by the Feynman-Kac formula
\begin{equation} \label{fkac2}
V_f[h](x) = E^x\left[\int_0^{\tau_\mathcal O} h(X_t) e^{-\int_0^t f(X_s)ds} dt\right],~~x \in \mathcal O,
\end{equation}
where $X_s$ is a standard $d$-dimensional Brownian motion started at $x$ with exit time $\tau_\mathcal O$ from $\mathcal O$ (see \cite{CZ95}, p.88). We also recall (e.g., Theorem 1.17 in \cite{CZ95}) $$\sup_{x \in \mathcal O} E^x \tau_\mathcal O \le K(vol(\mathcal O), d)<\infty.$$
Some key properties of $S_f, V_f$ are summarised in the following result.

\begin{proposition}\label{inverse}
Let $\mathcal O$ be a bounded $C^\infty$-domain in $\mathbb R^d$. Suppose $f \in C(\bar {\mathcal O})$ satisfies $f \ge f_{\min}>0$ on $\bar{\mathcal O}$. Then for any $h \in C(\bar{\mathcal O})$, $V_f[h] \in C_0(\mathcal O)$ satisfies
\begin{equation}
S_f[V_f[h]] =h~\text{on } \mathcal O
\end{equation}
and is the unique solution of (\ref{inhom}). If $h \in C^2(\bar {\mathcal O}) \cap C_0(\mathcal O)$ then we also have 
\begin{equation}
V_f[S_f[h]] =h~\text{on } \mathcal O.
\end{equation}
The operator $V_f[h] = \int_\mathcal O v(\cdot,y)h(y)dy$ admits a symmetric kernel $v(x,y)=v(y,x), x,y \in \mathcal O,$ and extends to a self-adjoint operator on $L^2(\mathcal O)$.
\end{proposition}
\begin{proof}
These results follow from Theorems 3.18 and 3.22 in \cite{CZ95}, where we notice that the space $\mathbb F(D,q)$ is defined just before Proposition 3.16 in that reference, which for $f=-q$ bounded away from zero contains all bounded functions. We notice further that $V_f$ is defined in eq.(39) in \cite{CZ95}, and for $f \ge 0$ we obviously have from (\ref{fkac2}) that $\sup_x V_f[1](x) \le  \sup_x E^x\tau_d <\infty$, so that the conditions of Theorems 3.18 and 3.22 in \cite{CZ95} are verified. Symmetry of the kernel $v$ follows from Corollary 3.18 in \cite{CZ95}. Lemma \ref{pest} below implies that operator $V_f$ is continuous on $L^2(\mathcal O)$ and hence extends to a self-adjoint operator on that space.
\end{proof}

We next turn to the mapping properties of $V_f, S_f$. For $\mathcal O$ a bounded $C^\infty$-domain, $\partial \mathcal O$ is a compact smooth manifold and the spaces $H^\beta(\partial \mathcal O), \mathcal C^\beta(\partial \mathcal O)$ can be defined as usual \cite{LM72, T83}. The Schr\"odinger operator $S_f$ is \textit{properly elliptic} in the sense of \cite{LM72}, p.110f., and if $tr[u] =u_{|\partial \mathcal O}$ is the usual boundary trace map, then the operator $(S_f, tr)$ can be shown to realise a topological isomorphism of $H^{\beta+2}(\mathcal O)$ onto $H^{\beta}(\mathcal O) \times H^{\beta+3/2}(\partial \mathcal O)$ for every $\beta \ge 0$, see Theorem II.5.4 in \cite{LM72} or Theorem 4.3.3 in \cite{T83}. Likewise, $(S_f, tr)$ realises an isomorphism of the H\"older-Zygmund space $\mathcal C^{\beta+2}(\mathcal O)$ onto $\mathcal C^{\beta}(\mathcal O) \times \mathcal C^{\beta+2}(\partial \mathcal O)$ for all  $\beta \ge 0$, see Theorem 4.3.4 in \cite{T83}. These isomorphisms are proved under the further assumption that the only smooth solution to the problem $S_f[u]=0$ on $\mathcal O$ s.t.~$u=0$ on $\partial \mathcal O$ equals $u=0$ identically, which is true in view of (\ref{fkac2}) and Proposition \ref{inverse}. From the above isomorphisms we deduce 
\begin{equation}\label{hiso}
\|u\|_{H^{\beta+2}(\mathcal O)} \lesssim \|S_f [u]\|_{H^{\beta}(\mathcal O)} + \|tr [u]\|_{H^{\beta+2}(\partial \mathcal O)}~~\forall u \in H^{\beta+2}(\mathcal O), \beta \ge 0,
\end{equation}
using also $H^{\beta+2}(\partial \mathcal O) \subset H^{\beta+3/2}(\partial \mathcal O)$, and
\begin{equation}\label{ciso}
\|u\|_{\mathcal C^{\beta+2}(\mathcal O)} \lesssim \|S_f [u]\|_{\mathcal C^{\beta}(\mathcal O)} + \|tr [u]\|_{\mathcal C^{\beta+2}(\partial \mathcal O)} ~~\forall u \in \mathcal C^{\beta+2}(\mathcal O), \beta \ge 0.
\end{equation}
In the above references, the coefficient $f$ is assumed to be a smooth function, but basic arguments show that (\ref{hiso}), (\ref{ciso}) remain valid whenever $f\in \mathcal C^\beta(\mathcal O)$. The constants in the preceding inequalities then depend only on $\beta, d, \mathcal O$ and on a bound for  $\|f\|_{\mathcal C^\beta(\mathcal O)}$.

\begin{lemma} \label{pest}
For any $f$ as in Proposition \ref{inverse} and some constant $c$, $$\|V_f[h]\|_{L^p} \le c\|h\|_{L^p}~~\forall h \in C(\bar{\mathcal O}),~~p\in \{2, \infty\}.$$ 
\end{lemma}
\begin{proof}
Since $V_f[h] \in C_0(\mathcal O)$ by Proposition \ref{inverse}, its boundary trace vanishes and thus we immediately deduce the case $p=2$ from that proposition and (\ref{hiso}) with $\beta=0$ and since $H^2(\mathcal O)$ embeds continuously into $L^2(\mathcal O)$. The case $p=\infty$ follows immediately from (\ref{fkac2}).
\end{proof}

For the next more precise lemma we recall the dual norm (\ref{h0norm}).

\begin{lemma}\label{ape}
Let $f\in C^s(\bar {\mathcal O})$ satisfy $f>0$, $\|f\|_{C^s} \le D$ for some $s>0$, and let $h \in C(\bar {\mathcal O})$ be given. Suppose that $\omega$ is a solution of the inhomogeneous equation $$\frac{\Delta}{2} \omega -  f \omega = h ~\text{ on } \mathcal O, ~s.t.~\omega=0 ~\text{on } \partial \mathcal O.$$ 
We have for some constant $C(D, d, \mathcal O)$ that
\begin{equation} \label{ape-0}
\|\omega\|_{L^2} = \|V_f[h]\|_{L^2} \le C\|h\|_{(H_0^{2})^*}.
\end{equation} 
If moreover $\beta$ is a non-negative integer and $s\ge \beta$, then we also have for some constant $C'=C'(D,\beta, d, \mathcal O)$ that
\begin{equation}\label{ape0}
\|\omega\|_{H^{\beta+2}} = \|V_f[h]\|_{H^{\beta+2}}  \le C' \|h\|_{H^\beta}
\end{equation} 
and for all $h \in C^\beta(\mathcal O)$ and $s>\beta$,
\begin{equation}\label{ape01}
\|\omega\|_{\mathcal C^{\beta+2}} = \|V_f[h]\|_{\mathcal C^{\beta+2}} \le C' \|h\|_{\mathcal C^\beta}.
\end{equation} 
\end{lemma}
\begin{proof}
We first prove (\ref{ape0}) for $h \in H^\beta(\mathcal O)$. By Theorems 8.9, 8.13 in \cite{GT98} and Proposition \ref{inverse} the solution $\omega$ is unique, lies in $H^{\beta+2}(\mathcal O) \cap C_0(\mathcal O)$ and  can be represented as $\omega = V_{f}[h]$. Then (\ref{hiso}) gives 
\begin{equation*}
\|\omega\|_{H^{\beta+2}(\mathcal O)}  \lesssim \|S_f[V_f[h]]\|_{H^{\beta}(\mathcal O)} + \|tr[V_f[h]]\|_{H^{\beta+2}(\partial \mathcal O)} = \|h\|_{H^\beta(\mathcal O)}.
\end{equation*}
Replacing (\ref{hiso}) by (\ref{ciso}) and using Theorems 6.14, 6.19 in \cite{GT98} and again Proposition \ref{inverse} to ensure that $\omega \in C^{\beta+2}(\mathcal O) \cap C_0(\mathcal O)$ for $h \in C^s(\mathcal O), f \in  C^s(\bar{\mathcal O})$, the same estimate follows for $\mathcal C^{\beta}$-norms replacing $H^\beta$-norms, giving (\ref{ape01}) for such $h$. Since $C^s(\mathcal O), s>\beta,$ is dense in $C^\beta(\mathcal O) \subset \mathcal C^\beta(\mathcal O)$, the overall result follows from a basic approximation argument.

We finally prove (\ref{ape-0}): By Proposition \ref{inverse} the operator $V_f$ is self-adjoint and takes values in $C_0(\mathcal O)$. We thus see  from (\ref{ape0}) with $\beta=0$ that
\begin{align} \label{dua}
\|w\|_{L^2} &= \|V_f [h]\|_{L^2} = \sup_{\phi \in C_c(\mathcal O), \|\phi\|_{L^2}\le 1} \left| \int_\mathcal O \phi V_f[h] \right| \notag \\
&= \sup_{\phi \in C_c(\mathcal O), \|\phi\|_{L^2}\le 1} \left|\int_\mathcal O V_f [\phi] h \right| \le  \sup_{\phi \in C_c(\mathcal O),\|\phi\|_{L^2}\le 1} \|V_f[\phi]\|_{H^2} \|h\|_{(H_0^2)^*} \le C\|h\|_{(H_0^2)^*},
\end{align}
completing the proof.
 \end{proof}
 
 We turn to existence and properties of solutions to the homogeneous Schr\"odinger equation (\ref{PDE}), and some basic stability estimates for the `solution maps' of (\ref{PDE}), (\ref{inhom}).
 
\begin{proposition}\label{sbaest} 
Let $f>0$ satisfy $\|f\|_{C^s(\mathcal O)} \le D$ for some $s>0, D>0$, and assume $g \in C^{s+2}(\bar{\mathcal O})$.

A) A unique solution $u_f \in C^{2}(\bar{\mathcal O})$ of (\ref{PDE}) exists and has Feynman-Kac representation (\ref{fkac}). 

Moreover, for every non-negative integer $0 \le \beta \le  s$ we have
\begin{equation} \label{reg0}
\|u_f\|_{H^{\beta+2}(\mathcal O)} \leq D' \|g\|_{H^{\beta+2}(\partial \mathcal O)} 
\end{equation}
and if  $0 \le \beta < s$ then we also have
\begin{equation} \label{reg1}
\|u_f\|_{\mathcal C^{\beta+2}(\mathcal O)} \leq D'  \|g\|_{\mathcal C^{\beta+2}(\partial \mathcal O)}
\end{equation}
where $D'$ depends only on $D,d, \mathcal O$. 

\medskip

B)If $f,h \in C^s(\mathcal O), s>0,$ and $u_f,u_h \in C(\bar {\mathcal O})$ are solutions to (\ref{PDE}) with coefficients $f,h$, respectively, then
we have $$\|u_f-u_h\|_{L^2} \le c  \|f-h\|_{(H^2_0)^*} \le c \|f-h\|_{L^2} $$ where $c>0$ depends only on upper bounds for $\|f\|_{C^s(\mathcal O)}, \|h\|_{C^s(\mathcal O)}, \|g\|_{C^{s+2}(\bar{\mathcal O})}$ and on $d, \mathcal O$.

\medskip

C) If $f,h \in C(\bar {\mathcal O})$ both satisfy $f,h \ge f_{\min}>0$, then for all $q \in C(\bar {\mathcal O})$, $$\|V_f(q)-V_h(q)\|_{L^2} \lesssim \|f-h\|_{L^2} \|q\|_\infty.$$

\end{proposition}
\begin{proof}
A) The existence result follows from Theorem 6.14 in \cite{GT98}. The Feynman-Kac representation (\ref{fkac}) is derived, e.g., in Theorem 4.7 in \cite{CZ95}. To prove (\ref{reg0}) we notice that Theorem 8.13 in \cite{GT98} and the hypotheses imply $u_f \in H^{\beta+2}$ so that (\ref{hiso}) gives $$\|u_f\|_{H^{\beta+2}(\mathcal O)} \lesssim \|S_f [u_f]\|_{H^{\beta}(\mathcal O)} + \|tr[u_f]\|_{H^{\beta+2}(\partial \mathcal O)} \lesssim \|g\|_{H^{\beta+2}(\partial \mathcal O)}.$$ We can then prove (\ref{reg1}) completely analogously, using Theorem 6.19 in \cite{GT98} to establish $u_f \in C^{\beta+2}(\mathcal O)$ and then (\ref{ciso}).

\smallskip

B) We notice that $\omega=u_f-u_h$ solves the inhomogeneous equation $(\Delta/2) \omega -f \omega = (f-h)u_h$ on $\mathcal O$ with $\omega=g-g=0$ on $\partial \mathcal O$. By Proposition \ref{inverse}, Lemma \ref{ape} and (\ref{mult}) we thus have 
\begin{align*}
\|u_f-u_h\|_{L^2} &= \|V_f[(f-h)u_h]\|_{L^2} \le C\|u_h(f-h)\|_{(H^2_0)^*} \\
& \le C \sup_{\|g\|_{H^2} \le 1}\|g u_h\|_{H^2} \|f-h\|_{(H^2_0)^*} \leq c \|u_h\|_{\mathcal C^2} \|f-h\|_{(H^2_0)^*}
\end{align*}
and the result follows since $\|u_h\|_{\mathcal C^2} \le const$ by part A) with $\beta =0$.

\smallskip

C) Notice that $v=V_f[q] - V_h[q]$ solves $\frac{\Delta v}{2} - fv = (f-h)V_h[q]$ on $\mathcal O$ subject to zero boundary conditions. Thus $v= V_f[(f-h)V_h[q]]$ and so by Lemma \ref{pest} we must have $\|v\|_{L^2} \lesssim \|f-h\|_{L^2} \|V_h[q]\|_{L^\infty} \lesssim \|f-h\|_{L^2} \|q\|_\infty.$
\end{proof}

\subsection{Some properties of H\"older-type spaces}\label{cgamma}

\begin{proposition}\label{hzimb}
We have $C_c^{\alpha}(\mathcal O) \subset \mathcal C_c^{\alpha}(\mathcal O) \subset \mathcal C^{\alpha, W}(\mathcal O)$ for all $0 <\alpha <S$, and $\|f\|_{\mathcal C^{\alpha, W}(\mathcal O)} \le c_2' \|f\|_{\mathcal C^{\alpha}(\mathcal O)}  \le c_2 \|f\|_{C^{\alpha}(\mathcal O)}$, with uniform constants $c_2, c_2'$.
\end{proposition}
\begin{proof}
We first prove the result for $f \in C_c^{\alpha}(\mathcal O)$, which has a zero extension from $\mathcal O$ to $\mathbb R^d$ that defines an element of $C^{\alpha}(\mathbb R^d)$ and the global H\"older norm is equal to the intrinsic one. Thus by (\ref{holdwav}) we have for all the interior wavelets $\Phi_{l,r}^\mathcal O=\Phi_{l,r}$ supported within $\mathcal O$ that $$2^{l(\alpha+d/2)}|\langle f, \Phi^\mathcal O_{l,r}\rangle_{L^2(\mathcal O)}| = 2^{l(\alpha+d/2)} |\langle f, \Phi_{l,r}\rangle_{L^2(\mathbb R^d)}| \le c \|f\|_{C^{\alpha}(\mathcal O)},$$ and for the boundary wavelets, by the support of $f$ in $\mathcal O$ and using (\ref{wavprop}) gives
\begin{align*}
\sup_{l,r} 2^{l(\alpha+d/2)}|\langle f, \Phi^\mathcal O_{l,r} \rangle_{L^2(\mathcal O)}| &\le c  \sup_{l,r} \sum_{|m-m'| \le K} |d^l_{m,m'}| 2^{l(\alpha+d/2)} |\langle f, \Phi_{l,m'}\rangle_{L^2(\mathbb R^d)}| \le cD \|f\|_{C^{\alpha}(\mathcal O)}.
\end{align*}
The preceding proof also applies to $f \in \mathcal C_c^\alpha(\mathcal O)$ since the wavelet norm (\ref{holdwav}) is an equivalent norm on $\mathcal C^\alpha(\mathbb R^d)$ (Chapter 4, \cite{GN16}), and since $C_c^{\alpha}(\mathcal O) \subset \mathcal C_c^{\alpha}(\mathcal O)$, see after (\ref{holdlognorm}).
\end{proof}

Some further properties are the content of the next proposition. For these we recall the usual definition of Besov spaces $B^\alpha_{pq}(\mathbb R^d), \alpha \in \mathbb R,$ by tensor Littlewood-Paley decomposition and of $B^{\alpha, W}_{pq}(\mathbb R^d), \alpha \in \mathbb R,$ by a tensor wavelet basis, each equivalent to each other and to the classical definition in terms of moduli of continuity when $\alpha>0$. See \cite{T83, M92, T08} and also Section 4.3.1 in \cite{GN16} (adapted to the multi-dimensional setting) for details. 

\begin{proposition}\label{pain}
A) Let $\alpha>0$ but $\alpha-2\ell <0$ for some $\ell \in \mathbb N$. Then for any $h \in C^{\alpha}_c(\mathcal O) \subset C_c^\alpha(\mathbb R^d)$ and $\Delta^\ell = \Delta \dots^{\ell \text{ times}} \Delta$ the iterated Laplace operator, we have
\begin{equation} \label{regul}
 \|\Delta^\ell h\|_{B^{\alpha-2\ell}_{\infty \infty}(\mathbb R^d)}  \le   C  \|h\|_{C^{\alpha}(\mathcal O)}.
\end{equation}
B) Let $\alpha <0, \rho > |\alpha|$ and $g \in C^{\rho}(\mathcal O)$. Then, using the definition of the $\mathcal C^{\alpha,W}(\mathcal O)$-norms in (\ref{hzbc}) also for negative values of $\alpha$, we have for all $h \in C_c(\mathcal O)$ and some constant $c>0$
\begin{equation} \label{multc}
\|hg\|_{\mathcal C^{\alpha,W}(\mathcal O)} \le c \|h\|_{B^{\alpha}_{\infty \infty}(\mathbb R^d)} \|g\|_{C^\rho(\mathcal O)}.
\end{equation}
C) For every $\alpha>0$ there exist finite constants $A>0$ depending on $d,\alpha$ such that the $\|\cdot\|_\infty$-metric entropy of the unit ball of $C_c^{\alpha}(\mathcal O)$ can be bounded as
\begin{equation}\label{vdv}
\log N(\{\psi: \|\psi\|_{C_c^{\alpha}(\mathcal O)} \le 1\}, \eta, \|\cdot\|_\infty) \le (A/\eta)^{d/\alpha} ,~~~ 0<\eta < A/2.
\end{equation}
\end{proposition}
\begin{proof}
For A), we have the inequalities $$ \|\Delta^\ell h\|_{B^{\alpha-2\ell}_{\infty \infty}(\mathbb R^d)} \le c\|h\|_{B^{\alpha}_{\infty \infty}(\mathbb R^d)} \le c' \|h\|_{C^{\alpha}(\mathcal O)},$$ where the first inequality follows from Theorem 2.3.8 in \cite{T83} (or as in the proof of Proposition 4.3.19 in \cite{GN16}), and where we have used the continuous imbedding $C_c^{\alpha}(\mathcal O) \subset C^{\alpha}(\mathbb R^d) \subset B^{\alpha}_{\infty \infty}(\mathbb R^d)$ (see \cite{T83}) in the second inequality. 

\smallskip

We now prove B): since $h$ has compact support in $\mathcal O$ we can multiply $g$ by a function in $C^\infty_c(\mathcal O)$ so that $hg = h \bar g$ on $\mathcal O$ and $\bar g \in C^\rho_c(\mathcal O)$ satisfies $\|\bar g\|_{C^\rho(\mathcal O)} \le c \|g\|_{C^\rho(\mathcal O)}$. Then Theorem 2.8.2 in \cite{T83} implies $$\|h \bar g\|_{B^{\alpha}_{\infty \infty}(\mathbb R^d)} \le c' \|h\|_{B^\alpha_{\infty \infty}(\mathbb R^d)} \|\bar g\|_{B^\rho_{\infty \infty}(\mathbb R^d)} \le c'' \|h\|_{B^\alpha_{\infty \infty}(\mathbb R^d)} \|g\|_{C^\rho(\mathcal O)}.$$ The $\|h \bar g\|_{B^{\alpha}_{\infty \infty}}$-norm bounds the $\|h \bar g\|_{B^{\alpha,W}_{\infty \infty}}$-norm up to a constant multiple which implies the desired decay of the wavelet coefficients of $h \bar g$ -- the bound on the $\|hg\|_{\mathcal C^{\alpha,W}}$-norm now follows by just repeating the estimates from the proof of Proposition \ref{hzimb}.

\smallskip

Finally the metric entropy bound in Part C) is proved in Theorem 2.7.1 in \cite{vdVW96} (the domain $\mathcal O$ there has to be bounded and convex but we can always extend elements of $C^\alpha_c(\mathcal O)$ by zero to a larger bounded convex domain without increasing the norm).
\end{proof}

\subsection{Likelihood functions and a contraction theorem for general inverse problems}\label{obs}

Let $\mathcal F$ be a Polish space with Borel-$\sigma$-field $\mathcal B_\mathcal F$ and let $\mathbb H$ be a Hilbert  space that is separable for the norm induced by the inner product $\langle \cdot, \cdot \rangle_{\mathbb H}$, with Borel-$\sigma$-field $\mathcal B_\mathbb H$. Suppose $G: \mathcal F \to  \mathbb H$ is a Borel-measurable mapping and, for $\eps>0$ a scalar `noise level', consider the formal equation in $\mathbb H$ given by
\begin{equation} \label{modelab}
Y =  G(f) + \eps \mathbb W.
\end{equation}
Here $\mathbb W$ is a centred Gaussian white noise process $(\mathbb W(h): h \in \mathbb H)$ with covariance $ E \mathbb W(h) \mathbb W(g) = \langle h,g \rangle_{\mathbb H}$, defined on some probability space $(\Omega, \mathcal A, \mu)$ (we can take $\Omega=\mathbb R^\mathbb N$ with its cylindrical $\sigma$-algebra and $\mu$ the law of $\otimes_{k=1}^\infty N(0,1)$, see Example 2.1.11 in \cite{GN16}). Observing (\ref{modelab}) then means that we observe a realisation of the Gaussian process $(Y(h) = \langle Y, h \rangle_{\mathbb H} : h \in \mathbb H)$ with marginal distributions $Y(h) \sim N(\langle G(f), h \rangle_{\mathbb H}, \|h\|_{\mathbb H}^2)$. We sometimes write $\langle \mathbb W, h \rangle_{\mathbb H}$ for the random variable $\mathbb W(h)$.

If $(e_k: k \in \mathbb Z)$ form an orthonormal basis of $\mathbb H$, and if $w=(w_k: k \in \mathbb Z) \in \ell_2, w_k \ge 0, w_k \downarrow 0$ as $|k| \to \infty$, then we can define the new Hilbert space 
\begin{equation}
\mathbb H_w = \left\{f = \sum_{k} e_k f_k, \sum_k f_k^2 w_k^2 = \|f\|^2_{\mathbb H_w}<\infty \right\},
\end{equation}
as the completion of $\mathbb H$ with respect to the $\|\cdot\|_{\mathbb H_w}$-norm, where $f_k = \langle f, e_k \rangle_{\mathbb H}$. By definition $$E\|\mathbb W\|_{\mathbb H_w}^2 = \sum_k E\mathbb W(e_k)^2 w_k^2<\infty$$ so using Ulam's theorem (Theorem 7.1.4 in \cite{D02}) and separability of $\mathbb H_w$, the cylindrically defined law of $\mathbb W$ extends to a tight Gaussian probability measure on the Borel-$\sigma$-field $\mathcal B_{\mathbb H_w}$ of $\mathbb H_w$. The equation $Y = G(f) + \eps \mathbb W$ then makes rigorous sense in $\mathbb H_w$, with $P_{G(f)}^Y$ denoting the shifted Gaussian law of the random variable $Y: (\Omega, \mathcal A) \to (\mathbb H_w, \mathcal B_{\mathbb H_w})$. If we let the law $P_0^Y$ of $\eps \mathbb W$ serve as a common dominating measure then for $ G(f)$ contained in the RKHS $\mathbb H$ of $P_0^Y$, the Cameron-Martin theorem (e.g., as in Proposition 6.1.5 in \cite{GN16}) allows us to define the log-likelihood function
\begin{equation} \label{likeli}
\ell(f)=\log p_f(Y) \equiv \log \frac{dP_{G(f)}^Y}{dP_0^Y}(Y) = \frac{1}{\varepsilon^2} \langle Y, G(f) \rangle_{\mathbb H} - \frac{1}{2\varepsilon^2} \|G(f)\|_{\mathbb H}^2.
\end{equation}
The mapping $(\omega,f) \mapsto p_f(Y(\omega))$ is jointly measurable from $(\Omega \times \mathcal F, \mathcal A \otimes \mathcal B_\mathcal F)$ to $(\mathbb R, \mathcal B_\mathbb R)$. [Indeed, using the $\mathcal B_{\mathcal F}$-$\mathcal B_{\mathbb H}$-measurability of $G$ it suffices to show joint measurability of the map $L:(\omega, h) \mapsto \langle \mathbb W(\omega), h \rangle_\mathbb H $ defined on $(\Omega \times  \mathbb H, \mathcal A \otimes \mathcal B_\mathbb H$). By what precedes $(\langle \mathbb W(\omega), h \rangle_\mathbb H: h \in \mathbb H)$ is a centred Gaussian process so joint measurability follows from separability of $\mathbb H$ and Proposition 2.1.12 (see also Definition 2.1.2) in \cite{GN16}.]

If now $\Pi$ is a prior probability distribution on $(\mathcal F, \mathcal B_\mathcal F)$, then we can apply Bayes' theorem in the product space $$(\Omega \times \mathcal F, \mathcal A \otimes \mathcal B_{\mathcal F}, Q),~~dQ(y,f) = p_f(Y(\omega)) \mu(\omega) \Pi(f)$$ to deduce (as on p.7 in \cite{GvdV17}) that the posterior distribution of $f|Y= f|Y(\omega)$ equals 
\begin{equation}\label{postab}
\Pi(B|Y) = \frac{\int_B p_f(Y) d\Pi(f)}{\int_\mathcal F p_f(Y) d\Pi(f)} = \frac{\int_B p_f(Y(\omega)) d\Pi(f)}{\int_\mathcal F p_f(Y(\omega)) d\Pi(f)} ,~~~ B \in \mathcal B_\mathcal F.
\end{equation} 

A general contraction theorem can now be proved as in the standard direct setting \cite{GGV00, GvdV17, GN16}, noting that the induced `information distance' is $\|G(f)-G(f_0)\|_{\mathbb H}$ in the measurement model (\ref{modelab}).

\begin{theorem}\label{ggv}
Let $\Pi=\Pi_\varepsilon$ be a sequence of prior distributions on $(\mathcal F, \mathcal B_\mathcal F)$, and let $d(\cdot, \cdot)$ be a (measurable) distance function on $\mathcal F$. Let $\Pi(\cdot|Y)$ be the posterior distribution from (\ref{postab}) and suppose $Y=G(f_0)+\varepsilon \mathbb W$ has law $P^Y_{f_0} \equiv P_{G(f_0)}^Y$ for some fixed $f_0 \in \mathcal F$. For a sequence of numbers $\bar \eta_\varepsilon \to 0$ such that $\bar \eta_\varepsilon /\varepsilon \to \infty$ and $C',L$ fixed constants, suppose $\Pi$ satisfies 
\begin{equation} \label{smallball}
\Pi(f \in \mathcal F: \|G(f)-G(f_0)\|_{\mathbb H} < \bar \eta_\varepsilon) \ge e^{-C'(\bar \eta_\varepsilon/\varepsilon)^2}
\end{equation}
and that $$\Pi(\mathcal F \setminus \mathcal F_\varepsilon) \le L e^{-(C'+4)(\bar \eta_\varepsilon/\varepsilon)^2}$$ for a sequence of measurable sets $\mathcal F_\varepsilon \subset \mathcal F$ and numbers $\eta^*_\varepsilon$ for which we can find tests (indicator functions) $\Psi(Y)$ such that $$E^Y_{f_0}\Psi(Y) + \sup_{f \in \mathcal F_\varepsilon, d(f,f_0) \ge \eta^*_\varepsilon} E^Y_f (1-\Psi(Y)) \le Le^{-(C'+4)(\bar \eta_\varepsilon/\varepsilon)^2}.$$ Then we have for some $c>0$ and as $\varepsilon \to 0$
\begin{equation} \label{strate}
\Pi(d(f,f_0) \ge \eta^*_\varepsilon|Y) =O_{P_{f_0}^Y} \big(e^{-c(\bar \eta_\varepsilon/\varepsilon)^2 }\big) =o_{P_{f_0}^Y}(1).
\end{equation}
\end{theorem}
\begin{proof}
The result is proved just as Theorem 7.3.5 in \cite{GN16}, adapting the proof of Lemma 7.3.4 to the present situation by replacing the $L^2$-norm there by our $\mathbb H$-norm, $1/\sqrt n$ there by our $\varepsilon$, and $\varepsilon_n$ there by our $\bar \eta_\varepsilon$, respectively. Inspection of the proof shows that also the stronger conclusion (\ref{strate}) is satisfied (cf.~also Exercise 8.7 on p.230 in \cite{GvdV17}).
\end{proof}

\subsection{Information lower bounds in function space} \label{crlbs}

We now recall some standard facts from efficient estimation in infinite-dimensional parameter spaces, see Chapter 25 in \cite{vdV98} for an introduction to the general theory. Assume that for all $h$ in some linear subspace $H$ of an inner product space with Hilbert-norm $\|\cdot\|_{LAN}$, the log-likelihood-ratio process of a statistical model of laws $\{\mathbb P_{f+h}^{\varepsilon}: h \in H\}$ on some sequence of measurable spaces $(\mathcal X_\varepsilon)$ has locally asymptotically normal (LAN) expansion
$$\log \frac{d \mathbb P^\varepsilon_{f +\varepsilon h}}{d \mathbb P^\varepsilon_f} = D_\varepsilon(h) - \frac{1}{2} \|h\|_{LAN}^2,~~ h \in  H,$$
where, as $\varepsilon \to 0$, $D_\varepsilon(h)$ converges in distribution under $\mathbb P_f^\varepsilon$ to $D(h) \sim N(0, \|h\|_{LAN}^2)$ for every fixed $h \in H$.  Next, let $(C, \|\cdot\|_C)$ be a Banach space and consider a continuous linear map $$\kappa: (H, \|\cdot\|_{LAN}) \to (C, \|\cdot\|_C).$$ Theorem 3.11.5 in \cite{vdVW96} implies that the information lower bound for estimating the parameter $\kappa (f)$ is given by the Gaussian random variable $\mathcal G$ on $C$ with marginal distributions
\begin{equation} \label{crlb}
T(\mathcal G) \sim N(0, \|\tilde \kappa_T\|_{LAN}^2),~~T \in C^*,
\end{equation}
where $\tilde \kappa_T$ is the Riesz-representer of the continuous linear map $T\circ \kappa: (H, \|\cdot\|_{LAN}) \to \mathbb R$. Note that $\tilde \kappa_T$ necessarily lies in the completion $\bar H$ of $H$ for the $\|\cdot\|_{LAN}$-norm. In particular $$\liminf_{\varepsilon \to 0}  \inf_{\hat \kappa: \mathcal X_\varepsilon \to C}\sup_f \varepsilon^{-2} E_f \|\hat \kappa -\kappa(f)\|_{C}^2 \ge E \|\mathcal G\|_C^2$$ where the supremum extends over a $\varepsilon$-neighborhood of $f$ in $H$. The above result holds whenever $\mathcal G$ is a tight Borel random variable on $C$. If $C=\mathbb R$ then the lower bound is simply given by $\|\tilde \kappa\|_{LAN}^2$ where $\tilde \kappa$ is the Riesz-representer of the map $\kappa: (H, \|\cdot\|_{LAN}) \to \mathbb R$.

\subsection{Some facts about weak convergence of random probability measures} \label{weakprob}

The following result is given in Lemma 2 of (the supplement of) \cite{CR15}.

\begin{proposition}
Let $P_n, P, n \in \mathbb N,$ be random probability measures on $\mathbb R$. Suppose for any real $t$, the Laplace transform $\int_\mathbb R e^{tx}dP(x)$ is finite almost surely and that $\int_\mathbb R e^{tx}dP_n(x) \to \int_\mathbb R e^{tx}dP(x)$ in probability. Then for any metric $\beta$ for weak convergence of probability measures, $\beta(P_n,P) \to 0$ in probability.
\end{proposition}

As a consequence we deduce also that by the `Cramer-Wold device', if $P_n,P$ are random probability measures on $\mathbb R^K$ and $\int e^{\langle t, x\rangle}dP_n(x) \to \int e^{\langle t, x \rangle}dP(x)$ in probability for all $t \in \mathbb R^K$, then $\beta(P_n,P) \to 0$ in probability. Here $\langle \cdot, \cdot \rangle$ is the standard Euclidean inner product.

\medskip

Since convergence in probability implies convergence almost surely along a subsequence, if $\beta_S(P_n,P) \to 0$ in probability on any separable metric space $S$, then any `limiting' consequence of weak convergence such as the continuous mapping theorem, convergence of moments or uniform convergence along classes of Borel sets or functions, also holds in probability, simply by arguing by contradiction and extracting subsequential almost sure limits. See the appendix in \cite{CN13}, \cite{CR15} for more details.

\bigskip

\textbf{Acknowledgements.} This research was supported by the European Research Council under ERC grant agreement UQMSI (No. 647812). I would like to thank Gabriel P. Paternain for many helpful discussions and Sven Wang for helpful remarks about Section \ref{sop}. I am also grateful to two anonymous referees for valuable remarks and suggestions.

\bibliographystyle{plain}

\bibliography{bib}

\bigskip

\bigskip

\textsc{Statistical Laboratory}

\textsc{Department of Pure Mathematics and Mathematical Statistics}

\textsc{University of Cambridge, CB3 0WB, Cambridge, UK}

Email: r.nickl@statslab.cam.ac.uk

\end{document}